\documentclass[12pt]{article}

\usepackage{cmap}  
\usepackage[T2A]{fontenc}
\usepackage[utf8]{inputenc}
\usepackage[english]{babel}
\usepackage[usenames]{color}
\usepackage{amsmath,amssymb,amsthm}
\usepackage[colorlinks=true,linkcolor=blue,urlcolor=red,unicode=true,hyperfootnotes=false,bookmarksnumbered]{hyperref}
\usepackage{mathtools}
\usepackage{xcolor}
\usepackage{biblatex}
\usepackage{csquotes}
\usepackage{wrapfig}
\usepackage{setspace}
\usepackage{enumitem}
\usepackage{appendix}
\usepackage{caption}
\usepackage{subcaption}
\usepackage{bm}

\renewcommand{\phi}{\varphi}

\renewcommand{\le}{\leqslant}
\renewcommand{\ge}{\geqslant}
\renewcommand{\mid}{:}

\newcommand{\PP}{\mathbb{P}}
\newcommand{\EE}{\mathbb{E}}
\newcommand{\R}{\mathbb{R}}

\newcommand{\aaa}{\mathcal{A}}
\newcommand{\bb}{\mathcal{B}}
\newcommand{\ff}{\mathcal{F}}
\newcommand{\hh}{\mathcal{H}}

\newcommand{\cv}{\mathbf{c}}

\newcommand{\support}{\operatorname{supp}}
\newcommand{\probabilisticCenter}{C^*}
\newcommand{\mSet}{\mathbf{m}}

\newcommand{\randomSubset}{\mathbf{X}}

\newcommand{\indicator}{\operatorname{I}}

\newcommand{\targetFunction}{s^*(n, \ell, \mSet)}

\newcommand{\sDomain}{\mathcal{S}}
\newcommand{\sVector}{\mathbf{s}}

\newcommand{\saveSet}[1]{\overset{\bullet}{#1}}
\newcommand{\excludeSet}[1]{\overset{\circ}{#1}}
\newcommand{\notContainSet}[1]{\overset{\diamond}{#1}}

\newtheorem{theorem}{Theorem}
\newtheorem{problem}{Problem}
\newtheorem{claim}[theorem]{Claim}
\newtheorem{lemma}[theorem]{Lemma}
\newtheorem{corollary}[theorem]{Corollary}
\newtheorem{proposition}[theorem]{Proposition}

\newtheorem{definition}{Definition}

\title{Octopuses in the Boolean cube: families with pairwise small intersections, part II}
\author{
    Andrey Kupavskii\footnote{G-SCOP, CNRS, University Grenoble-Alpes, France and Moscow Institute of Physics and Technology, Russia and Sa; Email: {\tt kupavskii@ya.ru}},
    Fedor Noskov\footnote{Moscow Institute of Physics and Technology, HSE University, Russian Federation, Email: {\tt noskov.fedor.99@mail.ru}},
}

\begin{document}

\maketitle

\begin{abstract}
    The problem we consider originally arises from 2-level polytope theory. This class of polytopes generalizes a number of other polytope families. One of the important questions in this filed can be formulated as follows: is it true for a $d$-dimensional 2-level polytope that the product of the number of its vertices and the number of its $d-1$ dimensional facets is bounded by $d2^{d - 1}$? Recently, Kupavskii and Weltge~\cite{Kupavskii2020} settled this question in positive. A key element in their proof is a  more general result for families of vectors in $\R^d$ such that the scalar product between any two vectors from different families is either $0$ or $1$.

    Peter Frankl noted that, when restricted to the Boolean cube, the solution boils down to an elegant application of the Harris--Kleitman correlation inequality. Meanwhile, this problem becomes much more sophisticated when we consider several families.
    
    Let $\ff_1, \ldots, \ff_\ell$ be families of subsets of $\{1, \ldots, n\}$. We suppose that for distinct $k, k'$ and arbitrary $F_1 \in \ff_{k}, F_2 \in \ff_{k'}$ we have $|F_1 \cap F_2|\le m.$  We are interested in the maximal value of $|\ff_1|\ldots |\ff_\ell|$ and the structure of the extremal example.
    
    In the previous paper on the topic, the authors found  the asymptotics of this product for constant $\ell$ and~$m$  as $n$ tends to infinity. However, the possible structure of the families from the extremal example turned out to be very complicated. In this paper, we  obtain a strong structural result for the extremal families.
    
    
\end{abstract}

\section{Introduction}

A polytope $P \subset \R^d$ is called \textit{2-level} if for each facet $F$, all vertices of $P$ are contained in either a hyperplane $H$ defined by $F$ or in another hyperplane $H'$ parallel to $F$. Several important polytopes are 2-level, e.g. hypercubes, cross-polytopes, simplices; 2-level polytopes include a number of more sophisticated families like Hanner polytopes, Birkhoff polytopes, Hansen polytopes and others~\cite{2levelPolytopesReference}. They arise in such areas of mathematics as the semidefinite programming, communication complexity and polyhedral combinatorics.

A number of authors studied combinatorial structure of 2-level polytopes~\cite{bohn2015enumeration, fiorini2016two, aprile20182, bohn2019enumeration, Kupavskii2020}. Among other problems, they considered a beautiful conjecture about tradeoff between the number of vertices $f_0(P)$ and the number of $d-1$-dimensional facets $f_{d-1}(P)$ of a 2-level polytope $P \subset \R^d$. The authors of \cite{bohn2015enumeration} conjectured that  that $f_0(P) f_{d-1}(P) \le d 2^{d + 1}$ for all $d$ for any $2$-level polytope. Recently, Kupavskii and Weltge settled this conjecture~\cite{Kupavskii2020}. Additionally, they proved a somewhat more general result (see \cite{Kupavskii2020} for the relation between the two questions):
\begin{theorem}[\cite{Kupavskii2020}]
\label{theorem: kup-weltge}
Let $\aaa$, $\bb$ be families of vectors that both linearly span $\R^n$. Additionally, assume that $\langle a,b\rangle\in\{0, 1\}$ holds for all $a\in\aaa$, $b\in\bb$. Then we have $|\aaa|\cdot|\bb| \le (n + 1) 2^n$.
\end{theorem}
The bound is tight since one can take $\aaa = \{\textbf 0, e_1, \ldots, e_n\}$ and $\bb = \{0, 1\}^n$.

Restricted to the Boolean cube $\{0, 1\}^n$, Theorem~\ref{theorem: kup-weltge} is an elegant consequence of the Harris--Kleitman inequality, as  discussed in our first paper on the topic~\cite{KupaNos1}. (As a sidenote, let us mention that, while via an affine transformation we can reduce the general case of the problem to the case when, say, family $\aaa$ is in $\{0,1\}^n$, we cannot in general transform both families into subsets of $\{0,1\}^n$.) At the same time, the problem  becomes much more sophisticated when we consider products of several families of sets in the Boolean cube. This is the problem that we address in this  paper. In order to formulate the question precisely, we need to introduce some notation.

\subsection{Notation}

In this work, we deal with families (or collections) of finite sets.
Put  $[n] = \{1, \ldots, n\}$, and, more generally, $[a, b] = \{a, a + 1, \ldots, b\}$. Given a set $X$, let $2^{X}$ and $\binom{X}{k}$ ( $\binom{X}{\le k}$) stand for the set of all subsets of $X$ and set of all $k$-element ($\le k$-element) subsets of $X$. 
Obviously, $\binom{X}{\le k} = \bigsqcup_{t = 0}^k \binom{X}{t}$, where $\binom{X}{0} = \{\varnothing \}$. We also denote $\binom{n}{\le m} := \left| \binom{[n]}{\le m} \right| = \sum_{t = 0}^m \binom{n}{t}$.

There is a trivial bijection between $2^{[n]}$ and the Boolean cube $\{0, 1\}^{n}$. Given a set $X$, we consider its characteristic vector $\cv^X$ whose $i$-th coordinate $\cv^X_i$ equals 0 if $i \not \in X$ and $1$ otherwise.

For two families $\ff_1$ and $\ff_2$ we denote 
\begin{align*}
    \ff_1\wedge \ff_2&:=\{A_1\cap A_2: A_i\in \ff_i\},\\
    \ff_1\vee \ff_2&:=\{A_1\cup A_2: A_i\in \ff_i\}.
\end{align*}
We assume that operations $\vee, \wedge$ have higher priority than $\cup$ and $\cap$, so, for example, $\ff_1 \vee \ff_2 \cup \ff_1 \wedge \ff_2 = (\ff_1 \vee \ff_2) \cup (\ff_1 \wedge \ff_2)$.

This work is dedicated to families with the <<$\mSet$-\textit{overlapping property}>>, which is defined below.
\begin{definition}
Let $\mSet = (\mSet_S)_{S \in \binom{[\ell]}{2}}$ be a vector of non-negative integer numbers indexed by unordered pairs $\{k, k'\} \in {[\ell]\choose 2}$. For simplicity we suppress brackets in $\mSet_{\{k, k'\}}$ and assume that $\mSet_{k, k'}$, $\mSet_{k', k}$, and $\mSet_{\{k, k'\}}$ identify the same entry. Families $\ff_1, \ldots, \ff_{\ell}\subset 2^{[n]}$ {\em satisfy an $\mSet$-overlapping property} if for any distinct $k_1, k_2 \in [\ell]$ and any sets $F_1 \in \ff_{k_1}$, $F_2 \in \ff_{k_2}$
\begin{align*}
    |F_1 \cap F_2| \le \mSet_{k_1, k_2}.
\end{align*}
In other words, $\ff_{k_1}\wedge \ff_{k_2}\subset {[n]\choose \le m_{k_1,k_2}}$. If all of $\mSet_{k_1, k_2}$ are equal to $m$, the property is referred to as <<$m$-overlapping>>.
If $m = 1$ then the property is referred to as <<overlapping>>.
\end{definition}
We use the standard asymptotic notation, i.e., $f = o(g), f = \Omega(g)$ etc. for some functions $f,g$ is with respect to $n\to \infty.$ Additionally, we say $f \gg g$ is for any constant $C$ there exists $n_0$ such that for all $n \ge n_0$ we have $f(n) > C g(n)$.

We also employ the following standard  notation for restrictions of a family $\ff$. Given sets $A\subset B,$ denote
\begin{align*}
    \ff|_B&:=\{F\cap B: F\in \ff\},\\
    \ff(A)&:=\{F\setminus A: A\subset F, F\in \ff\},\\
    \ff(\bar A)&:=\{F: A\cap F = \emptyset, F\in \ff\},\\
    \ff(A,B)&:=\{F \setminus B : F \in \ff \text{ and } F\cap B = A\}.
\end{align*}
Dealing with singletons, we suppress brackets for simplicity, i.e. $\ff(x) = \ff(\{x\})$ and $\ff(\overline{x})=\ff \left (\overline{\{x\}} \right )$.

\subsection{Problem statement and results}

In this work, we study the following problem.
\begin{problem}
\label{problem: main}
Let $n, \ell$ be positive integers, $\mSet$ be a vector of $\binom{\ell}{2}$ non-negative integers and $\ff_1, \ldots, \ff_\ell \subset 2^{[n]}$ be families with the $\mSet$-overlapping property. What is the maximal value $s^*(n, \ell, \mSet)$ of the product $|\ff_1|\cdot\ldots\cdot |\ff_\ell|$? What is the structure of the extremal example?
\end{problem}

If all coordinates of $\mSet$ are equal to $m$, we denote $s^*(n, \ell, m) := s^*(n, \ell, \mSet)$.

It is easy to see that $s^*(n, \ell, 0) = 2^n$: indeed, supports of sets in distinct families are disjoint. Recently, Aprile, Cevallos, and Faenza \cite{aprile20182} showed that $s^*(n, 2, 1) = (n + 1) 2^n$. Later, Frankl (personal communication) gave a simple and elegant proof that $s^*(n, 2, t) = 2^n\sum_{i=0}^t \binom{n}{i}$ using the Harris--Kleitman correlation inequality. We provide his proof in the first part of this research~\cite{KupaNos1}. In \cite{ryser1974subsets} Ryser studied a similar question about the structure of sets with overlapping property. In particular, he showed that if for $n\notin\{9, 10\}$ $n$ sets of size at least 3 have the overlapping property, then they form either a finite projective plane or a symmetric group divisible design.

\medskip 

There is an equivalent formulation of Problem~\ref{problem: main} for $\mSet = (m,m,\ldots, m)$.
\begin{problem}
\label{problem: main_graphs}
Let $n, \ell, m$ be integers and $H$ be a complete $(m + 1)$-uniform hypergraph on $n$ vertices whose edges are coloured in $\ell$ colours. Let $k_i$, $i=1, \ldots, \ell$ be the number of cliques of colour $i$ in $H$. What is the maximum value $\tilde s(n, \ell, m)$ of $k_1\cdot\ldots\cdot k_\ell$?
\end{problem}
In particular, $\tilde s(n, 2, 1)$ is the maximal value of the product of the number of cliques and number of independent sets in a simple graph $G$.
In \cite{aprile20182} it was shown that Problem~\ref{problem: main} and Problem~\ref{problem: main_graphs} are equivalent and the following holds (for 
the proof, see~\cite[Claim 14]{KupaNos1}).

\begin{proposition}
\label{proposition: formulation equivalence}
Let $n, \ell, m$ be integers, then $s^*(n, \ell, m) = \tilde s(n, \ell, m)$.
\end{proposition}

\medskip

In the previous paper~\cite{KupaNos1}, we established the following theorem.
\begin{theorem}
\label{theorem: maximal families product}
Let $\ell, $ be positive integers and let $\mSet$ be a vector of integers as above.  Then, as $n\to \infty$, we have  the following. 
\begin{align}
\label{eq: target function asymptotics old}
    \targetFunction =
    \left (1 + O \left (\frac{1}{\sqrt{n}} \right ) \right ) \cdot
    2^n\cdot
    \prod_{S\in{[\ell]\choose 2}}
        \bigg(
            \frac 1{\mSet_S!}
            \Big(
                \frac{\mSet_S \cdot n}{
                    \sum_{S'\in {[\ell]\choose 2}}
                        \mSet_{S'}
                }
            \Big)^{\mSet_S}
        \bigg).
\end{align}
\end{theorem}

The main result of this paper is a structural result concerning the extremal example. It also allows us to tighten the asymptotics.  

For an arbitrary oriented graph $T_\ell = ([\ell], E)$ and $k\in[\ell]$ define $In_k = \{k' \mid (k', k) \in E\}$ and $Out_k = \{k' \mid (k, k') \in E\}$.

\begin{theorem}
\label{theorem: maximal families structure}
Suppose $\ell, \mSet$ are fixed while $n$ tends to infinity. Then the asymptotic of $s^*(n, \ell, \mSet)$ is the following:
\begin{align}
\label{eq: target function asymptotics}
    \targetFunction = 
    \left (1 + O \left (\frac{1}{n} \right ) \right ) \cdot 
    2^n\cdot 
    \prod_{S\in{[\ell]\choose 2}}
        \bigg(
            \frac 1{\mSet_S!}
            \Big(
                \frac{\mSet_S \cdot n}{
                    \sum_{S'\in {[\ell]\choose 2}}
                        \mSet_{S'}
                }
            \Big)^{\mSet_S}
        \bigg).
\end{align}
Next, let $\ff_k$, $k \in [\ell]$ be the families from the extremal example. Then there is an oriented graph $T_\ell = ([\ell], E)$ and there are sets $A_S$ indexed by $S \in E$ such that the following holds:
\begin{enumerate}
    \item for two distinct indices $k_1, k_2 \in [\ell]$, there is an edge between $k_1$ and $k_2$ (in one of the directions)  iff $\mSet_{\{k_1,k_2\}} > 0$; 
    \item sets $A_S$, $S \in E$, form a partition of $[n]$, and, in particular,
    \begin{align*}
        |A_S| = (1 + O(n^{-1})) \cdot \frac{\mSet_S \cdot n}{
                    \sum_{S'\in {[\ell]\choose 2}}
                        \mSet_{S'}
                };
    \end{align*}
    \item for any $k \in [\ell]$, we have
    \begin{align*}
        2^{
            \bigsqcup_{k' \in In_k} 
                A_{(k', k)}
        } \vee \bigvee_{k' \in Out_k} \binom{A_{(k, k')}}{\le \mSet_{\{k, k'\}}}
        \subset
        \ff_k
    \end{align*}
\end{enumerate}
if $n$ is large enough.
\end{theorem}

We assume that $2^{
            \bigsqcup_{k' \in In_k} 
                A_{(k', k)}
        } = \{\varnothing\}$ if $In_k = \varnothing$. In what follows, we omit the parentheses in $A_{(k, k')}$. Moreover, slightly abusing notation, we sometimes let $A_{k, k'}$ and $A_{k', k}$ identify the same object without keeping information in subscript about the direction of the edge between $k$ and $k'$ in $T_\ell$.

 \textbf{Remark.} Theorem~\ref{theorem: maximal families structure} does not precisely describe the extremal example. However, a family
\begin{align*}
    2^{
            \bigsqcup_{k' \in In_k} 
                A_{(k', k)}
        } \vee \bigvee_{k' \in Out_k} \binom{A_{(k, k')}}{\le \mSet_{\{k, k'\}}}
\end{align*}
contains $1 - O(n^{-1})$ fraction of an extremal family $\ff_k$, so it presents almost full information about $\ff_k$. As we will show in Section~\ref{section: graph coloring}, the extremal example is essentially non-trivial even when $m=1$ and $\ell = 5$.



From here, we can derive a cleaner formula for 
the asymptotic of $s^*(n, \ell, m)$.
\begin{corollary}
\label{corollary: asympotic of m-uniform target function}
Suppose $\ell$ and $m$ are fixed while $n$ tends to infinity. Then we have the following: 
\begin{align*}
    s^*(n, \ell, m) = \left (1 + O \left ( \frac{1}{n}\right ) \right ) 
\left [
     \frac{1}{m!}
     \left (
        \frac{n}{\binom{\ell}{2}}
     \right )^m
    \right ]^{\binom{\ell}{2}}
    2^n.
\end{align*}
\end{corollary}

\subsection{Extremal coloring of the complete graph}
\label{section: graph coloring}

According to Proposition~\ref{proposition: formulation equivalence}, Problem~\ref{problem: main} is equivalent to the problem of an extremal edge coloring of the complete graph on $n$ vertices with respect to some functional of monochromatic cliques. The goal of this section is to illustrate, to what extent this coloring can be non-trivial, and to discuss some interesting phenomena related to it.

In order to illustrate the complexity of the problem in the situation when it is still tractable, we choose $\ell = 5$ and obtain the unique extremal example. Let $\mathbf{T}_5$ be the tournament from Figure~\ref{fig: extremal coloring 5}, left. Using the techniques developed for Theorem~\ref{theorem: maximal families structure}, we get the following result.

\begin{theorem}
\label{theorem: extremal graph coloring}
Let $\ff_1, \ldots, \ff_{5}$ form an extremal example of overlapping families. Then the underlying tournament from Theorem~\ref{theorem: maximal families structure} must be isomorphic to $\mathbf{T}_5$. Assume that it is $\mathbf{T}_5$. For $k \in [3]$, define $i_k$ as the unique element of $In_k$ and $o_k$ as the unique element of $[3] \cap Out_k$. Then, there exists a partition $W_1, W_2, W_3$ of $A_{4, 5}$ such that
\begin{align*}
    \ff_k = 2^{A_{i_k, k}} \vee \Bigg ( 
        & \bigvee_{k' \in Out_k} \binom{A_{k, k'}}{\le 1} {\bm \cup} \bigcup_{\substack{s \in \{4, 5\}, \\ s' \in \{4, 5\} \setminus \{s\}}} \binom{A_{o_k, s}}{\le 1} \vee \binom{A_{k, s'}}{\le 1} \\
        & {\bm \cup} \binom{W_{k}}{\le 1} \vee \binom{A_{k, o_k}}{\le 1} \cup \binom{W_{o_k}}{\le 1}
    \Bigg ) \cup \binom{[n]}{1}, \quad k \in [3],
\end{align*}
and
\begin{align*}
    \ff_k = 2^{\bigsqcup_{k' \in In_k} A_{k', k}}  \vee \bigvee_{k' \in Out_k} \binom{A_{k, k'}}{\le 1} \cup \binom{[n]}{1}, \quad k \in \{4, 5\},
\end{align*}
provided $n$ is large enough.
\end{theorem}

\begin{figure}[!ht]
    \centering
    \includegraphics[width=\textwidth]{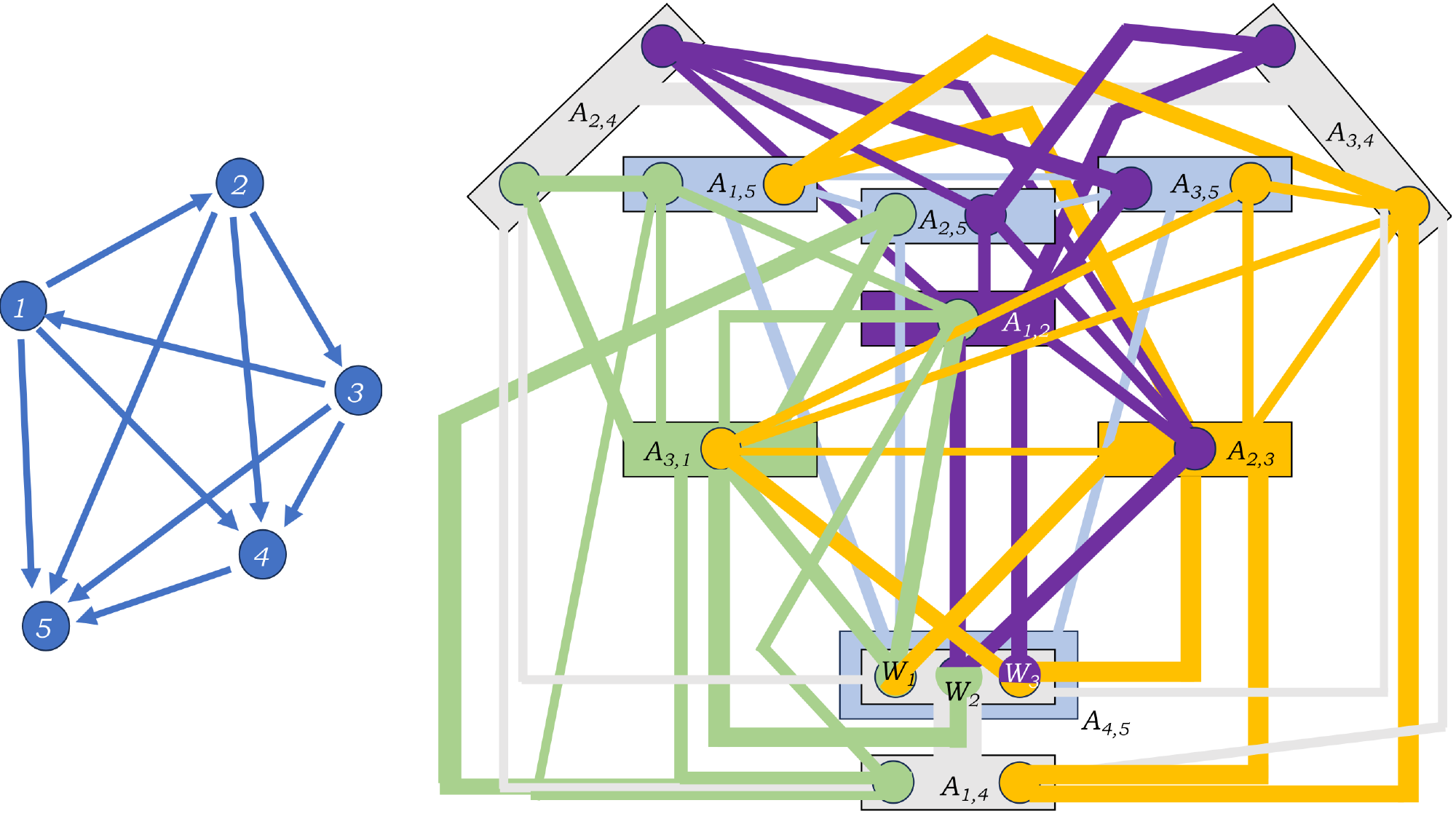}
    \caption{{\it Left}: tournament $\mathbf{T}_5$. {\it Right}: coloring of $K_n$ into 5 colors that maximizes $\tilde{s}(n, 5, 1)$. A set $A_S$ is colored in $p$ on the picture if and only if all edges $\{x, y\} \in \binom{A_S}{2}$ are colored in $p$. Given $S, S'$ that are both distinct from $\{4, 5\}$, the edge between $A_S$ and $A_{S'}$ is colored in $p$ if and only if all edges $\{x, y\}$, $x \in A_S$ and $y \in A_{S'}$ are colored in $p$. Finally, for any $k, \hat k \in [3]$ the edge between $W_{\hat k}$ and $A_{k', k}, k' \in In_k$ is colored in $p$ if and only if all edges $\{x, y\}$, $x \in A_{k', k}$ and $y \in W_{\hat k}$, are colored in $p$.}
    \label{fig: extremal coloring 5}
\end{figure}

The partition $A_S, S \in E(\mathbf{T}_5),$ of $[n]$ and the partition $W_1, W_2, W_3$ of $A_{4, 5}$ should be chosen to maximize $s^*(n, 5, 1)$. So, we have $|W_k| = (1/3 + O(n^{-1})) |A_{4, 5}|$ due to convexity.

We draw the optimal coloring on Figure~\ref{fig: extremal coloring 5}, right. The proof of Theorem~\ref{theorem: extremal graph coloring} is sketched in Section~\ref{sec5}.

The first point to note is that the tournament for the extremal example is not transitive. More generally, we prove (see Lemma~\ref{lemma: functional 1 for tournament})  that for $m = 1$ and any $\ell$ the tournament that corresponds to an extremal example should maximize the following functional:
\begin{align}
\label{eq: 1-functional in intro}
    \sum_{1 \le s_1 < s_2 \le \ell} \indicator \{In_{s_1} \cap In_{s_2} \neq \varnothing \},
\end{align}
where $\indicator\{\cdot \}$ is the indicator function. Clearly, it is bounded by $\binom{\ell}{2}$. Surprisingly, this bound is tight for $\ell \ge 24$ since a uniformly sampled random tournament achieves it with probability at least $1 - \binom{\ell}{2} (3/4)^{\ell - 2} > 0$ due to union bound. Moreover, it means that for large enough $\ell$ almost any tournament has~\eqref{eq: 1-functional in intro} equal to $\binom{\ell}{2}$. We conjecture that the same holds for $\ell \ge 7$. For $\ell = 6$, we are able to obtain bound $\eqref{eq: 1-functional in intro} \le \binom{\ell}{2} - 1$, while for the Paley tournament~\footnote{Let $p$ be a prime such that $p\equiv 3\mod 4$. Then, the Paley tournament $PT_p = (\mathbb{Z}_p, E)$ is defined as follows: $(x, y) \in E$ if and only if $x - y$ is a quadratic residue modulo $p$.} on $7$ vertices we have $\eqref{eq: 1-functional in intro} = \binom{\ell}{2}$.

The second point we want to highlight is that for $\ell = 5$ each $\ff_k$ admits decomposition:
\begin{align*}
    |\ff_k| = \mathcal{G}^0_k \sqcup \mathcal{G}^1_k \sqcup \mathcal{G}^2_k \sqcup \mathcal{R}_k,
\end{align*}
where  $\mathcal{G}^i_k$ has size $O(n^{-i}) |\ff_k|$ for $i=0,1,2$,  $|\mathcal{R}_k| = e^{-\Omega(n)} |\ff_k|$, and each $\mathcal{G}^i_k$ has its own particular structure. We conjecture the same phenomenon can be observed in general, and, additionally, there exists an algorithm that can produce such decomposition in polynomial time in $\ell$ for large enough $n$.

\subsection{Organization of the paper}

The work is organised as follows. In Section~\ref{section: tools} we list some standard tools that we use. In Section~\ref{section: sketch of the proof} we describe general proof scheme.  In Section~\ref{section: further structural results} we prove Theorem~\ref{theorem: maximal families structure}. In Section~\ref{sec5}, we give a sketch of the proof of Theorem~\ref{theorem: extremal graph coloring}. In Appendix, we give some omitted proofs.

Throughout this paper, the asympotic notation, i.e. $O(\cdot)$, $\Omega(\cdot)$, is for fixed  $\ell$ and $\mSet$ and w.r.t. $n\to \infty$.

\section{Tools}
\label{section: tools}

\subsection{Correlation inequalities}

Correlation inequalities play a crucial role in the first part of our work~\cite{KupaNos1}. However, we also use it an this paper as in ancillary tool in Lemma~\ref{lemma: direction of tentacles}.

\begin{corollary}[Theorem 4.1 from~\cite{Rinott1991}]
\label{corollary: rinott sets theorem}
For any sets collections $\aaa_1, \ldots, \aaa_m$,
\begin{align}
    \prod_{k = 1}^m |\aaa_k| \le \prod_{k = 1}^m \left | \bigvee_{S: |S| = k} \left (\bigwedge_{s \in S} \aaa_s \right ) \right |.
\end{align}
\end{corollary}

\subsection{Concentration inequalities}

One probabilistic tool that we require is a concentration inequality for martingales. In this case, we employ Azuma's inequality (see~\cite{Alon2016}):
\begin{theorem}[Azuma's iniequality]
\label{theorem: Hoeffding's iniequality}
Let $X_0, X_1 \ldots, X_n$ be a martingale sequence, such that $|X_{i} - X_{i - 1}| \le c_i$, $i \in [n]$. Then for all $t > 0$:
\begin{align*}
    \PP \left [
        X_n - X_0 \ge t
    \right ]
    \le
    \exp \left (
        - \frac{
            t^2
        }{
            \sum_{i = 1}^n c_i^2
        }
    \right ).
\end{align*}
\end{theorem}

\subsection{The Kruskal---Katona theorem}
\label{subsection: Kruskal-Katona theorem}

The last tool we list in this Section is the Kruskal---Katona theorem. First, we provide the definition of the upper shadow:
\begin{definition}
\label{definition: upper shadow}
    Given an arbitrary family $\ff \subset \binom{[n]}{k}$ for some $k$, its upper shadow $\partial_{u} \ff$ is defined as follows:
    \begin{align*}
        \partial_u \ff = \left \{F \in \binom{[n]}{k + 1} \mid \exists F' \in \ff \text{ s.t. } F' \subset F \right \}.
    \end{align*}
\end{definition}
Second, we recall that the lexicographical order on $\binom{X}{k}$ is a total order in which $A <_{Lex} B$ iff $\min(A \triangle B) \in A$. Here we assume that $X$ is finite and linearly ordered. Then, the Kruskal---Katona theorem states the following:
\begin{theorem}
\label{theorem: Kruskal-Katona theorem}
    Let $X$ be a finite linearly ordered set and $k$ a positive integer. Given an arbitrary family $\ff \subset \binom{X}{k}$ of some fixed size, the size of its upper shadow is the smallest  when $\ff$ consists of first $|\ff|$ sets in lexicographical order.
\end{theorem}
The usual form of the theorem provides a bound on the lower shadow $\partial_l \ff = \{F \in \binom{[n]}{k - 1} \mid \exists F' \in \ff \text{ s.t. } F \subset F'\}$. The formulation above can be easily derived by passing to set complements. For the proof, one can see, for example, Peter Frankl's survey~\cite{frankl_shifting_1987}. 

\subsection{Lemmas from part I~\cite{KupaNos1}}

Below, we list the results from our previous paper that we employ.

\begin{lemma}[\cite{KupaNos1}, Section 6.4]
\label{lemma: extremal exmaple asymptotics}
If there are subfamilies $\ff_k'$ of extremal families $\ff_k$ such that 
\begin{enumerate}[label=(\roman*)]
    \item $|\ff_k'| \ge (1 - \delta_n) |\ff_k|$ for each $k$ and some $\delta_n$, 
    \item each $x \in [n]$ lies in at most 2 subfamilies.
\end{enumerate}
Then
\begin{align*}
    \targetFunction = 
    \left (1 + O \left ( \delta_n \right ) \right ) \cdot
    2^n\cdot
    \prod_{S\in{[\ell]\choose 2}}
        \bigg(
            \frac 1{\mSet_S!}
            \Big(
                \frac{\mSet_S \cdot n}{
                    \sum_{S'\in {[\ell]\choose 2}}
                        \mSet_{S'}
                }
            \Big)^{\mSet_S}
        \bigg).
\end{align*}
\end{lemma}

In the paper, this lemma was not stated as it is provided above. However, its proof completely coincides with the proof of the upper bound in Theorem~2, proposed in Section 6.4.

In what follows, we build the required families $\ff_k'$ and thus obtain more precise asymptotics of $\targetFunction$. Also, we need the families $\ff_k'$ that we constructed in our previous paper.

\begin{lemma}[\cite{KupaNos1}, Lemma 21]
\label{lemma: element excluding}
Let $\ff_1, \ldots, \ff_\ell$ be the families from an extremal example. Then there are subfamilies $\ff_k' \subset \ff_k$, $k \in [\ell]$ such that
\begin{enumerate}
    \item Every element $x \in [n]$ is contained in at most two subfamilies $\ff_{k}'$'s.
    \item For every $k\in[\ell]$ it holds that $|\ff'_k| \ge (1 - \delta_n) |\ff_k|$, where $\delta_n = O(n^{-1/2})$.
\end{enumerate}
\end{lemma}

We will be extensively working with the degrees of $\ff_1,\ldots,\ff_\ell$. We use the notion of the {\it normalized degree} of an element, defined as follows:
\begin{align*}
    d(x, \ff) = \frac{
        |\ff(x)|
    }{
        |\ff|
    }.
\end{align*}
 For simplicity, we write $d_k(x)$ instead of $d(x, \ff_k)$. The degree $d_k(F)$ of a set $F$ is defined analogously.

In addition, we describe the following property of degrees:

\begin{proposition}[\cite{KupaNos1}, Proposition 20]
\label{proposition: multiplicative property of degrees}
Let $\ff_1, \ldots, \ff_\ell$ be families from the extremal example. Then for any subset of indices of $K \subset [\ell]$ and sets $F_k \in \ff_k$, $k \in K$
\begin{align*}
    \prod_{k \in K} d_k(F_k) \le C_D 
        n^{- \sum_{\{k, k'\} \in \binom{K}{2}} |F_k \cap F_{k'}| },
\end{align*}
where $C_D$ is some constant depending on $\mSet$, $\ell$.
\end{proposition}

Finally, we take the following point of view on our problem developed  in our previous work. For each family $\ff_k$ we construct the following hypergraph:
\begin{align*}
    H_k = \left (
        [n],
        \bigcup_{k' \in [\ell] \setminus \{k\}}
            \left (\ff_{k'} \right )^{(\mSet_{k, k'} + 1)}
    \right ),
\end{align*}
where $\ff^{(t)}$ stand for $\ff \cap \binom{[n]}{t}$.
We call a subset $I$ of $[n]$ \textit{independent} in the hypergraph $H_k$ if it does not contain any edge of $H_k$. Besides, we call $\ff_k$ \textit{maximal} if we cannot add any subset of $[n]$ without breaking the $\mSet$-overlapping property. If $\ff_k$ is maximal, then the number of independent sets in $H_k$ equals $|\ff_k|$.

\begin{claim}[\cite{KupaNos1}, Claim 15]
\label{claim: hypergraph point of view}
    If $\ff_k$ is maximal then it consists of independent sets in $H_k$ defined above.
\end{claim}

In what follows, we shall work with families $\ff_1,\ldots,\ff_\ell$ that maximize the product. Due to extremality, they possess some useful properties, in particular, they must be down-closed.




\section{Sketch of the proof of Theorem~\ref{theorem: maximal families structure}}
\label{section: sketch of the proof}

Throughout this section, we assume that $\ff_k$'s form an extremal example.

The proof strategy of Theorem~\ref{theorem: maximal families structure} is based on a delicate analysis of normalized degrees. In the previous paper~\cite{KupaNos1}, we showed that one can take an inclusion-maximal set $M_k$ from each family $\ff_k$ such that $M_k$ covers $[n]$ almost completely. First, we refine this result. In Section~\ref{subsection: probabilistic centers}, we prove that sets
\begin{align*}
    \probabilisticCenter_k = \left \{ x \mid d_k(x) \ge \frac{1}{2} - \alpha_n \right \}, k \in [\ell],
\end{align*}
where $\alpha_n = O(n^{-1})$ is some function of $n$, lie in $\ff_k$ and partition $[n]$. We showed that if the opposite holds, then one can modify families $\ff_k$ such that the product of their cardinalities will be larger than $\prod_{k \in [\ell]} |\ff_k|$, see the proof of Lemma~\ref{lemma: covering by centers}.

We call $\probabilisticCenter_k$ the {\it probabilistic center of $\ff_k$}. In paper~\cite{KupaNos1}, we stated that extremal families look like octopuses. We briefly discussed this metaphor, but in this work we show that it is indeed true. A probabilistic center $\probabilisticCenter_k$ is the ``body'' of a family $\ff_k$. Meanwhile, the huge part of the paper is dedicated to studying the ``tentacles''. We provide the precise definition of a tentacle a bit later in this section. 

In Section~\ref{subsection: direction of tentacles}, we establish that, between two octopuses, the tentacles are essentially directed only from, say, the first to the second. For example, a green tentacle on Figure~\ref{fig: direction of tentacles} is exponentially small: the normalized degree of $x$ in the green family is  exponentially small. In this section, we only give some intuition behind this fact for $m = 1$.

Consider $\ff_k', k\in [\ell]$, from Lemma~\ref{lemma: element excluding}. In the previous paper~\cite{KupaNos1}, we showed that
\begin{align*}
    s^*(\ell, n, 1) & = (1 + O(n^{-1/2})) \left ( \frac{n}{\binom{\ell}{2}}\right )^{\binom{\ell}{2}} 2^n \\
    & \le \frac{1}{(1 - O(n^{-1/2}))^\ell} \prod_{k \in [\ell]}|\ff_k'| \\
    & \le \frac{2^n}{(1 - O(n^{-1/2}))^\ell} \prod_{S \in \binom{[\ell]}{2}} |\wedge_{s \in S} \ff_{s}'|.
\end{align*}
Since $\wedge_{s \in S} \ff_S'$, $S \in \binom{[\ell]}{2}$, have disjoint supports, we obtain
\begin{align*}
    \left (
        \frac{n}{\binom{\ell}{2}}
    \right )^{\binom{\ell}{2}} \le (1 + O(n^{-1})) \prod_{S \in \binom{[\ell]}{2}} \binom{|\support \wedge_{s \in S} \ff_s' |}{2}.
\end{align*}
Note that $|\support(\wedge_{s \in S} \ff_s')|$ should be linear in $n$ to match the inequality above. Because each element of the ground set is contained in supports of at most $2$ primed families, $\support(\wedge_{s \in S} \ff_s') \subset \cup_{s \in S} \probabilisticCenter_s$. Without loss of generality, consider $S = \{1, 2\}$ and assume $|\support(\ff_1'\wedge \ff_2') \cap \probabilisticCenter_2| = \Theta(n)$. Let $x \in \probabilisticCenter_1$.

\begin{wrapfigure}[14]{l}{0.4\textwidth}
    \centering
    \includegraphics[width=0.35\textwidth]{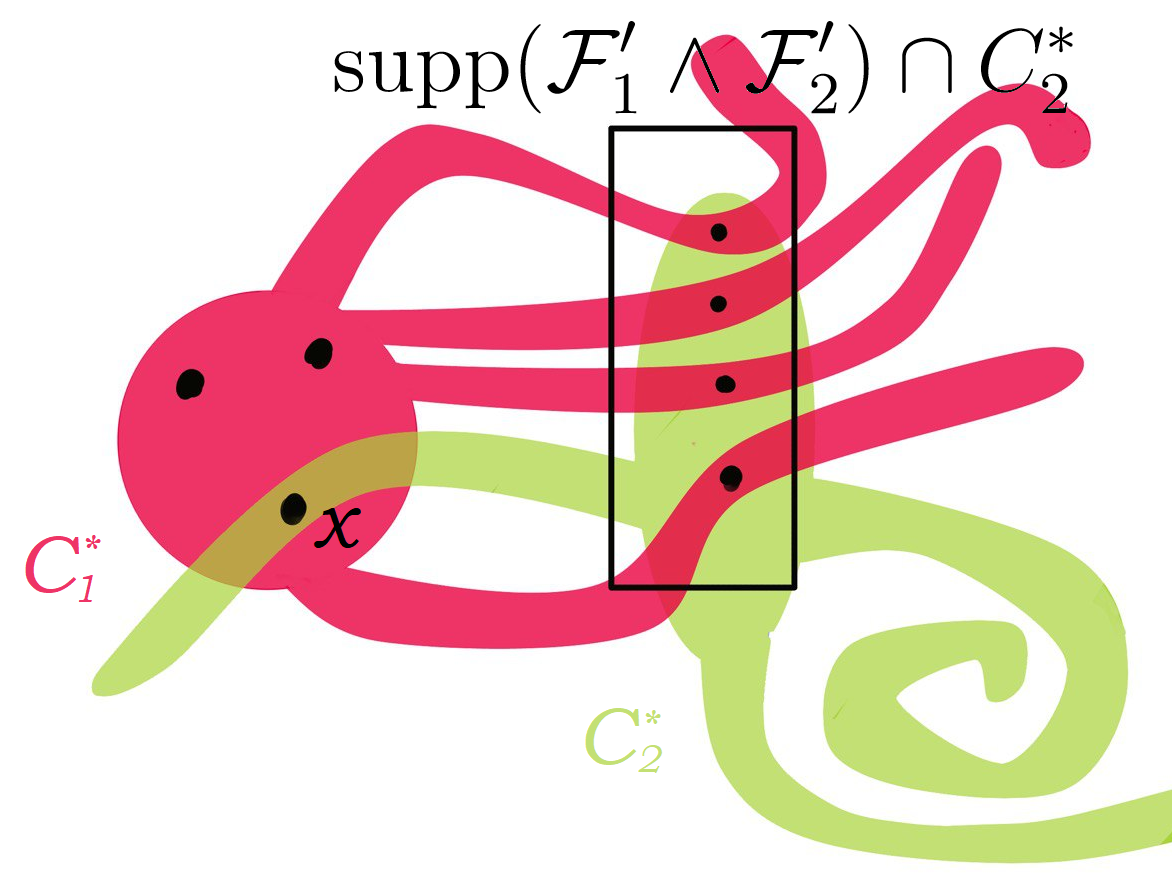}
    \caption{Illustration of the ``tournament'' effect.}
    \label{fig: direction of tentacles}
\end{wrapfigure}

We picture families $\ff_1$ and $\ff_2$ on Figure~\ref{fig: direction of tentacles}. Since $x \in {\color{red} \probabilisticCenter_1}$, for almost all sets $F$ in $\ff_1$ a set $F \cup \{x\} \in {\color{red}\ff_1}$. And almost all subsets of $\color{red} \probabilisticCenter_1$ can be continued by almost any red tentacle. Thus, if a set $F' \in {\color{green}\ff_2}$ contains $x$, then it cannot intersect red tentacles. But the number of red tentacles is $\Theta(n)$! Hence, the number of subsets in $\color{green} \probabilisticCenter_2$ that can be extended by $x$ is $2^{|{\color{green} \probabilisticCenter_2}| - \Theta(n)}$. Consequently, the degree $d_2(x)$ is exponentially small. For a precise argument, see Lemma~\ref{lemma: direction of tentacles}.\footnote{Note that, while the tentacle in the other direction is exponentially small, it does not mean that it is empty. In fact, we can add such sets to the extremal example, which greatly complicates its structure. See Section~\ref{section: graph coloring} for the example.}

Therefore, there is a tournament between ``octopuses'' which is defined by the direction of tentacles. In general, there may be a pair of families $\ff_{k_1}, \ff_{k_2}$ such that $\mSet_{k_1, k_2} = 0$. In that case, the corresponding edges are missing, and the tournament becomes an oriented graph. For a reminder, this is a part of the statement of Theorem~\ref{theorem: maximal families structure}.

Then, a {\it tentacle} of an extremal family $\ff_k$ is a subset $F \subset \bigcup_{k' \in Out_k} \probabilisticCenter_{k'}$ that intersects each $\probabilisticCenter_{k'}, k' \in Out_k$, in at most $\mSet_{k, k'}$ elements. However, we slightly abuse this definition during the proof and define a tentacle as a more specified thing. We hope that the meaning is always clear from the context.

\begin{wrapfigure}[16]{r}{0.45\textwidth}
    \centering
    \includegraphics[width=0.4\textwidth]{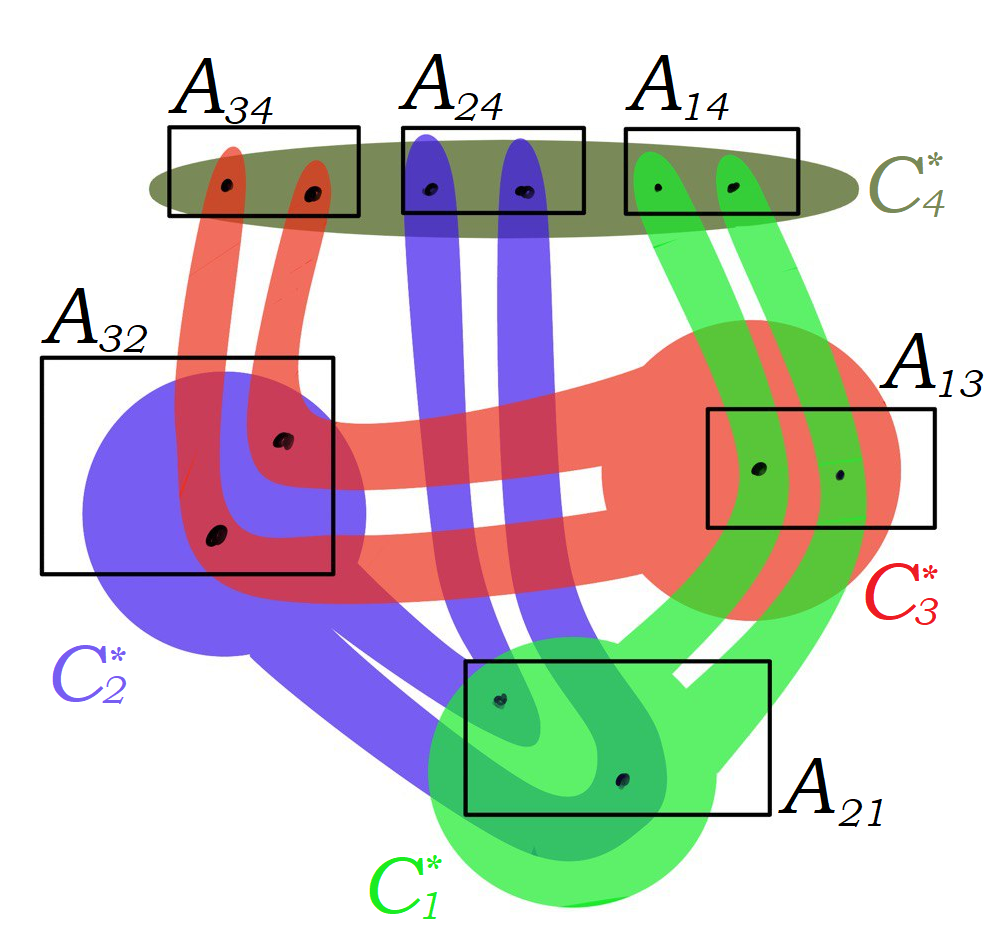}
    \caption{Octopuses described in Section~\ref{section: sketch of the proof}.}
    \label{fig: octopuses}
\end{wrapfigure}

Some decomposition of families into bodies and tentacles is presented on Figure~\ref{fig: octopuses}. Here, we choose $m = 1$ and  $\ell = 4$. Let $A_{ij}$ be from the statement of Theorem~\ref{theorem: maximal families structure} with parentheses omitted. Then, we have 
\begin{align*}
    {\color{green} \ff_1} \simeq 2^{\color{green} \probabilisticCenter_1} \vee \binom{A_{13}}{\le 1} \vee \binom{A_{14}}{\le 1},
\end{align*}
where $\binom{A_{13}}{\le 1} \vee \binom{A_{14}}{\le 1}$ are tentacles of ${\color{green} \ff_1}$. We put ``$\simeq$'' here because there exists a negligibly small fraction of additional sets in ${\color{green} \ff_1}$ that are not captured by this formula. One can also see the tournament structure on Figure~\ref{fig: octopuses}.

The last part of the argument is a symmetrization that turns out to be sophisticated technically, but in its heart is a simple idea inspired by the famous Zykov's symmetrization~\cite{zykov1949some}. Informally, it can be stated as follows: if some tentacle has small normalized degree and some other tentacle has large degree, removing the small tentacle and replacing it with the large tentacle increases the product. We describe the precise argument in Section~\ref{subsection: symmetrization lemma}. For example, the symmetrization argument also implies that each center $\probabilisticCenter_k$ can be partitioned into sets $A_{k', k}, k' \in In_k,$ such that for each $x \in A_{k', k}$ we have
\begin{align*}
    d_{k'}(x) = \Theta(n^{-1}),
\end{align*}
see Section~\ref{subsection: localization of tentacles} for the proof.

Finally, we apply the symmetrization argument as follows. We consider some specified family of tentacles and fix the tentacle $S_0$ of minimal normalized degree in this family. Then, the symmetrization argument implies that the normalized degree of $S_0$ is equal to the average degree of this specified family, up to a factor of $1 - o(1)$. We compute the average degree by some double counting, see Section~\ref{subsection: asymptotics of normalizaed degrees}.  That provides the lower bound on all normalized degrees of the specified family. The upper bound is established by a different, but simple, technique that we do not describe here.

All tentacles can be partitioned into certain specified families such as the mentioned above. And if all tentacles of a family $\ff_k$ have large enough degree, then it should contain an octopus described in the statement of Theorem~\ref{theorem: maximal families structure} as $2^{
            \bigsqcup_{k' \in In_k} 
                A_{\{k', k\}}
        } \vee \bigvee_{k' \in Out_k} \binom{A_{\{k', k\}}}{\le \mSet_{\{k, k'\}}}$.

In what follows, we will not refer to the octopus metaphor except some comments. We use it only to give some intuition. All further formulations are exact.

\section{Proof of Theorem~\ref{theorem: maximal families structure}}
\label{section: further structural results}

We start this section with the following simple claim:
\begin{claim}
\label{claim: degrees consitence}
Consider the families $\ff_1, \ldots, \ff_\ell$ from the extremal example and their arbitrary subfamilies $\ff_k'$. Let $\delta_n$ be such that $|\ff_k'| \ge (1 - \delta_n) |\ff_k|$. Then
\begin{align*}
    \frac{
        |\mathcal H_{k}|
    }{
        |\ff_k|
    }
    - 
    \delta_n
    \leqslant
    \frac{
        |\mathcal{H}_k \cap \ff'_k|
    }{
        |\ff_k'|
    }
    \leqslant
    (1 - \delta_n)^{-1}
    \frac{
        |\mathcal{H}_k|
    }{
        |\ff_k|
    }
\end{align*}
for an arbitrary subfamily $\mathcal H_k$ of $\ff_k$.
\end{claim}

\begin{proof}
Clearly,
\begin{align*}
    \frac{
        |\mathcal H_k \cap \ff_k'|
    }{
        |\ff_k'|
    }
    =
    \frac{
        |\ff_k|
    }{
        |\ff_k'|
    }
    \frac{
        |\mathcal{H}_k \cap \ff_k'|
    }{
        |\ff_k|
    }
    \le
    (1 - \delta_n)^{-1}
    \frac{
        |\mathcal H_k|
    }{
        |\ff_k|
    }.
\end{align*}
On the other hand,
\begin{align*}
    \frac{
        |\mathcal H_k \cap \ff_k'|
    }{
        |\ff_k'|
    }
    \ge
    \frac{
        |\mathcal H_k| - |\ff_k \setminus \ff_k'|
    }{
        |\ff_k|
    }
    \ge
    \frac{
        |\mathcal H_k|
    }{
        |\ff_k|
    }
    - \delta_n.
\end{align*}
\end{proof}

\subsection{Probabilistic centers}
\label{subsection: probabilistic centers}


\begin{lemma}
\label{lemma: covering by centers}
There is an absolute constant $C$ and $\alpha_n =\frac{C}{n}$ such that the following holds. Define the {\em probabilistic center} of $\ff_k$ as follows.
\begin{align}
    C_k^* = \left \{
        x \mid d_k(x) \ge \frac{1}{2} - \alpha_n
    \right \}.
\end{align}
Then the families $\ff_1, \ldots, \ff_\ell$ from the extremal exmaple satisfy the following properties:
\begin{itemize}
    \item the probabilistic center $C^*_k$ belongs to $\ff_k$ for each $k$,
    \item the probabilistic centers $C^*_k$ are disjoint and cover the ground set $[n]$,
\end{itemize}
if $n$ is large enough.
\end{lemma}

First, we prove a simple but useful claim:
\begin{claim}
\label{claim: inner set degree}
If families $\ff_{1}, \ldots, \ff_\ell$ are extremal, then for any $k \in [\ell]$ and any subset $F$ of $\probabilisticCenter_k$ 
\begin{align}
\label{eq: center joint degree}
    d_k(F) \geqslant 2^{-|F|} - 2 \alpha_n.
\end{align}
\end{claim}
\begin{proof}
We prove it by induction on the size of $F$. If $|F| = 1$ then \eqref{eq: center joint degree} holds by the definition of the probabilistic center. Otherwise, suppose \eqref{eq: center joint degree} holds for some $F'$, and we want to prove \eqref{eq: center joint degree} for $F = F' \cup \{x\}$. For any $x \in \probabilisticCenter_k$ the fraction of sets $X \in \ff_k$ such that $X \cup \{x\} \not \in \ff_k$ is at most $1 - 2 d_k(x) \leqslant 2 \alpha_n$ since $\ff_k$ is down-closed. Denote this fraction by $f(x)$. It is easy to see that
\begin{align*}
    d_k(F' \cup \{x\}) \geqslant \frac 1 2 (d_k(F') - f(x)) \geqslant 2^{-|F|} - 2 \alpha_n.
\end{align*}
\end{proof}

Now we are ready to prove the lemma.

\begin{proof}[Proof of Lemma~\ref{lemma: covering by centers}]
Due to Claim~\ref{claim: inner set degree}, $\probabilisticCenter_k$ cannot contain any set $G \in \ff_{k'}$, $k' \neq k$ of size at least than $\mSet_{k, k'} + 1$, provided $n$ is sufficiently large. Otherwise, $d_k(G) = 0$ contradicts $d_k(G) \ge 2^{- \mSet_{k, k'} - 1} - 2 \alpha_n$. Using Claim~\ref{claim: hypergraph point of view}, we infer that $\probabilisticCenter_k \in \ff_k$.

From Proposition~\ref{proposition: multiplicative property of degrees} for arbitrary $k_1$, $k_2$ we get
\begin{align}
\label{eq: lemma covering by centers, two families tradeoff}
    d_{k_1}(x)d_{k_2}(x) \le \frac{C_D}{n}.
\end{align}
It implies that probabilistic centers are disjoint.

Finally, we prove that they cover $[n]$ completely. Suppose there is $x$ such that $d_{k}(x) < \frac{1}{2} - \alpha_n$ for all $k$. Without loss of generality, assume that $d_1(x) =\max_k\{d_k(x)\}.$ We are going to show that $\prod_{k=1}^\ell |\mathcal G_k|>\prod_{k=1}^\ell |\mathcal F_k|$, for some $\mathcal G_1, \ldots, \mathcal G_\ell$ that also satisfy the $\mSet$-overlapping property. This will contradict the extremality of $\ff_1,\ldots, \ff_k$. We define $\mathcal G_k$ as follows: $$\mathcal G_1:= \ff_1(\bar x)\cup \{F\cup \{x\}: F\in \ff_1(\bar x)\}$$ and $\mathcal G_k = \ff_k(\bar x)$ for $k\ge 2.$ It is easy to check that $\mathcal G_i$ are indeed $\mSet$-overlapping. Note that $$\prod_{k=1}^\ell |\mathcal G_k| = 2\prod_{k=1}^\ell |\ff_k(\bar x)|$$ and that $$\prod_{k=1}^\ell |\mathcal F_k| = \prod_{k=1}^\ell \frac 1{1-d_k(x)}|\ff_k(\bar x)|.$$ The latter equality follows from $$1 = \frac{|\ff_k(\bar x)|+|\ff_k(x)|}{|\ff_k|} = \frac{|\ff_k(\bar x)|}{|\ff_k|}+d_k(x).$$
Thus, we are left to show that $\prod_{k=1}^\ell \frac 1{1-d_k(x)} < 2.$

If $d_1(x)<\frac 13$ then from~\eqref{eq: lemma covering by centers, two families tradeoff} we get  $d_k(x)\le \sqrt{\frac{C_D}n}$ for each $k\ge 2$. Thus, 
$$\prod_{k=1}^\ell \frac 1{1-d_k(x)}\le \frac 1{1-1/3}\cdot \Big(\frac 1{1-O(n^{-1/2})}\Big)^{\ell-1}<2.$$

If $d_1(x)>\frac 13$ then from~\eqref{eq: lemma covering by centers, two families tradeoff} we get  $d_k(x)\le \frac{3 C_D}n$ for each $k\ge 2$. Thus, 
$$\prod_{k=1}^\ell \frac 1{1-d_k(x)}\le \frac 1{1-d_1(x)}\cdot \Big(\frac 1{1-O(n^{-1})}\Big)^{\ell-1}= (1+O(n^{-1}))\frac 1{1-d_1(x)}.$$
In order for the last expression to be less than $2,$ it is enough to have $d_1(x) = \frac 12 - \alpha_n$ for some specific $\alpha_n = O(n^{-1})$, which proves the lemma.
\end{proof}

From this lemma, we see that any set $F$ from a particular family $\ff_k$ can be decomposed into two parts: one part that lies in the probabilistic center $\probabilisticCenter_k$ and another part ({\it tentacles}) that is formed by intersections with other centers. 

\subsection{Direction of tentacles}
\label{subsection: direction of tentacles}

In this section, we use the following notation for arbitrary families $\hh_{k}$, $k \in [\ell]$ such that $2^{\probabilisticCenter_{k}} \subset \hh_k$:
\begin{align*}
    \hh_{k_1 \setminus k_2} & = \{F \in \hh_{k_1} \mid F \cap \probabilisticCenter_{k_2} = \varnothing \}, \\
    \hh_{k_1 \to k_2} & = \hh_{k_1} \setminus \hh_{k_1 \setminus k_2},
\end{align*}
where $k_1, k_2$ are distinct indices from $[\ell]$. Moreover, we define
\begin{align*}
    d_{k_1 \to k_2}^{(\ge s)} = \frac{1}{|\ff_{k_1}|} |\{F \in \ff_{k_1} \mid |F \cap \probabilisticCenter_{k_2}| \ge s\}|.
\end{align*}

Another important step in the proof of Theorem~\ref{theorem: maximal families structure} is the following.

\begin{lemma}
\label{lemma: direction of tentacles}
Consider extremal families $\ff_1, \ldots, \ff_\ell$ with probabilistic centers $\probabilisticCenter_1, \ldots, \probabilisticCenter_\ell$. Suppose that there are subfamilies $\ff_k' \in \ff_k$, $k \in [\ell],$ such that
\begin{enumerate}[label=(\roman*)]
    \item $|\ff_k'| \ge (1 - \delta_n) |\ff_k|$ for some $\delta_n = O(n^{-1/2})$ and each indices $k \in [\ell]$, \label{tentacle lemma: lower family volume condition}
    \item each element $x \in [n]$ appears in sets from only two of the $\ff_k'$, \label{tentacle lemma: 2-factor structure condition}
    \item $\probabilisticCenter_k \in \ff_k'$. \label{tentacle lemma: probabilistic center presence condition}
\end{enumerate}
Then for any indices $k_1, k_2$ such that $\mSet_{k_1, k_2} > 0$ and some $k_1', k_2'$ such that $\{k'_1, k_2'\} = \{k_1, k_2\}$, we have
\begin{align*}
    d_{k_1' \to k_2'} \ge d_{k_1' \to k_2'}^{(\ge \mSet_{k_1, k_2})} = 1 - O(\max\{n^{-1}, \delta_n\}), \\
    d_{k_2' \to k_1'} = e^{-\Omega(n)}.
\end{align*}
\end{lemma}

\noindent\textit{Proof.}
Consider subfamilies $\ff_k'$, $k \in [\ell]$. According to Theorem~\ref{theorem: maximal families product}, we have
\begin{align*}
    \prod_{k = 1}^\ell |\ff_k'| 
    \ge 
    (1 - \delta_n)^\ell 
    \Omega(n^{\sigma})
    2^n
\end{align*}
for $\sigma = \sum_{S \in \binom{[\ell]}{2}} \mSet_S$. For any $k_1\ne k_2$, sets from the families $\ff_{k_1}'$ and $\ff_{k_2}'$ can  intersect only inside $\probabilisticCenter_{k_1} \sqcup \probabilisticCenter_{k_2}$ due to~\ref{tentacle lemma: 2-factor structure condition} and~\ref{tentacle lemma: probabilistic center presence condition}. Otherwise, there is an element that belongs to three families. 

Consider an arbitrary $s \in \left [\mSet_{k_1, k_2} \right ]$. Define
\begin{align*}
    d_{k_1 \to k_2}'^{(\ge s)} = \frac{
        |\{ F \in \ff_{k_1}' \mid |F \cap \probabilisticCenter_{k_2}| \geqslant s\}|
    }{
        |\ff_{k_1}'|
    }.
\end{align*}
Let $t = \mSet_{k_1, k_2} + 1 - s$. Then families
\begin{align*}
    \ff_{k_1 \setminus k_2}'^{(<s)} &= \{F \in \ff_{k_1}' \mid |F \cap \probabilisticCenter_{k_2}| < s \}, \\
    \ff_{k_2 \setminus k_1}'^{(<t)} &= \{F \in \ff_{k_2}' \mid |F \cap \probabilisticCenter_{k_1}| < t \}
\end{align*}
are at most $\mSet_{k_1, k_2} - 1$-overlapping whenever $\mSet_{k_1, k_2} > 0$. Thus,
\begin{align*}
    \left |\ff_{k_1 \setminus k_2}'^{(<s)} \right |
    \left |\ff_{k_1 \setminus k_2}'^{(<t)} \right | 
    \prod_{k \not \in \{k_1, k_2\}}
         |\ff_k'|
    =
    O (n^{\sigma - 1} 2^n).
\end{align*}
due to Theorem~\ref{theorem: maximal families product}. Since $\left |\ff_{k_1 \setminus k_2}'^{(<s)} \right | = \left (1 - d_{k_1 \to k_2}'^{(\ge s)} \right ) |\ff_{k_1}'|$, we conclude
\begin{align}
\label{eq: direction dilemma}
    \left (
        1 - d_{k_1 \to k_2}'^{(\ge s)}
    \right )
    \left (
        1 - d_{k_2 \to k_1}'^{(\ge t)}
    \right )
    = O(n^{-1}) \cdot (1 - \delta_n)^{-\ell}
\end{align}
which is $O(n^{-1})$ due to conditions of the lemma. In this case, we have
\begin{align*}
    \text{either } &  d_{k_1 \to k_2}'^{(\ge s)} = 1 - O(n^{-1/2})  \\
    \text{or } & d_{k_1 \to k_2}'^{(\ge t)} = 1 - O(n^{-1/2}).
\end{align*}
Without loss of generality, we suppose that the first inequality holds. Using Claim~\ref{claim: degrees consitence}, we get
\begin{align}
\label{eq: expectation s, direction of tentacles}
    d_{k_1 \to k_2}^{(\ge s)} \ge 1 - O(n^{-1/2})
\end{align}
Our next aim is to show that this inequality implies 
\begin{align*}
    d_{k_2 \to k_1}^{(\ge t)} = e^{-\Omega(n)}.
\end{align*}

Fix some $T \in \binom{\probabilisticCenter_{k_1}}{t}$. We have
\begin{align}
    f_{k_1}(T)  & := \frac{
        \left | \{
            F \in \ff_{k_1} \mid F \cup T \not \in \ff_{k_1}
        \} \right |
    }{
        |\ff_{k_1}|
    } = \sum_{T' \subset T} \frac{
        \left | \{ 
            F \in \ff_{k_1}(T', T) \mid F \not \in \ff_{k_1}(T)
        \} \right |
    }{
        |\ff_{k_1}|
    } \nonumber \\
    & = \sum_{T' \subset T} \frac{
        |\ff_{k_1}(T', T)| - |\ff_{k_1}(T)|
    }{
        |\ff_{k_1}|
    } = 1 - 2^{|T|} d_{k_1}(T), \label{eq: f(T) fraction is small}
\end{align}
where the third equality is since $\ff_{k_1}(T) \subset \ff_{k_1}(T', T)$ due to the fact that $\ff_{k_1}$ is down-closed.

\begin{wrapfigure}[14]{r}{0.4\textwidth}
    \centering
    \includegraphics[width=0.35\textwidth]{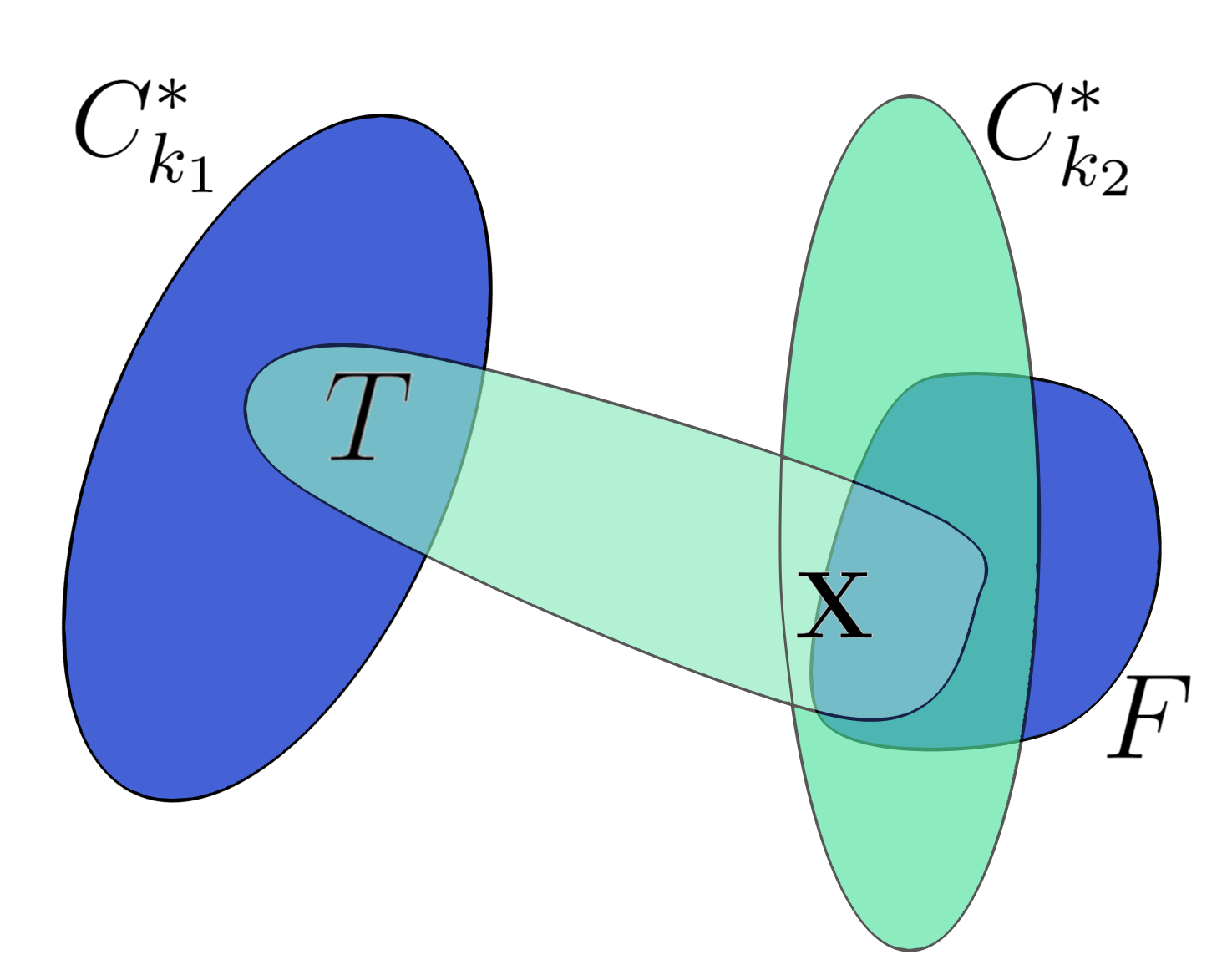}
    \caption{Illustration for the proof of Lemma~\ref{lemma: direction of tentacles}. Different sets' colors correspond to different families.}
    \label{fig: tentacle lemma illustration}
\end{wrapfigure}

Consider a random subset $\randomSubset$ uniformly distributed over $2^{\probabilisticCenter_{k_2}}$. All random variables $\varepsilon_x = \indicator[x \in \randomSubset] \sim Bern(1/2)$ for different $x \in \probabilisticCenter_{k_2}$ are jointly independent.

Assume $\randomSubset \cup T \in \ff_{k_2}$. If $|F \cap \randomSubset| \ge s$, then $F \cup T \not \in \ff_{k_1}$, so
\begin{align*}
    \mathbf{f}_{k_1}^{\ge s} :=  \frac{
        \left | \{
           F \in \ff_{k_1} \mid | F \cap \randomSubset | \ge s 
        \} \right |
    }{
        | \ff_{k_1}|
    } \le f_{k_1}(T).
\end{align*}
The latter  is at most $2^{t + 1} \alpha_n$ due to~\eqref{eq: f(T) fraction is small} and Claim~\ref{claim: inner set degree}. Consequently,
\begin{align}
\label{eq: random subset implication}
    \PP \left [
        \randomSubset \cup T \in \ff_{k_2}
    \right ]
    & \le
    \PP \left [
        \mathbf{f}_{k_1}^{\ge s} \le 2^{t + 1} \alpha_n
    \right ].
\end{align}
For the illustration see Figure~\ref{fig: tentacle lemma illustration}. 

To bound the RHS of~\eqref{eq: random subset implication}, we use Azuma's inequality. Define $d_{k_1}(P, \probabilisticCenter_{k_1}) = |\ff_{k_2}(P, \probabilisticCenter_{k_2})|/|\ff_{k_1}|$. Then
\begin{align}
\label{eq: random intersection fraction}
    \mathbf{f}_{k_1}^{\ge s} =  \sum_{p = s}^{\mSet_{k_1, k_2}} 
        \sum_{P \in \binom{\probabilisticCenter_{k_2}}{p}}
            d_{k_1}(P, \probabilisticCenter_{k_2})
            \prod_{x \in P}
                \varepsilon_x.
\end{align}
It implies
\begin{align}
\label{eq: expectation in tentacle lemma lower bound}
    \EE \left [ \mathbf{f}_{k_1}^{\ge s} \right ] & =  \EE \left [ 
        \sum_{p = s}^{\mSet_{k_1, k_2}} 
            \sum_{P \in \binom{\probabilisticCenter_{k_2}}{p}}
                d_{k_1}(P, \probabilisticCenter_{k_2})
                \prod_{x \in P}
                    \varepsilon_x
    \right ] \nonumber \\
    & \ge
    2^{-\mSet_{k_1, k_2}} d_{k_1 \to k_2}^{(\ge s)} \nonumber \\
    & \ge 2^{-\mSet_{k_1, k_2}} \left (1 - O(n^{-1/2}) \right )
\end{align}
due to~\eqref{eq: expectation s, direction of tentacles}. Consider an exposure martingale 
\begin{align*}
    Y_{w} = \EE \left [
        \mathbf{f}_{k_1}^{\ge s} \, | \, \sigma(\varepsilon_{x_1}, \varepsilon_{x_2}, \ldots, \varepsilon_{x_w})
    \right ], \quad Y_0 = \EE \left [ \mathbf{f}_{k_1}^{\ge s} \right ],
\end{align*}
where $\sigma(\varepsilon_{x_1}, \varepsilon_{x_2}, \ldots, \varepsilon_{x_w})$ is a sigma-algebra generated by $\varepsilon_{x_i}$, $i \in [w]$ and $x_1, x_2, \ldots, x_{|\probabilisticCenter_{k_2}|}$ are elements of $\probabilisticCenter_{k_2}$ in some order. Let $U_w = \{x_1, \ldots, x_w \}$, then
\begin{align*}
    |Y_{w} - Y_{w - 1}| & = 
    \left |
        \left (\varepsilon_{x_w} - \frac{1}{2} \right )
        \sum_{p = s-1}^{\mSet_{k_1, k_2}-1}
            \sum_{P \in \binom{\probabilisticCenter_{k_2} \setminus \{x_w\}}{p}}
                2^{-|P \setminus U_{w-1}|}
                d_{k_1}(P \cup \{x_w\}, \probabilisticCenter_{k_1}) 
                \prod_{x \in P \cap U_{w-1}} \varepsilon_{x}
    \right | \\
    & \le \frac{1}{2}
    \sum_{p = s-1}^{\mSet_{k_1, k_2}-1}
        \sum_{P \in \binom{\probabilisticCenter_{k_2} \setminus \{x_w\}}{p}}
            2^{-|P \setminus U_{w-1}|}
            d_{k_1}(P \cup \{x_w\}, \probabilisticCenter_{k_1}) \\
    & \le \frac{1}{2}
    \sum_{p = s-1}^{\mSet_{k_1, k_2}-1}
        \sum_{P \in \binom{\probabilisticCenter_{k_2} \setminus \{x_w\}}{p}}
            2^{-|P \setminus U_{w-1}|}
            d_{k_1}(P \cup \{x_w\}).
\end{align*}
Due to Proposition~\ref{proposition: multiplicative property of degrees}
\begin{align*}
    d_{k_1}(P\cup\{x_w\}) d_{k_2}(P \cup \{x_w\}) \le C_D n^{-|P| - 1}.
\end{align*}
Since $d_{k_2}(P \cup \{x_w\}) \ge 2^{-\mSet_{k_1, k_2}} - 2 \alpha_n$, we get $d_{k_1}(P \cup \{x_{w}\}) = \frac{2^{\mSet_{k_1, k_2}} C_D n^{-|P| - 1}}{1 - O(n^{-1})} = O(n^{-|P|-1})$ and there is a constant $C$ that does not depend on $w$ and such that
\begin{align*}
    |Y_w - Y_{w-1}| \le \frac{C}{n}.
\end{align*}
Due to Theorem~\ref{theorem: Hoeffding's iniequality}, we have
\begin{align*}
    \PP \left [
        Y_0 - Y_{|\probabilisticCenter_{k_2}|} \ge Y_0 - 2^{t + 1} \alpha_n
    \right ]
    \le 
    \exp \left (
        - \frac{n}{2 C^2} [Y_0 - 2^{t + 1} \alpha_n]^2
    \right )
\end{align*}
By the definition of conditional expectation, $Y_{|\probabilisticCenter_{k_2}|}$ and $Y_0$ stand for $\mathbf{f}_{k_1}^{\ge s}$ and $\EE [\mathbf{f}_{k_1}^{\ge s}]$ respectively. Since $\EE \left [ \mathbf{f}_{k_1}^{\ge s}\right ] = \Omega(1)$ due to~\eqref{eq: expectation in tentacle lemma lower bound}, we rearrange terms inside the probability and obtain
\begin{align*}
    \PP \left [ \mathbf{f}_{k_1}^{\ge s} \le 2^{t + 1} \alpha_n \right ] = e^{- \Omega(n)}.
\end{align*}
It provides the same bound on $\PP \left [ \randomSubset \cup T \in \ff_{k_2} \right ]$ via~\eqref{eq: random subset implication} uniformly over $T$. Note that 
\begin{align*}
    |\ff_{k_2}(T)| \le 
    \left |
        \ff_{k_2}|_{[n] \setminus \probabilisticCenter_{k_2}}
    \right |
    \PP \left [ 
        \randomSubset \cup T \in \ff_{k_2}
    \right ]
    2^{|\probabilisticCenter_{k_2}|}.
\end{align*}
Since $|\ff_{k_2}| \ge 2^{|\probabilisticCenter_{k_2}|}$ and $\left | \ff_{k_2} |_{[n] \setminus \probabilisticCenter_{k_2}} \right | \le n^{\sum_{k} \mSet_{k_1, k}}$, it implies
\begin{align}
\label{eq: tentacle lemma-degree expon bounding}
    d_{k_2}(T) \le n^{\sum_k \mSet_{k_1, k}} e^{- \Omega(n)}.
\end{align}
For any $T' \subset \probabilisticCenter_{k_1}$ of size greater than $t$ we have
\begin{align}
\label{eq: tentacle lemma-degree bounding}
    d_{k_2}(T') \le \frac{1}{\binom{|T'|}{t}} \sum_{T \in \binom{T'}{t}} d_{k_2}(T)
\end{align}
because of down-closeness. Therefore,
\begin{align*}
    d_{k_2 \to k_1}^{(\ge t)} & \le \sum_{t' \ge t} \sum_{T' \in \binom{\probabilisticCenter_{k_1}}{t'}} d_{k_2}(T', \probabilisticCenter_{k_1}) \\
    & \le \sum_{t' \ge t} \sum_{T' \in \binom{\probabilisticCenter_{k_1}}{t'}} d_{k_2}(T') \\
    & \overset{\text{due to~\eqref{eq: tentacle lemma-degree bounding}}}{\le} \sum_{t' = t}^{\mSet_{k_1, k_2}} \sum_{T' \in \binom{\probabilisticCenter_{k_1}}{t'}} \frac{1}{\binom{|T'|}{t}} \sum_{T \in \binom{T'}{t}} d_{k_1}(T) \\
    & \overset{\text{due to~\eqref{eq: tentacle lemma-degree expon bounding}}}{\le} P_t(n) e^{-\Omega(n)} = e^{-\Omega(n)},
\end{align*}
where the degree of a polynomial $P_t(n)$ is bounded as $n$ tends to infinity. Consider
\begin{align*}
    s_1 & = \max \left \{ s \mid d_{k_1 \to k_2} ^{(\ge s)} > \frac 1 2 \right \}, \\
    s_2 & = \max \left \{ s \mid  d_{k_2 \to k_1} ^{(\ge s)} > \frac 1 2 \right \}.
\end{align*}
We have $d_{k_1 \to k_2}^{(\ge s_1 + 1)}, d_{k_2 \to k_1}^{(\ge s_2 + 1)} = e^{-\Omega(n)}$. Note that  for large enough $n$ it holds that $s_1 + s_2 \le \mSet_{k_1, k_2}$ due to~\eqref{eq: direction dilemma} and Claim~\ref{claim: degrees consitence}, and
\begin{align*}
    \left | \ff_{k_1 \setminus k_2}^{(< s_1)} \right | = \left (1 - e^{- \Omega(n)} \right ) |\ff_{k_1}|, \\
    \left | \ff_{k_2 \setminus k_1}^{(< s_2)} \right | = \left (
        1 -  e^{- \Omega(n)}
    \right ) |\ff_{k_2}|.
\end{align*}
And due to Claim~\ref{claim: degrees consitence}
\begin{align}
    \frac{\left |\ff_{k_1 \setminus k_2}^{(\le s_1)} \cap \ff_{k_1}' \right |}{|\ff_{k_1}'|} \ge 1 - \delta_n - e^{- \Omega(n)}, \label{eq: tentacle lemma intersection with s1} \\
    \frac{\left |\ff_{k_2 \setminus k_1}^{(\le s_2)} \cap \ff_{k_2}' \right |}{|\ff_{k_2}'|} \ge 1 - \delta_n - e^{-\Omega(n)}. \label{eq: tentacle lemma intersection with s2}
\end{align}
Condition~\ref{tentacle lemma: lower family volume condition} guarantees, that
\begin{align*}
    \prod_{k \in [\ell]} |\ff_{k}| \le (1 - \delta_n)^{-\ell} \prod_{k \in [\ell]} |\ff_k'|.
\end{align*}
Combining the above with inequalities~\eqref{eq: tentacle lemma intersection with s1},~\eqref{eq: tentacle lemma intersection with s2}, we obtain
\begin{align}
\label{eq: tentacle lemma families product upper bound}
    \prod_{k \in [\ell]} |\ff_{k}| \le 
    \left (1 - \delta_n - e^{-\Omega(n)} \right )^{-2} 
    (1 - \delta_n)^{-\ell} 
    \left |\ff_{k_1 \setminus k_2}^{(\le s_1 )} \cap \ff_{k_1}' \right |
    \left |\ff_{k_2 \setminus k_1}^{(\le s_2 )} \cap \ff_{k_2}' \right |
    \prod_{k \not \in \{k_1, k_2\}} |\ff_k'|.
\end{align}
For a reminder, for an arbitrary $x$ there are at most 2 primed families such that some of their sets contain $x$. Put $A_{k_1, k_2} = \support \left (\ff_{k_1}' \wedge \ff_{k_2}' \right )$. Families
\begin{align*}
    \hh_k = 2^{\probabilisticCenter_k} \vee \bigvee_{k' \neq k} \binom{A_{k, k'}}{\le \mSet_{k, k'}}
\end{align*}
are $\mSet$-overlapping, and, consequently,
\begin{align}
\label{eq: LB for direction of tentacles}
    2^n \prod_{S \in \binom{[\ell]}{2}} \binom{A_S}{\le \mSet_S}
    & =
    \prod_{k \in [\ell]} |\hh_k| \le
    \prod_{k \in [\ell]} |\ff_k| \\
    & \overset{\text{due to~\eqref{eq: tentacle lemma families product upper bound}}}{\le}
    \left (1 - o(1)\right ) |\ff_{k_1 \setminus k_2}^{(\ge s_1 + 1)} \cap \ff_{k_1}'|
    |\ff_{k_2 \setminus k_1}^{(\ge s_2 + 1)} \cap \ff_{k_2}'| \prod_{k \in [\ell] \setminus \{k_1, k_2\}}
        |\ff_{k}'|. \nonumber
\end{align}
Due to Condition~\ref{tentacle lemma: 2-factor structure condition}, we have $\bigwedge_{s \in S} \ff_{k}' = \{\varnothing\}$ for any $S$ of cardinality greater than 2. Thus, from Theorem~\ref{corollary: rinott sets theorem} we infer
\begin{align}
\label{eq: rinot-saks for tentacles}
    |\ff_{k_1 \setminus k_2}^{(\le s_1 )} \cap \ff_{k_1}'| 
    |\ff_{k_2 \setminus k_1}^{(\le s_2 )} \cap \ff_{k_2}'| 
    & \prod_{k \in [\ell] \setminus \{k_1, k_2\}}
        |\ff_{k}'|
    \le \\
    & \le 2^n
    \binom{|A_{k_1, k_2} \cap \probabilisticCenter_{k_1}|}{\le s_2}
    \binom{|A_{k_1, k_2} \cap \probabilisticCenter_{k_2}|}{\le s_1}
    \prod_{S \in \binom{[\ell]}{2} \setminus \{k_1, k_2 \}}
        \binom{A_S}{\le \mSet_S} \nonumber
\end{align}
If $s_1 \neq 0$ and $s_2 \neq 0$ then
\begin{align*}
    \binom{|A_{k_1, k_2} \cap \probabilisticCenter_{k_1}|}{\le s_2}
    \binom{|A_{k_1, k_2} \cap \probabilisticCenter_{k_2}|}{\le s_1}
    \le c
    \binom{|A_{k_1, k_2}|}{\le s_1 + s_2}.
\end{align*}
for some $c < 1$. Comparing inequalities~\eqref{eq: rinot-saks for tentacles} and~\eqref{eq: LB for direction of tentacles} implies that either $s_1 = 0$ or $s_2 = 0$.
Thus, due to equation~\eqref{eq: direction dilemma},
\begin{align*}
    d_{k_1 \to k_2}'^{(\ge \mSet_{k_1, k_2})} & \ge 1 - O(n^{-1}), \\
    d_{k_1 \to k_2}^{(\ge \mSet_{k_1, k_2})} & \ge (1 - \delta_n) (1 - O(n^{-1}))  = 1 - O(\delta_n  + n^{-1}),
\end{align*}
where Claim~\ref{claim: degrees consitence} is used. Also, $d_{k_2 \to k_1} = e^{-\Omega(n)}$. \hfill \qedsymbol{}

Subfamilies described in the assumptions of the lemma exist due to Lemma~\ref{lemma: element excluding}. Thus, the lemma  guarantees that tentacles are directed mostly from one probabilistic center to another. Thus, we can construct an oriented graph $T_\ell$ on the vertex set $[\ell]$, that contains an oriented edge $(k_1, k_2)$ if 
\begin{align*}
    d_{k_2 \to k_1} = e^{- \Omega(n)}.
\end{align*}
In what follows, we assume that $\ff_k$'s come from some particular extremal example; the oriented graph $T_\ell$ is built according to this example and is supposed to be fixed here and after.
\begin{definition}
\label{definition: oriented graph}
    Suppose that $\ff_1, \ldots, \ff_\ell$ from an extremal example. Then the corresponding oriented graph $T_\ell$ with vertices $[\ell]$ has edges defined as follows
    \begin{itemize}
        \item if $\mSet_{k_1, k_2} > 0$ then the edge that connects $k_1$ and $k_2$ is directed from $k_1$ to $k_2$ if and only if $d_{k_2 \to k_1} = e^{-\Omega(n)}$,
        \item if $\mSet_{k_1, k_2} = 0$ then there is no edge between $k_1$ and $k_2$.
    \end{itemize}
\end{definition}
According to Lemma~\ref{lemma: direction of tentacles} and Lemma~\ref{lemma: element excluding}, this oriented graph is defined correctly for the extremal example and large enough $n$. 

We call a vertex $k$ a {\it source} in the tournament $T_\ell$ if it has no incoming edges and a {\it sink} if it has no outgoing edges. We have the following claim.
\begin{claim}
\label{claim: linear size of centers}
If a family from the extremal example is not a source,  then its probabilistic center has size linear in $n$.
\end{claim}
\begin{proof}
If a family $k$ is not a source, then there is an index $k'$ such that $(k', k) \in E(T_\ell)$. Hence, $d_{k' \to k}^{(\ge \mSet_{k', k})} \ge 1 - \delta_n$ for $\delta_n = O(n^{-1/2})$ due to Lemma~\ref{lemma: element excluding}. Combining Claim~\ref{claim: inner set degree} and Proposition~\ref{proposition: multiplicative property of degrees}, we obtain that $d_{k'}(S) = O(n^{-|S|})$ uniformly over $S \in \binom{\probabilisticCenter_{k}}{\mSet_{k, k'}}$. Since $d_{k' \to k}^{(\ge \mSet_{k', k})} = \sum_{S \in \binom{\probabilisticCenter_{k}}{\mSet_{k, k'}}} d_{k'}(S)$, $|\probabilisticCenter_{k}|$ should be linear in $n$.
\end{proof}

\subsection{Symmetrization lemma}
\label{subsection: symmetrization lemma}

The central lemma of this subsection uses a symmetrization argument inspired by the famous Zykov's symmetrization~\cite{zykov1949some} who used it for an alternative proof of Turan's theorem. Before we will provide the statement, we introduce some new notation. Given a family $\ff$ and a set $S$, define
\begin{align*}
    \ff \left (\excludeSet{S} \right ) & = \{ F \setminus S \mid F \in \ff \text{ and } S \subset F \}, \\
    \ff \left (\saveSet{S} \right ) & = \{ F \mid F \in \ff \text{ and } S \subset F \}.
\end{align*}
The family $\ff \left (\overset{\circ}{S} \right )$ coincides with previously defined $\ff(S)$. However, we need some notation to describe precisely whether the set $S$ remains or not. Symbols ``$\circ$'' and ``$\bullet$'' can be naturally combined in our context, for example,
\begin{align*}
    \ff \left ( \excludeSet{S} \sqcup \saveSet{T} \right ) = \{ F \setminus S \mid F \in \ff \text{ and } S \cup T \subset F \}.
\end{align*}
Besides, we remind that $\ff(\overline{S})$ is usually used for the sets that do not intersect $S$. For the sets that do not contain $S$ we use $\ff(\notContainSet{S}) := \ff \setminus \left [\ff(\saveSet{S}) \right ]$.

As before, we suppress braces of singletons.

Under a {\it tentacle} of a family $\ff_{k_0}$ we mean any its set $S_0 \subset \bigcup_{k' \mid (k, k') \in E(T_\ell)} \probabilisticCenter_{k'}$. As it was discussed in Section~\ref{section: sketch of the proof}, the idea of the lemma is enlarging families by chopping off tentacles of small degree and replacing them with a copy of a large degree tentacle.  The method can be described as follows: if $d_{k_0}(S_0)$ is small, choose $S$ of a large degree, then for each set $F \in \ff_{k_0}(\excludeSet{S})$ add $F \cup S_0$ to the family $\ff_k$. 

Such a scheme typically breaks the $\mSet$-overlapping property. Instead, we use the modified procedure. It consists of three steps:
\begin{enumerate}
    \item \label{item: simmetrization procedure, step 1} Replace all families $\ff_k$ with $\ff_k' = \ff_k(\notContainSet{S_0})$.
    \item \label{item: simmetrization procedure, step 2} Choose $S$ such that $\left | \ff_{k_0}(\excludeSet{S}) \cap \bigcap_{S' \subsetneq S_0} \ff_{k_0}(\excludeSet{S'})\right | / |\ff_{k_0}|$ is large and for each porbabilistic center $\probabilisticCenter_k$ it holds that $|S \cap \probabilisticCenter_k | = |S_0 \cap \probabilisticCenter_k|$.
    \item For each $F \in \ff_k(\excludeSet{S}) \cap \bigcap_{S' \subsetneq S_0} \ff_k(\excludeSet{S'})$, add $F \cup S_0$ to $\ff_{k_0}'$. Denote this new family by $\tilde \ff_{k_0}$.
\end{enumerate}
If we would not intersect $\ff_{k_0}(\excludeSet{S})$ with $\bigcap_{S' \subsetneq S_0} \ff_{k_0}(\excludeSet{S'})$, then the $\mSet$-overlapping property can be broken. Indeed, no one guarantees that there is no set $F$ in $\ff_{k_0}(\excludeSet{S})$ and a set $F' \in \ff_{k'}'$ for some $k' \in [\ell] \setminus \{k_0\}$ such that $|(F \cup S_0) \cap F'| > \mSet_{k_0, k'}$. On the other hand, families $\ff_{k_0}(\saveSet{S'})$ and $\ff_{k'}(\notContainSet{S_0})$ preserve the $\mSet$-overlapping property for any $S' \subsetneq S_0$, and the intersection from the step~\ref{item: simmetrization procedure, step 2} makes the procedure correct. If $|S_0| = 1$ then $\bigcap_{S' \subsetneq S_0} \ff_k(\excludeSet{S'})$ is assumed to be $\ff_k$.

Under some conditions $|\tilde \ff_{k_0}| \cdot \prod_{k' \in [\ell] \setminus \{k_0\}} |\ff_{k'}'|$ turns out to be larger than the initial product, and that leads to contradiction if the families $\ff_k$ form an extremal example. For instance, step~\ref{item: simmetrization procedure, step 1} should not reduce families significantly, and thus all degrees $d_k(S_0)$ must be small. In what follows, we show that it is true if the tentacle $S_0$ intersect several centers. However, if $S_0 \subset \probabilisticCenter_{k_1}$ for some $k_1$, $d_{k_1}(S_0)$ is asymptotically constant, and that makes us find another way of pruning instead of the one described in step~\ref{item: simmetrization procedure, step 1}. 

\begin{claim}
\label{claim: double counting} 
Let $\ff_k$ be the families from the extremal example and $k_0, k_1$ be arbitrary distinct indices such that $k_0 \to k_1$ is an edge in the oriented graph $T_\ell$. Let $S_0 \in \ff_{k_0}$ be a subset of $\probabilisticCenter_{k_1}$ of size at most $\mSet_{k_0, k_1}$ and $\mathcal{L} \subset \binom{\probabilisticCenter_{k_1}}{|S_0|}$ be an $|S_0|$-uniform subfamily of $\ff_{k_0}$. Define
\begin{align*}
    \mathcal{R} = \left \{
        T \in \binom{[n] \setminus \probabilisticCenter_{k_0}}{\mSet_{k_0, k_1} + 1 - |S_0|}
        \mid 
        d_{k_1}(T) \ge n^{- 2 \mSet_{k_0, k_1} + |S_0| - 3}
    \right \}.
\end{align*}
Define a bipartite graph $G = (\mathcal{L}, \mathcal{R}; E)$ with edges
\begin{align*}
    E = \left \{
        \{S, T\} 
        \mid
        S \in \mathcal{L}, 
        T \in \mathcal{R} 
        \text{ and } S \cup T \in \ff_{k_0}(\notContainSet{S_0})
    \right \}.
\end{align*}
Then there is $S_1 \in \mathcal{L}$ and a constant $C$ such that $$deg_G(S_1) \le \frac{C n^{\mSet_{k_0, k_1}} \log_2 n }{|\mathcal{L}|},$$
where $deg_G(\cdot)$ is a usual vertex degree in the graph $G$.
\end{claim}

\textbf{Remark.} For the proof of Claim~\ref{claim: double counting}, we do not need to exclude $S_0$ in the definition of edges $E$. But we will use it in what follows.

\begin{proof}
    We denote by $\mathcal{N}_G(\cdot)$ a set of neighbours in the graph $G$. First, we uniformly bound the number of neighbours $|\mathcal{N}_G(T)|$ of a vertex $T \in \mathcal{R}$. Let $\mathcal{N}_G(T)^{\uparrow} := \{F \subset \probabilisticCenter_{k_1} \mid \exists F' \in \mathcal{N}_G(T) \text{ s.t. } F' \subset F \}$ be the upper closure of $\mathcal{N}_G(T)$ in $\probabilisticCenter_{k_1}$. Then, by the definition of the graph $G$,
    \begin{align}
    \label{eq: graph claim T degree upper closure bound}
        |\ff_{k_1}(T)|_{\probabilisticCenter_{k_1}}| \le 2^{|\probabilisticCenter_{k_1}|} - |\mathcal{N}_G(T)^{\uparrow}|.
    \end{align}
    
    To estimate the size of $\mathcal{N}_G(T)^\uparrow$ and, thus, the size of $\mathcal{N}_G(T)$, we use Theorem~\ref{theorem: Kruskal-Katona theorem} which bound the upper shadow $\partial_u$ (see Definition~\ref{definition: upper shadow}) of any uniform family. We use $\partial_u^i$ for an $i$-th consecutive composition of upper shadows: $\partial_u^i = \partial_u \circ \partial_u \ldots \circ \partial_u$ $i$ times. Then $\mathcal{N}_G(T)^\uparrow$ can be expressed in terms of shadows:
    \begin{align*}
        \mathcal{N}_G(T)^\uparrow = \bigsqcup_{i = 0}^{|\probabilisticCenter_{k_1}| - |S_0| - 1} \partial_u^i \mathcal{N}_G(T).
    \end{align*}
    
    Denote by $\mathcal{B}(|\mathcal{N}_G(T)|)$ first $|\mathcal{N}_G(T)|$ sets of $\binom{\probabilisticCenter_{k_1}}{|S_0|}$ in the lexicographical order (we refer reader to Section~\ref{subsection: Kruskal-Katona theorem}, in Tools, for definitions). We call such family an {\it initial segment} in $\binom{\probabilisticCenter_{k_1}}{|S_0|}$ of size $|\mathcal{N}_G(T)|$. One can easily check that $\partial_u^i \mathcal{B}(|\mathcal{N}_G(T)|)$ is a again an initial segment of $\binom{\probabilisticCenter_{k_1}}{|S_0| + i}$.

    Consider the smallest $s$ such that
    \begin{align*}
        |\mathcal{N}_G(T)| \le s \binom{|\probabilisticCenter_{k_1}|}{|S_0| - 1}.
    \end{align*}
    Hence $|\mathcal{N}_G(T)| > (s - 1) \binom{|\probabilisticCenter_{k_1}|}{|S_0| - 1}$, and $\mathcal{B}(|\mathcal{N}_G(T)|)$ contains all sets of size greater than $|S_0|$ that start with $t$-th smallest element of $\probabilisticCenter_{k_1}$, for each $t < s$. Thus, due to Theorem~\ref{theorem: Kruskal-Katona theorem},
    \begin{align*}
        |\mathcal{N}_G(T)^{\uparrow}| \ge |\mathcal{B}(|\mathcal{N}_G(T)|)^\uparrow| \ge 2^{|\probabilisticCenter_{k_1}|} - 2^{|\probabilisticCenter_{k_1}| - s + 1} - \binom{|\probabilisticCenter_{k_1}|}{\le |S_0| - 1}.
    \end{align*}
    Since $d_{k_1}(T) \le n^{\sigma} \left |\ff_{k_1}(T)|_{\probabilisticCenter_{k_1}} \right | / |\ff_{k_1}|$ and $|\ff_{k_1}| \ge 2^{|\probabilisticCenter_{k_1}|}$, using~\eqref{eq: graph claim T degree upper closure bound} we obtain
    \begin{align*} 
        d_{k_0}(T) 2^{|\probabilisticCenter_{k_1}|} \le n^{\sigma} \left (2^{|\probabilisticCenter_{k_1}| - s + 1} + \binom{|\probabilisticCenter_{k_1}|}{\le |S_0| - 1} \right ).
    \end{align*}
    The definition of $\mathcal{R}$ and Claim~\ref{claim: linear size of centers} ensure us that $s$ can be at most logarithmic in $n$. Thus, $|\mathcal{N}_G(T)| = O \left ( n^{|S_0| - 1} \log_2 n \right )$.
    We get,
    \begin{align*}
        \sum_{S \in \mathcal{L}} deg_G(S) = \sum_{T \in \mathcal{R}} deg_G(T) \le |\mathcal{R}| \cdot O(n^{|S_0| - 1} \log_2 n) \le n^{\mSet_{k, k_0}} O (\log_2 n).
    \end{align*}
    Finally, choose $S_1$ such that $deg_G(S_1) \le |\mathcal{L}|^{-1} n^{\mSet_{k, k_0}} O(\log n)$.
\end{proof}

Now we are ready to formalize the symmetrization argument.

\begin{lemma}
\label{lemma: symmetrization argument}
Let $\ff_1, \ldots, \ff_\ell$ be families from the extremal example. Take any $k_0$ and a subfamily $\sDomain_{k_0}$ of $\bigvee_{k' \in Out_{k_0}} \binom{\probabilisticCenter_{k'}}{\le \mSet_{k_0, k'}}$ such that $\sDomain_{k_0} |_{\probabilisticCenter_{k'}}$ is uniform for any $k' \in Out_{k_0}$. Let $S_0$ be arbitrary set from $\sDomain_{k_0}$ and $\Delta < 1$ be a positive real number, satisfying one of the (not necessary mutually exclusive) conditions:
\begin{enumerate}[label=(\alph*)]
    \item \label{symmetrization lemma: case 1} there exists $k_1$ such that $S_0 \subset \probabilisticCenter_{k_1}$ and $\max_{k' \in [\ell] \setminus \{k_0, k_1\}} d_{k'}(S_0) \le \frac{\Delta / 2}{\ell - 2}$,
    \item \label{symmetrization lemma: case 2} $\max_{k' \in [\ell] \setminus \{k_0\}} d_{k'}(S_0) \le \frac{\Delta / 2}{\ell - 1}$.
\end{enumerate} 
Then, if \ref{symmetrization lemma: case 1} holds, then
\begin{align*}
    \mathcal{I}_\Delta = \left \{ S \in \sDomain_{k_0} \mid \frac{1}{|\ff_{k_0}|} \left | \ff_{k_0}(\excludeSet{S}) \cap \bigcap_{S' \subsetneq S_0} \ff_{k_0}(\excludeSet{S'}) \right | \ge d_k(S_0) + \Delta \right \}
\end{align*}
has cardinality at most 
\begin{align*}
    \frac{1 - \Delta/2}{\Delta(1 - \Delta) / 2 - O(n^{-\mSet_{k_0, k_1} - 2})} \cdot O(n^{-1} \log n).
\end{align*}
If \ref{symmetrization lemma: case 2} holds, then $\mathcal{I}_\Delta = \varnothing$.
\end{lemma}

\begin{proof}
Here and after we use notation of Claim~\ref{claim: double counting}. If \ref{symmetrization lemma: case 1} holds, apply Claim~\ref{claim: double counting} for $\mathcal{L} = \mathcal{I}_\Delta$. As a result, we obtain $S_1 \in \mathcal{I}_\Delta$ such that its neighbourhood $\mathcal{N}_G (S_1) \subset \mathcal{R}$ in the graph $G$ has cardinality $|\mathcal{I}_\Delta|^{-1} \cdot O(n^{\mSet_{k_0, k_1}} \log n)$. Moreover, we have
\begin{align}
\label{eq: T out of R, symmetrization lemma}
    \sum_{T \in \binom{[n] \setminus \probabilisticCenter_{k_0}}{\mSet_{k_0, k_1} + 1 - |S_0|} \setminus \mathcal{R}} d_{k_1}(T)
    \le n^{\mSet_{k_0, k_1} + 1 - |S_0|} n^{- 2 \mSet_{k_0, k_1} + |S_0| - 3} = n^{- \mSet_{k_0, k_1} - 2}.
\end{align}
If $S_0$ completely lies in $\probabilisticCenter_{k_1}$ for some $k_1$, i.e. we are in case~\ref{symmetrization lemma: case 1}, then modify families in the following way:
\begin{align}
\label{eq: graph symmetrization primed families definition}
    \ff_{k}' = \begin{cases}
    \ff_{k}(\notContainSet{S_0}), & k \neq k_1, \\
    \ff_{k \setminus k_0} \cap \bigcap_{T \in \binom{[n] \setminus \probabilisticCenter_{k_0}}{\mSet_{k_0, k_1} + 1 - |S_0|} \setminus \mathcal{R}} \ff_{k_1}(\notContainSet{T}), & k = k_1.
    \end{cases}
\end{align}
In case~\ref{symmetrization lemma: case 2}, just consider $\ff_{k}' =  \ff_{k}(\notContainSet{S_0})$ for all $k$ and choose $S_1 \in \mathcal{I}_\Delta$ arbitrary.

Next, we need to add sets to $\ff_{k_0}$. Define:
\begin{align*}
    \mathcal{V}_{k_0} =
    \begin{cases}
        \left [ \ff_{k_0}(\excludeSet{S_1}) \cap \bigcap_{S' \subsetneq S_0} \ff_{k_0}(\excludeSet{S'}) \right ] \setminus \bigcup_{T \in \mathcal{N}_G(S_1)} \ff_{k_0}(\excludeSet{S_1} \cup \saveSet{T}), & \text{if \ref{symmetrization lemma: case 1} holds}, \\
        \ff_{k_0}(\excludeSet{S_1}) \cap \bigcap_{S' \subsetneq S_0} \ff_{k_0}(\excludeSet{S'}), & \text{if~\ref{symmetrization lemma: case 2} holds}.
    \end{cases}
\end{align*}
Also define $\tilde{\ff}_{k_0} = \ff_{k_0}' \sqcup (\mathcal{V}_{k_0} \vee \{S_0\})$. The families $\ff_{k}'$, $k \neq k_0$, and $\tilde \ff_{k_0}$ are the desired ``symmetrized'' families.

We should check that the $\mSet$-overlapping property is preserved. First, we check that the families $\tilde{\ff}_{k_0}$ and $\ff_k', k \in [\ell] \setminus \{k_0\},$ are down-closed. The case of $\ff_{k}'$, $k  \in [n] \setminus \{k_0\}$, is obvious. Proving that for $\tilde{\ff}_{k_0}$ is a bit more tricky. Consider a set $F \in \tilde{\ff}_{k_0}$.
\begin{itemize}
    \item If $S_0 \subset F$ then $F \not \in \ff_{k_0}'$ and, hence, $F \setminus S_0 \in \mathcal{V}_{k_0}$. By the definition of $\mathcal{V}_{k_0}$, we have $F \setminus S_0 \in \ff_{k_0}(\excludeSet{S_1}) \cap \bigcap_{S' \subsetneq S_0} \ff_{k_0}(\excludeSet{S'})$. The latter family is down-closed; consequently, for any $G \subset F \setminus S_0$ we also have $G \in \ff_{k_0}(\excludeSet{S_1}) \cap \bigcap_{S' \subsetneq S_0} \ff_{k_0}(\excludeSet{S'})$. Moreover, if \ref{symmetrization lemma: case 1} holds, for any $T \in \mathcal{N}_G(S_1)$, we have $T \not \subset F$ which implies $T \not \subset G$, and so $G \not \in \bigcup_{T \in \mathcal{N}_G(S_1)} \ff_{k_0}(\excludeSet{S_1} \cup \saveSet{T})$. Thus, $G \cup S_0 \in \mathcal{V}_{k_0} \vee \{S_0\}$, and so $G \cup S_0 \in \tilde{\ff}_{k_0}$.  
    
    Since $G \in \ff_{k_0}(\excludeSet{S_1}) \cap \bigcap_{S' \subsetneq S_0} \ff_{k_0}(\excludeSet{S'})$, for any $S' \not \subset S_0$ we have $G \cup S' \in \ff_{k_0}$. But $S_0 \subsetneq G \cup S'$, and, hence, $G \cup S' \in \ff_{k_0}' \subset \tilde{\ff}_{k_0}$. As a result, any subset of $F$ belongs to $\tilde{\ff}_{k_0}$.
    
    \item If $S_0 \not \subset F$, then $F \in \ff_{k_0}'$ which is down-closed.
\end{itemize}

Since we verified down-closeness, it is enough to check that there are no two families indexed by $k'$ and $k''$ that both contain a set of size $\mSet_{k', k''} + 1$ due to Claim~\ref{claim: hypergraph point of view}. Consider several cases.
\begin{itemize}
    \item If $k_0 \not \in \{k', k''\}$, then $\ff_{k'}' \subset \ff_{k'}$ and $\ff_{k''}' \subset \ff_{k''}$, and, consequently, they can not share a set of size $\mSet_{k', k''} + 1$.
    \item If $k_0 \in \{k', k''\}$, say $k'' = k_0$, and $k' \neq k_1$ then $\mathcal{V}_{k_0} \vee \{S_0\}$ and $\ff_{k'}' = \ff_{k'}(\notContainSet{S_0})$ have no common sets. At the same time, $\ff_{k_0}' \subset \ff_{k_0}$ and $\ff_{k'}' \subset \ff_{k'}$ and, thus, $\tilde \ff_{k_0} \cap \ff_{k'}'$ does not contain a set of size $\mSet_{k_0, k'} + 1$.
    \item The final case $\{k', k''\} = \{k_0, k_1\}$ appears only if the case~\ref{symmetrization lemma: case 1} holds. Suppose that there exists $F \in \tilde \ff_{k_0} \cap \ff_{k_1}'$ of size $\mSet_{k_0, k_1} + 1$. If $S_0 \subset F$, then $F \setminus S_0 \in \mathcal{V}_{k_0}$. In particular, 
    \begin{align*}
        F \setminus S_0 \not \in \bigcup_{T \in \mathcal{N}_G(S_1)} \ff_{k_0}(\excludeSet{S_1} \cup \saveSet{T})
    \end{align*}
    by the definition of $\mathcal{V}_{k_0}$. Since $|F \setminus S_0| = \mSet_{k_0, k_1} + 1 - |S_0|$, it implies $F \setminus S_0 \not \in \mathcal{R}$. At the same time, the definition of $\ff_{k_1}'$ requires
    \begin{align*}
        F \in \bigcap_{T \in \binom{[n] \setminus \probabilisticCenter_{k_0}}{\mSet_{k_0, k_1} + 1 - |S_0|} \setminus \mathcal{R}} \ff_{k_1}(\notContainSet{T}),
    \end{align*}
    that necessarily implies either $F \setminus S_0 \in \mathcal{R}$ or $F \cap \probabilisticCenter_{k_0} \neq \varnothing$. We saw the former does not hold just above. By the definition~\eqref{eq: graph symmetrization primed families definition}, sets from $\ff_{k_1}'$ do not intersect $\probabilisticCenter_{k_0}$, so the latter also leads to contradiction.
\end{itemize}

The last point is to consider the product of the modified families. In case~\ref{symmetrization lemma: case 1}, we have
\begin{align}
    \frac{|\tilde \ff_{k_0}|}{|\ff_{k_0}|} & = \frac{|\ff_{k_0}(\notContainSet{S_0})|}{|\ff_{k_0}|} + \frac{|\mathcal{V}_{k_0}|}{|\ff_{k_0}|} \nonumber \\
    & \ge \left (1 - d_{k_0}(S_0) \right ) + \left (\frac{\left |\ff_{k_0}(\excludeSet{S_1}) \cap \bigcap_{S' \subsetneq S_0} \ff_{k_0}(\excludeSet{S'}) \right |}{|\ff_{k_0}|}- \sum_{T \in \mathcal{N}_G(S_1)} d_{k_0}(S_1 \cup T) \right ). \label{eq: tilde F estimation, symmetrization lemma}
\end{align}
Since $S_1 \in \mathcal{I}_{\Delta}$, we have $\left |\ff_{k_0}(\excludeSet{S_1}) \cap \bigcap_{S' \subset S_0} \ff_{k_0}(\excludeSet{S'}) \right | \ge \left (d_{k_0}(S_0) + \Delta \right ) |\ff_{k_0}|$ by the definition. Meanwhile, by the definition of $\mathcal{R}$, we have $T \cap \probabilisticCenter_{k_0} = \varnothing$ for each $T \in \mathcal{N}_G(S_1)$. Due to Proposition~\ref{proposition: multiplicative property of degrees}, it holds that
\begin{align*}
    d_{k_0}(S_1 \cup T) \prod_{k \in [\ell] \setminus \{k_0\}} d_{k} \left ((S_1 \cup T) \cap \probabilisticCenter_{k} \right ) = O(n^{- |S_1 \cup T|}) = O(n^{- \mSet_{k_0, k_1} - 1}),
\end{align*}
so $d_{k_0}(S_1 \cup T) = O(n^{-\mSet_{k_0, k_1} - 1})$ due to Claim~\ref{claim: inner set degree}. Claim~\ref{claim: double counting} guarantees that $|\mathcal{N}_G(S_1)| = deg_G(S_1) = \frac{O(n^{\mSet_{k_0, k_1}}) \log n}{|\mathcal{I}_{\Delta}|}$, and, combining the above with inequality~\eqref{eq: tilde F estimation, symmetrization lemma}, we obtain
\begin{align}
\label{eq: tilde F result, symmetrization lemma}
    \frac{|\tilde \ff_{k_0}|}{|\ff_{k_0}|} \ge 1 - d_{k_0}(S_0) + d_{k_0}(S_0) + \Delta - \frac{O(n^{-1} \log n)}{|\mathcal{I}_{\Delta}|} = 1 + \Delta - \frac{O(n^{-1} \log n)}{|\mathcal{I}_{\Delta}|}.
\end{align}
At the same time,
\begin{align*}
    \frac{|\ff_{k_1}'|}{|\ff_{k_1}|} & = \frac{1}{|\ff_{k_1}|} \left | \ff_{k_1 \setminus k_0} \cap \bigcap_{T \in \binom{[n] \setminus \probabilisticCenter_{k_0}}{\mSet_{k_0, k_1} + 1 - |S_0|} \setminus \mathcal{R}} \ff_{k_1}(\notContainSet{T}) \right |  \\
    & = \frac{1}{|\ff_{k_1}|} \left |\ff_{k_1} \setminus \left [\ff_{k_1 \to k_0} \cup \bigcup_{T \in \binom{[n] \setminus \probabilisticCenter_{k_0}}{\mSet_{k_0, k_1} + 1 - |S_0|} \setminus \mathcal{R}} \ff_{k_1}(\saveSet{T}) \right ] \right |  \\
    & \ge 1 - d_{k_1 \to k_0} - \sum_{T \in \binom{[n] \setminus \probabilisticCenter_{k_0}}{\mSet_{k_0, k_1} + 1 - |S_0|} \setminus \mathcal{R}} d_{k_1}(T). 
\end{align*}
The sum can  be bounded by inequality~\eqref{eq: T out of R, symmetrization lemma}, and we have $d_{k_1 \to k_0} = e^{- \Omega(n)}$ according to Lemma~\ref{lemma: direction of tentacles}. Thus,
\begin{align}
    \frac{|\ff_{k_1}'|}{|\ff_{k_1}|}
    & \ge
    1 - e^{-\Omega(n)} - n^{-\mSet_{k_0, k_1} - 2} \nonumber \\
    & \ge 1 - O(n^{-\mSet_{k_0, k_1} - 2}). \label{eq: k_1 estimation, symmetrization lemma}
\end{align}

Finally, we have $|\ff_{k}'| / |\ff_{k}| \ge 1 -  \frac{\Delta / 2}{\ell - 2}$ for $k \not \in \{k_0, k_1\}$ due to the condition. Thus, 
\begin{align*}
    \frac{|\tilde \ff_{k_0}| \prod_{k \in [\ell] \setminus \{k_0\}} |\ff_k'|}{\prod_{k \in [\ell]} |\ff_k|} & \ge 
    \left (1 + \Delta -  \frac{O(n^{-1} \log n)}{|\mathcal{I}_\Delta|}\right ) \left ( 1 - O(n^{- \mSet_{k_0, k_1} - 2}) \right ) \left (1 - \frac{\Delta/2}{\ell - 2} \right )^{\ell - 2} \\
    & \ge  \left (1 + \Delta -  \frac{O(n^{-1} \log n)}{|\mathcal{I}_\Delta|}\right ) \left ( 1 - O(n^{- \mSet_{k_0, k_1} - 2}) \right ) \left (1 - \Delta / 2 \right ),  
\end{align*}
where we used bounds~\eqref{eq: tilde F result, symmetrization lemma} and~\eqref{eq: k_1 estimation, symmetrization lemma}. Due to extremality of the families $\ff_k$, the right-hand side can not be large than $1$. Thus,
\begin{align*}
    \left (1 + \Delta -  \frac{O(n^{-1} \log n)}{|\mathcal{I}_\Delta|}\right )
    \left (1 - \Delta/2 \right )
    & \le 
     1 + O(n^{- \mSet_{k_0, k_1} - 2}) , \\
    1 + \Delta -  \frac{O(n^{-1} \log n)}{|\mathcal{I}_\Delta|}
    & \le \frac{ 1 + O(n^{- \mSet_{k_0, k_1} - 2})}{1 - \Delta/2}, \\
    |\mathcal{I}_\Delta| & \le \frac{1 - \Delta/2}{\Delta (1 - \Delta) / 2 - O(n^{-\mSet_{k_0, k_1} - 2}) } \cdot O \left ( \frac{\log n}{n} \right ).
\end{align*}
If \ref{symmetrization lemma: case 2} holds, assume that $\mathcal{I}_\Delta \neq \varnothing$. Then
\begin{align*}
    \frac{|\tilde{\ff_{k_0}}|}{|\ff_{k_0}|} & = \frac{|\ff_{k_0} (\excludeSet{S_0})|}{|\ff_{k_0}|} + \frac{|\mathcal{V}_{k_0}|}{|\ff_{k_0}|} \\
    & \ge 1 - d_{k_0}(S_0) + \frac{\left |\ff_{k_0}(\excludeSet{S_1}) \cap \bigcap_{S' \subset S_0} \ff_{k_0}(\excludeSet{S'}) \right |}{|\ff_{k_0}|}\\
    & \overset{\text{definition of } \mathcal{I}_\Delta}{\ge} 1 - d_{k_0}(S_0) + d_{k_0}(S_0) + \Delta \\
    & = 1 + \Delta.
\end{align*}
Due to the condition in case~\ref{symmetrization lemma: case 2}, we have $|\ff_{k}'| \ge \left (1 - \frac{\Delta / 2}{\ell - 1} \right ) |\ff_{k}|$ for $k \in [\ell] \setminus \{k_0\}$. Thus,
\begin{align*}
    \frac{|\tilde \ff_{k_0}| \cdot \prod_{k \in [\ell] \setminus \{k_0\}} |\ff_{k}'|}{
        \prod_{k \in [\ell]} |\ff_k|
    } & \ge (1 + \Delta) \left (1 - \frac{\Delta / 2}{\ell - 1} \right )^{\ell - 1} \\ 
    & \ge (1 + \Delta ) ( 1 - \Delta/2) \\
    & = 1 + \Delta/2 - \Delta^2/2 \\
    & > 1.
\end{align*}
Thus, the assumption that $\mathcal{I}_\Delta \neq \varnothing$ was wrong.


\end{proof}

\subsection{Localization of tentacles}
\label{subsection: localization of tentacles}

In this section, we refine structure of the extremal examples. In  Section~\ref{subsection: direction of tentacles}, we showed that each extremal example forms an oriented graph. Using our octopus metaphor, it means that for any two octopuses only one puts their tentacles into another's body. Next, we prove that tentacles of different octopuses essentially do not overlap.

We assume that there exists at least one pair $\{k, k'\}$ such that $\mSet_{k, k'} > 0$. Otherwise, results of this section become meaningless. The first result is about the source in the oriented graph $T_\ell$ if it exists.

\begin{proposition}
\label{proposition: empty center of the source}
Let $\ff_k$ be families from an extremal example and $T_\ell$ be an oriented graph defined above. If $k_0 \in V(T_\ell)$ is a source and $n$ is large enough, then $\probabilisticCenter_{k_0} = \varnothing$.
\end{proposition}

\begin{proof}
Let $k_0$ be a source and suppose the probabilistic center $\probabilisticCenter_{k_0}$ of $\ff_{k_0}$ is not empty. Take $x_0 \in \probabilisticCenter_{k_0}$. Consider the families $\mathcal G_k = \ff_k (\overline{x_0})$. Due to Lemma~\ref{lemma: direction of tentacles}, if $k \in [\ell] \setminus \{k_0\}$, then $|\mathcal G_k| \ge \left (1 - e^{-\Omega(n)} \right ) |\ff_k|$, and $|\mathcal G_{k_0}| \ge \frac{1}{2} |\ff_{k_0}|$ because of down-closeness. Next, consider an arbitrary index $k_1 \neq k_0$ such that $\mSet_{k_1, k_2} \neq 0$ for some $k_2$. Define
\begin{align*}
    \mathcal G_k' = 
    \begin{cases}
        \mathcal G_k, & \text{if } k \in [\ell] \setminus \{k_1, k_2\}, \\
        \mathcal G_k \cup 
         \{ F \cup  \{x_0\} \mid F \in \ff_{k \setminus k_0} \}, & \text{if } k = k_1, \\
        \mathcal G_k \cup 
        \{ F \cup  \{x_0\} \mid F \subset \probabilisticCenter_k \}, & \text{if } k = k_2.
    \end{cases}
\end{align*}

In this way, we obtain that $|\mathcal G_{k}'| \ge \left (1 - e^{- \Omega(n)} \right )|\ff_k|$ if $k \in [\ell] \setminus \{k_0, k_1, k_2\}$, and $|\mathcal G_{k_1}'| = (2 - e^{-\Omega(n)})|\ff_{k_1}|$ due to Lemma~\ref{lemma: direction of tentacles}. For $k = k_2$ we have
\begin{align*}
    |\ff_{k_2}| & \le n^{\sum_{k \in [\ell] \setminus k_2} \mSet_{k, k_2}} \cdot 2^{|\probabilisticCenter_{k_2}|},
\end{align*}
and, consequently,
\begin{align*}
    \frac{
        |\mathcal G_{k_2}'|
    }{
        |\ff_{k_2}|
    }
    \ge
    \begin{cases}
        1 - e^{-\Omega(n)} + n^{-\sum_{k \in [\ell] \setminus k_2} \mSet_{k, k_2}}, & \text{if } k_0 \neq k_2\\
        \frac 1 2 + n^{-\sum_{k \in [\ell] \setminus k_2} \mSet_{k, k_2}}, & \text{if } k_0 = k_2.
    \end{cases}
\end{align*}
Analogously,
\begin{align*}
    \frac{
        |\mathcal G_{k_0}'|
    }{
        |\ff_{k_0}|
    }
    \ge 
    \begin{cases}
        \frac 1 2 , & \text{if } k_0 \neq k_2 \\
        \frac{1}{2} +n^{-\sum_{k \in [\ell] \setminus k_2} \mSet_{k, k_2}}, & \text{if } k_0 = k_2.
    \end{cases}
\end{align*}
Thus, if $k_2 = k_0$, we have
\begin{align*}
    \prod_{k \in [\ell]} |\mathcal G_{k}'| \ge 2 \left (1 - e^{-\Omega(n)} \right ) \left (\frac 1 2 + n^{-\sum_{k \in [\ell] \setminus k_2} \mSet_{k, k_2}} \right ) \left (1 - e^{-\Omega(n)} \right )^{\ell - 2} \prod_{k \in \ell} |\ff_k| > \prod_{k \in [\ell]} |\ff_k|.
\end{align*}
If $k_2 \neq k_0$, we have
\begin{align*}
     \prod_{k \in [\ell]} |\mathcal G_{k}'| \ge 2 \left (1 - e^{-\Omega(n)} \right ) \left (1 - e^{-\Omega(n)} + n^{-\sum_{k \in [\ell] \setminus k_0} \mSet_{k, k_2}} \right ) \frac{1}{2} \left (1 - e^{-\Omega(n)} \right )^{\ell - 2} \prod_{k \in \ell} |\ff_k| > \prod_{k \in [\ell]} |\ff_k|.
\end{align*}
 Both cases contradict the extremality of $\ff_k$ since families $\mathcal G_{k}'$'s are $\mSet$-overlapping. Thus, the source has an empty center. 
\end{proof}

Together with Lemma~\ref{lemma: symmetrization argument}, Proposition~\ref{proposition: empty center of the source} allows to give a refined description of the structure of the families. We use the following notation: for an extremal family $\ff_k$ and an oriented graph $T_\ell$ define
\begin{align*}
    In_k & = \{k' \mid (k', k) \in E(T_\ell) \}, \\
    Out_k & = \{k' \mid (k, k') \in E(T_\ell) \}.
\end{align*}
Then, we state the following.

\begin{corollary}
\label{corollary: tentacles structure}
For two distinct families $\ff_k, \ff_{k'}$ from the extremal example, define
\begin{align*}
    A_{k, k'} = \left \{ x \in \probabilisticCenter_{k'} \mid d_k(x) \ge \frac{n^{-1}}{4 \ell - 2} \right \}.
\end{align*}
Then, for any edge $(k, k')$ in the oriented graph $T_\ell$, the set $A_{k, k'}$  has size $\Omega(n)$, and \newline $(A_{k, k'})_{(k, k') \in E(T_\ell)}$ is a partition of $[n]$. In addition, for each $k$ the center $\probabilisticCenter_k$ is partitioned by $A_{k, k'}$, $k' \in In_k$.
\end{corollary}
\begin{proof}
Given $x$, there is $k_0$ such that $x \in \probabilisticCenter_{k_0}$ due to Lemma~\ref{lemma: covering by centers}. Proposition~\ref{proposition: empty center of the source} guarantees that there is $k_1 \in In_{k_0}$. Moreover, from Lemma~\ref{lemma: direction of tentacles}
\begin{align*}
    \sum_{y \in \probabilisticCenter_{k_0}} d_{k_1}(y) \ge d_{k_1 \to k_0} = 1 - O(n^{-1/2}),
\end{align*}
and, consequently, the set $J = \{y \in \probabilisticCenter_{k_0} \mid d_{k_1}(y) \ge n^{-1}/2 \}$ has size $\Omega(n)$. 

Assume that $x \not \in \bigcup_{(k, k') \in E(T_\ell)} A_{k, k'}$. Then, $d_{k'}(x) < n^{-1}/(4\ell - 2)$ for any $k' \neq k_0$. Apply Lemma~\ref{lemma: symmetrization argument} with $S_0 = \{x\}$, $\sDomain_k = \binom{\probabilisticCenter_{k_0}}{1}$ and $\Delta = n^{-1}(1/2 - (4\ell - 2)^{-1})$. Then $\binom{J}{1} \subset \mathcal{I}_\Delta$. On the one hand, $\mathcal{I}_\Delta$ has logarithmic size, and on the other hand, $|\mathcal{I}_\Delta| \ge |J| = \Omega(n)$, a contradiction.

Obviously, $\probabilisticCenter_{k_0}$ is partitioned by sets $A_{k, k_0}$ for $k \in In_{k_0}$. Since $J \subset A_{k_1, k_0}$ by definition, $A_{k_1, k_0}$ has size $\Omega(n)$.
\end{proof}

Now, it is easy to obtain the following improvement of Lemma~\ref{lemma: element excluding}:
\begin{proposition}
\label{proposition: subfamilies 1/n}
Let $\ff_k$, $k \in [\ell]$ be families from the extremal example. Define
\begin{align*}
    Supp_k = \bigsqcup_{k' \in In_k} A_{k', k}\sqcup \bigsqcup_{k' \in Out_k} A_{k, k'} = \probabilisticCenter_{k} \sqcup \bigsqcup_{k' \in Out_k} A_{k, k'}.
\end{align*}
Then, for each $k$
\begin{align*}
    \left | \ff_k|_{Supp_k} \right | \ge \left (1 - \delta_n \right ) |\ff_k|,
\end{align*}
where $\delta_n = O(n^{-1})$.
\end{proposition}
\begin{proof}
Consider some particular $k_0 \in [\ell]$. Let $x \not \in Supp_{k_0}$. Then there are $k_1$ and $k_2$, $k_0 \not \in \{k_1, k_2\}$ such that $d_{k_1}(x) = \Theta(1)$ due to Lemma~\ref{lemma: covering by centers} and $d_{k_2}(x) = \Theta(n^{-1})$ according to Corollary~\ref{corollary: tentacles structure}. From Proposition~\ref{proposition: multiplicative property of degrees} we have
\begin{align*}
    d_{k_0}(x) d_{k_1}(x) d_{k_2}(x) = O(n^{-3}),
\end{align*}
and, consequently, $d_{k_0}(x) = O(n^{-2})$. Finally, we get
\begin{align*}
    \frac{|\ff_{k_0}|_{Supp_{k_0}}|}{|\ff_{k_0}|} \ge 1 - \sum_{x \not \in Supp_{k_0}} d_{k_0}(x) = 1 - O(n^{-1}). & \qedhere
\end{align*}
\end{proof}

One nice property of $\ff_k|_{Supp_k}$ is that $\ff_k|_{Supp_k} \subset \hh_k$, where
\begin{align}
\label{eq: definition of hh_k}
    \hh_k = 2^{\bigcup_{k' \in In_k} A_{k', k}} \vee \bigvee_{k' \in Out_k} \binom{A_{k, k'}}{\le \mSet_{k, k'}}.
\end{align}
For families $\hh_k, k\in [\ell]$, the following claim holds.
\begin{claim}
\label{claim: hh_k approximation quality}
Let $\ff_k$ be an extremal family. Then $\hh_k$ defined in~\eqref{eq: definition of hh_k} has cardinality
\begin{align*}
    |\hh_k| = (1 - O(n^{-1})) |\ff_k|.
\end{align*}
\end{claim}
\begin{proof}
    The proof is a consequence of the following chain of inequalities:
    \begin{align*}
        \prod_{k \in [\ell]} \left |\ff_{k} |_{Supp_k} \right | \le \prod_{k \in [\ell]} |\hh_k| \le \prod_{k \in [\ell]} |\ff_k|.
    \end{align*}
    Since $|\ff_{k}|_{Supp_k}| = (1 - O(n^{-1})) |\ff_k|$ for any $k$, the statement holds.
\end{proof}

Note, that the degrees $|\hh_k(S)|$ for different $S$ can be expressed preciesly as a function of $A_{k, k'}, k' \in Out_{k}$. For example, we have 
\begin{align*}
    \frac{|\hh_k(x)|}{|\hh_k|} = \frac{
        \binom{|A_{k, k'} \setminus \{x\}|
            }{
                \le \mSet_{k,k'} - 1
            }
    }{
        \binom{|A_{k, k'}|}{\le \mSet_{k,k'}}
    } = \frac{\mSet_{k, k'}}{|A_{k, k'}|} \left (1 + O(n^{-1}) \right ).
\end{align*}
for any $(k, k') \in E(T_\ell)$ and $x \in A_{k, k'}$.  In what follows, we show that for any set $S$ of cardinality at most $\sum_{k, k'} \mSet_{k, k'}$, we have
\begin{align*}
    \frac{|\ff_k(S)|}{|\ff_k|} \approx \frac{|\hh_k(S)|}{|\hh_k|},
\end{align*}
for an extremal $\ff_k$. However, we start by the following simple claim.
\begin{claim}
\label{claim: 1 element minimum}
Suppose $\ff_k, k\in [\ell]$ form an extremal example. Then, for any edge $(k_1, k_2)$ in the corresponding oriented graph $T_\ell$, it holds that
\begin{align*}
    \min_{x \in A_{k_1, k_2}} d_{k_1}(x) \le \frac{\mSet_{k_1, k_2}}{|A_{k_1, k_2}|} + O \left (n^{-3/2} \right ).
\end{align*}
\end{claim}

\begin{proof}
    Note that
    \begin{align*}
        d_{k_1 \to k_2}^{(\ge \mSet_{k_1, k_2})} & = \sum_{S \in \binom{\probabilisticCenter_{k_2}}{\mSet_{k_1, k_2}}} d_{k_2}(S) \ge \frac{1}{\mSet_{k_1, k_2}} \sum_{x \in \probabilisticCenter_{k_2}} d_{k_1}(x) -  d_{k_1 \to k_2}^{(< \mSet_{k_1, k_2})},
    \end{align*}
    where the last inequality follows from the double counting. Due to Lemma~\ref{lemma: direction of tentacles}, we have $d_{k \to k'}^{(\ge \mSet_{k, k'})} \le 1 - O(n^{-1/2})$, and, consequently, we have
    \begin{align}
    \label{eq: sum over probabilistic set, minimal degree claim}
        \sum_{x \in \probabilisticCenter_{k_2}} d_{k_1}(x) \le \mSet_{k_1, k_2} - O(n^{-1/2}).
    \end{align}
    If $x \in \probabilisticCenter_{k_2} \setminus A_{k_1, k_2}$, then there exists $k_3 \in In_{k_2}$ such that $d_{k_3}(x) = \Omega(n^{-1})$ due to Corollary~\ref{corollary: tentacles structure}. Due to Proposition~\ref{proposition: multiplicative property of degrees}, we have
    \begin{align*}
        d_{k_3}(x) d_{k_2}(x) d_{k_1}(x) = O(n^{-3}).
    \end{align*}
    Since $d_{k_1}(x) = \Omega(1)$ due to the definition of $\probabilisticCenter_k, k \in [\ell]$, we get $d_{k_2}(x) = O(n^{-2})$. Thus, the sum~\eqref{eq: sum over probabilistic set, minimal degree claim} can be decomposed as follows:
    \begin{align*}
        \mSet_{k_1, k_2} - O(n^{-1/2}) \le \sum_{x \in \probabilisticCenter_{k_2}} d_{k_1}(x) =\sum_{x \in A_{k_1, k_2}} d_{k_1}(x) + \sum_{x \in \probabilisticCenter_{k_2} \setminus A_{k_1, k_2}} d_{k_1}(x) = \sum_{x \in A_{k_1, k_2}} d_{k_1}(x) + O(n^{-1}).
    \end{align*}
    Hence, we obtain
    \begin{align*}
        \min_{x \in A_{k_1, k_2}} d_{k_1}(x) \le \frac{\mSet_{k_1, k_2}}{|A_{k_1, k_2}|} + O(n^{-3/2}),
    \end{align*}
    where we used the fact that $|A_{k_1, k_2}| = \Omega(n)$ due to Corollary~\ref{corollary: tentacles structure}. 
\end{proof}

Since we want to prove $\mathcal H_k \subset \ff_k$, the natural question is arising: how families behave outside $Supp_k$? We will provide the answer after the following auxiliary proposition that continues Claim~\ref{claim: inner set degree} in some sense.

\begin{proposition}
\label{proposition: inner degree with an element}
Let $\ff_k, k \in [\ell]$ form an extremal example and suppose that $(k_1, k_2)$ is an edge in the corresponding oriented graph $T_\ell$. Then a set
\begin{align*}
    I_T^{k_1, k_2} = \{x \in A_{k_1,k_2} \mid d_{k_1}(T \cup \{x\}) \le 2^{-|T| - 1} d_{k_1}(x) \}
\end{align*}
has cardinality $O(\sqrt{n \log n})$ uniformly over $T \subset \binom{\probabilisticCenter_{k_1}}{\le \max_{S} \mSet_S}$.
\end{proposition}

The proof is given in the appendix, Section~\ref{subsection: proof of proposition inner degree with an element}.

The next corollary explains how $\ff_k$ behaves outside $Supp_k$. 

\begin{corollary}
\label{corollary: outside degree}
Let $\ff_{k_0}$ be a family from the extremal example. Then for each $S \in \bigvee_{k' \in Out_{k_0}} \binom{A_{k_0, k'}}{\le \mSet_{k_0, k'}}$ we have
\begin{align*}
    \sum_{x \not \in Supp_{k_0}} d_{k_0}(S \cup \{x\}) = O \left (  n^{-|S| - 1/2} \log^{1/2} n \right).
\end{align*}
\end{corollary}

\begin{proof}

Consider $\ff_{k_0}$ and $S$ from the statement and define $S_{k} = S \cap \probabilisticCenter_{k}$, $k \in [\ell]$. Consider $x$ outside $Supp_{k_0}$ and suppose that $x \in A_{k', k''} \subset \probabilisticCenter_{k''}$ for some $k',k''$. From Proposition~\ref{proposition: multiplicative property of degrees} we get
\begin{align*}
    d_{k_0}(S \cup \{x\}) d_{k'}(S_{k'} \cup \{x\}) d_{k''}(S_{k''} \cup \{x\}) \prod_{k \in [\ell] \setminus \{k_0, k', k''\}} d_{k}(S_{k}) = O(n^{-|S|-3}).
\end{align*}
The element $x$ contributes to the exponent $3$ and the partition of $S$ contributes $|S|$ . Using Claim~\ref{claim: inner set degree}, we infer $d_{k_0}(S \cup \{x\}) d_{k'}(S_{k'} \cup \{x\}) = O(n^{-|S|-3})$. Consider $I^{k', k''}_{S_{k'}}$ from Proposition~\ref{proposition: inner degree with an element}. If $x \not \in I^{k', k''}_{S_{k'}}$, then $d_{k'}(S_{k'} \cup \{x\}) = \Omega(d_{k'}(x)) = \Omega(n^{-1})$  by the definition of $I^{k', k''}_{S_{k'}}$, and so $d_{k_0}(S \cup \{x\}) = O(n^{-|S|-2})$.

If $x \in I_{S_{k'}}^{k', k''}$, then
\begin{align*}
    d_{k_0}(S \cup \{x\}) d_{k'}(S_{k'}) d_{k''}(S_{k''} \cup \{x\}) \prod_{k \in [\ell] \setminus \{k_0, k', k''\}} d_{k}(S_k) = O(n^{-|S|-1}),
\end{align*}
and, consequently, $d_{k_0}(S \cup \{x\}) = O(n^{-|S|-1})$. Define $I = \bigsqcup_{k' \in [\ell] \setminus \{k\}} \bigsqcup_{k'' \in Out_k'} I_{S_{k'}}^{k', k''}$. Note that $|I| = O(\sqrt{n \log n})$ by Proposition~\ref{proposition: inner degree with an element}. We get
\begin{align*}
    \sum_{x \in [n] \setminus Supp_k} d_{k_0}(S\cup \{x\}) \le \sum_{x \in [n] \setminus (Supp_k \cup I)} d_{k_0}(S \cup \{x\}) + \sum_{x \in I} d_{k_0}(S \cup \{x\}) = O \left (  n^{-|S| - 1/2} \log^{1/2} n \right).
\end{align*}
\end{proof}

For a reminder, we expect that for families from an extremal example it holds that
\begin{align*}
    \frac{|\ff_k(S)|}{|\ff_k|} \approx \frac{|\hh_k(S)|}{|\hh_k|}
\end{align*}
for any tentacle $S$. Corollary~\ref{corollary: outside degree} implies that for any $S$ we have
\begin{align*}
    \frac{|\ff_k(S)|}{|\ff_k|}
    \le 
    \frac{|\ff_k |_{Supp_k}(S)|}{|\ff_k|} + O \left ( n^{- |S| - 1/2} \sqrt{\log n}\right ).
\end{align*}
Since $|\hh_k| = |\ff_k| (1 - O(n^{-1}))$, we have
\begin{align*}
    \frac{|\ff_k(S)|}{|\ff_k|} \approx \frac{|\ff_k|_{Supp_k}(S)|}{|\hh_k|}.
\end{align*}
In the next section, we prove that $|\ff_k|_{Supp_k}(S)| \approx |\hh_k(S)|$ for any tecntacle $S$. The key ingredients of the proof will be induction on $|S|$ and the symmetrization.

\subsection{Asymptotics of normalized degrees}
\label{subsection: asymptotics of normalizaed degrees}

In this section, we aim to prove that
\begin{align*}
    \frac{|\ff_k(S)|}{|\ff_k|} \approx \frac{|\hh_k(S)|}{|\hh_k|}
\end{align*}
indeed holds for any tentacle $S$. Before we move on to the key lemma of this section, we need one auxiliary proposition. As it was mentioned above, we will use the symmetrization argument. Thus, we will need to check conditions of Lemma~\ref{lemma: symmetrization argument}. Let $S_0$ be from the statement of the lemma. Coniditions of the lemma require to bound normalized degrees of $S_0$. Unfortunately, bounds provided by Proposition~\ref{proposition: multiplicative property of degrees} turn out to be too rough in some cases. To overcome this issue, we impose the following proposition.

\begin{proposition}
\label{proposition: child  can not overlap}
Let $\ff_k$, $k \in [\ell]$ be the families from the extremal examples. Suppose that there are three families $k_0, k_1, k_2 \in [\ell]$ such that $(k_0, k_1), (k_1, k_2), (k_0, k_2) \in E(T_\ell)$. If $$\left |d_{k_0}(x_2) -\frac{\mSet_{k_0, k_2}}{|A_{k_0, k_2}|} \right | \le C_1 n^{-3/2} \log^{1/2} n ,$$ for some constant $C_1$ then $d_{k_1}(x_2) \le e^{-C_2 n }$ for some other constant $C_2$ that depends on $C_1$.
\end{proposition}

\begin{wrapfigure}[15]{l}{0.5\textwidth}
    \includegraphics[width=0.45\textwidth]{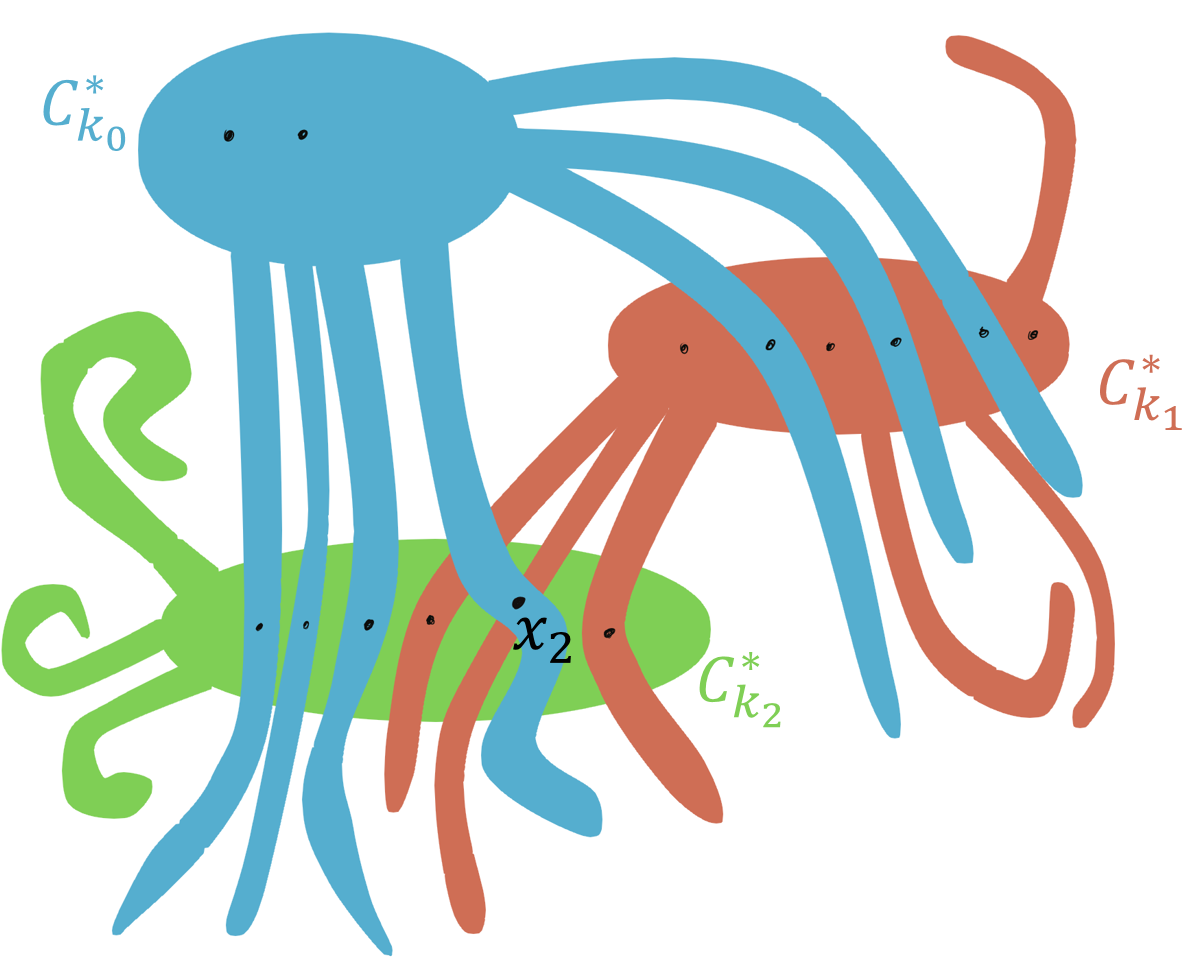}
    \caption{Illustration for Proposition~\ref{proposition: child  can not overlap}.}
\end{wrapfigure}

The proof is similar to that of Lemma~\ref{lemma: direction of tentacles} so we give it in the appendix, Section~\ref{subsection: proposition child  can not overlap}.

The remaining part of this section is dedicated to proof of the following fact: given $k \in [\ell]$ and  $S \in \bigvee_{k' \in Out_k} \binom{A_{k, k'}}{\le \mSet_{k, k'}}$, $d_k(S)$ has order $n^{-|S|}$. The proof is organized as multidimensional induction over vectors $\sVector_k = (|S \cap A_{k, k'}|)_{k' \in Out_k}$. We define a family $\sDomain_{\sVector, k} = \bigvee_{k' \in Out_k} \binom{A_{k, k'}}{\sVector_{k, k'}}$. We say $\sVector'_k \le \sVector_k$ if the coordinate-wise inequalities hold, and $\sVector'_k < \sVector_k$ if and only if $\sVector'_k \le \sVector_k$ and $\sVector'_k \neq \sVector_{k}$.

The induction first investigates all vectors $\sVector_k, k \in [\ell]$ with one non-zero entry. Then it deals with vectors $\sVector_k$ independently for different $k$. We formalize this procedure as follows:
\newline

\textbf{Induction hypothesis.} \textit{
    Let $\sVector_k$ be a vector defined above. 
    \begin{enumerate}
        \item If $\sVector_k$ has exactly one non-zero coordinate $\sVector_{k, k'}$, then for any $\hat k$ and $\tilde k \in Out_{\hat k}$, $s < \sVector_{k, k'}$ and $S \in \binom{A_{\hat k, \tilde k}}{s}$ it holds that
    \begin{align*}
        \left | d_{\hat k}(S) - \binom{\mSet_{\hat k, \tilde k}}{s} / \binom{|A_{\hat k, \tilde k}|}{s} \right | \le O \left ( \sqrt{\frac{\log n}{n}}\right ) n^{-s}.
    \end{align*}
        \item Otherwise, if $\sVector_k$ has more than one non-zero entry, then the previous statement holds for any $\hat k \in [\ell]$ and any possible $\sVector_{\hat k}$ with one non-zero position, and, in addition,
    \begin{align*}
        \left | d_k(S) - \prod_{k' \in Out_k} \binom{\mSet_{k, k'}}{\sVector_{k, k'}'} / \binom{|A_{k, k'}|}{\sVector_{k, k'}'} \right |
        \le O \left ( \sqrt{\frac{\log n}{n}}\right ) n^{-|S|}
    \end{align*}
    holds for any $\sVector_k' < \sVector_k$ and $S \in \sDomain_{\sVector', k}$.
    \end{enumerate} 
}

To illustrate how the proof works, consider $S \in \sDomain_{\sVector, k}$. The goal is to show $d_k(S) \approx \frac{|\hh_k(S)|}{|\hh_k|}$. Then, the steps of the proof are as follows:
\begin{itemize}
    \item using the induction hypothesis, show that 
    \begin{align*}
        \frac{1}{|\ff_k|} \sum_{\tilde{S} \in \sDomain_{\sVector, k}} \left | \ff_k(\tilde{S}) \cap \mathcal{G} \right | \ge \prod_{k' \in Out_{k}} \binom{\mSet_{k, k'}}{\sVector_{k, k'}} - O \left ( \sqrt{\frac{\log n}{n}}\right ), 
    \end{align*}
    where $\mathcal{G}$ is a family of some ``good'' sets;
    \item using the symmetrization argument, i.e. Lemma~\ref{lemma: symmetrization argument}, show that for $S_0 = \operatorname{argmin}_{\tilde{S} \in \sDomain_{\sVector, k}} d_k(\tilde{S})$, we have
    \begin{align*}
        d_k(S_0) \approx \frac{1}{|\sDomain_{\sVector, k}| |\ff_k|} \sum_{\tilde{S} \in \sDomain_{\sVector, k}} \left | \ff_k(\tilde{S}) \cap \mathcal{G} \right | \approx \frac{1}{|\sDomain_{\sVector, k}|} \sum_{\tilde{S} \in \sDomain_{\sVector, k}} d_k(\tilde{S});
    \end{align*}
    \item using the definition of $S_0$, establish the lower bound on $d_k(S)$: 
    \begin{align*}
        d_k(S) \ge \frac{1 - o(1)}{|\sDomain_{\sVector, k}|}
        \sum_{\tilde{S} \in \sDomain_{\sVector, k}} d_k(\tilde{S}) \ge \frac{1 - o(1)}{|\sDomain_{\sVector, k}|} \prod_{k' \in Out_{k}} \binom{\mSet_{k, k'}}{\sVector_{k, k'}};
    \end{align*}
    \item using Corollary~\ref{corollary: outside degree} and the fact that $\ff_k|_{Supp_k} \subset \hh_k$, obtain the upper bound on $d_k(S)$.
\end{itemize}

As we clarify the proof strategy, we are ready to state the result.

\begin{lemma}
\label{lemma: exact order}
Let $\ff_{k}$ be an arbitrary family from the extremal example and $S$ be a subset of $\bigcup_{k' \in Out_{k}} A_{k, k'}$ such that $\sVector_{k, k'} := |S \cap A_{k, k'}| \le \mSet_{k, k'} $. Then there is a constant $C$, depending on $\mSet$ only, such that
\begin{align*}
    \left |d_{k}(S) - \prod_{k' \in Out_{k}} \binom{\mSet_{k, k'}}{\sVector_{k, k'}} / \binom{|A_{k, k'}|}{\sVector_{k, k'}} \right |
    \le C \sqrt{\frac{\log n}{n}} n^{-|S|}.
\end{align*}
\end{lemma}
\begin{proof}
First, we obtain the lower bound on $d_{k}(S)$. Define $S_0 = \arg \min_{S \in \sDomain_{\sVector, k}} d_{k}(S)$. Choose $k'$ such that $\sVector_{k, k'} > 0$. Define $\sVector'_{k} < \sVector_{k}$ as follows
\begin{align*}
    \sVector'_{k, k''} =
    \begin{cases}
        \sVector_{k, k''}, & \text{if } k'' \neq k', \\
        \sVector_{k, k''} - 1, & \text{if } k'' = k'.
    \end{cases}
\end{align*}
Then, by double counting
\begin{align}
\label{eq: exact order lemma sum to reduce}
    \frac{1}{|\ff_k|} \sum_{\tilde{S} \in \sDomain_{\sVector, k}}
        \left |
            \ff_k(\excludeSet{\tilde{S}}) \cap \bigcap_{S' \subsetneq S_0} \ff_k(\excludeSet{S'})
        \right |
    =
    \frac{\sVector_{k, k'}^{-1}}{|\ff_k|}
    \sum_{\tilde{S} \in \sDomain_{\sVector', k}}
    \sum_{x \in A_{k, k'} \setminus \tilde{S}}
        \left |
            \ff_k(\excludeSet{\tilde{S}} \cup \excludeSet{\{x\}})
            \cap
            \bigcap_{S' \subsetneq S_0} \ff_k(\excludeSet{S'})
        \right |.
\end{align}
Due to down-closeness, we have
\begin{align*}
    \left |
        \ff_k(\excludeSet{\tilde{S}} \cup \excludeSet{\{x\}})
        \cap
        \bigcap_{S' \subsetneq S_0} \ff_k(\excludeSet{S'})
    \right |
    \ge
    \left |
        \ff_k(\excludeSet{\tilde{S}} \cup \saveSet{\{x\}})
        \cap
        \bigcap_{S' \subsetneq S_0} \ff_k(\excludeSet{S'})
    \right |.
\end{align*}
Since $d_{k \to k'}^{(\ge \mSet_{k, k'})} = 1 - O(n^{-1})$ from Lemma~\ref{lemma: direction of tentacles}, we may consider only sets that intersect $\probabilisticCenter_{k'}$ by exactly $\mSet_{k, k'}$ elements to bound sum~\eqref{eq: exact order lemma sum to reduce} from below. A set $T \subset A_{k, k'}$ of size $\mSet_{k, k'}$ such that $\tilde{S} \subset T$ is counted $\mSet_{k, k'} - \sVector_{k, k'}'$ times in the sum
\begin{align*}
     \sum_{x \in A_{k, k'} \setminus \tilde{S}}
        & \left |
            \ff_k(\excludeSet{\tilde{S}} \cup \saveSet{\{x\}})
            \cap
            \bigcap_{S' \subsetneq S_0} \ff_k(\excludeSet{S'})
        \right |.
\end{align*}
Thus,
\begin{align}
    \frac{1}{|\ff_k|} \sum_{\tilde{S} \in \sDomain_{\sVector', k}} 
    \sum_{x \in A_{k, k'} \setminus \tilde{S}}
        & \left |
            \ff_k(\excludeSet{\tilde{S}} \cup \saveSet{\{x\}})
            \cap
            \bigcap_{S' \subsetneq S_0} \ff_k(\excludeSet{S'})
        \right |
    \ge \nonumber \\
    & \ge \frac{\mSet_{k, k'} - \sVector_{k, k'}'}{|\ff_k|}
    \sum_{\tilde{S} \in \sDomain_{\sVector', k}}
        \left |
            \ff_k(\excludeSet{\tilde{S}})
            \cap
            \bigcap_{S' \subsetneq S_0} \ff_k(\excludeSet{S'})
        \right | - O(n^{-1}). \label{eq: reduction sum whole exact order lemma}
\end{align}
If $\sVector'_k$ is the zero vector, we already proved the following: $\sum_{x \in A_{k, k'}} d_{k}(x) \ge \mSet_{k, k'} - O(n^{-1})$.

Otherwise, we use the induction hypothesis. We may apply Corollary~\ref{corollary: outside degree} and show that for any $\tilde{S} \in \sDomain_{\sVector', k}$:
\begin{align*}
    & \prod_{k' \in Out_{k}} \binom{\mSet_{k, k'}}{\sVector_{k, k'}'} / \binom{|A_{k, k'}|}{\sVector_{k, k'}'} - O \left (\sqrt{\frac{\log n}{n^{1 + 2 |\tilde{S}|}}} \right )
    \overset{\text{I.H.}}{\le} 
    d_k(\tilde{S}) 
    \overset{\text{Cor.~\ref{corollary: outside degree}}}{\le} 
    \frac{|\ff_k(\tilde{S})|_{Supp_k}|}{|\ff_k|} + O \left (\sqrt{\frac{\log n}{n^{1 + 2 |\tilde{S}|}}} \right ) \\
    & \overset{\text{Claim~\ref{claim: hh_k approximation quality}}}{\le} (1 + O(n^{-1})) \frac{|\ff_k(\tilde{S})|_{Supp_k}|}{|\hh_k|} + O \left (\sqrt{\frac{\log n}{n^{1 + 2 |\tilde{S}|}}} \right )
    \le (1 + O(n^{-1})) \frac{|\hh_k(\tilde{S})|}{|\hh_k|} + O \left (\sqrt{\frac{\log n}{n^{1 + 2 |\tilde{S}|}}} \right )\\
    & \overset{\substack{\text{by the definition~\eqref{eq: definition of hh_k}} \\ \text{of } \hh_k}}{\le} \prod_{k' \in Out_{k}} \binom{\mSet_{k, k'}}{\sVector_{k, k'}'} / \binom{|A_{k, k'}|}{\sVector_{k, k'}'} + O \left (\sqrt{\frac{\log n}{n^{1 + 2 |\tilde{S}|}}} \right ).
\end{align*}
Given that $\ff_k|_{Supp_k} \subset \hh_k$ the claim of the inequality above also implies $|\hh_k(\tilde{S}) \setminus \ff_k(\tilde{S})|/|\hh_k| = O \left (n^{-|\tilde{S}| - 1/2} \log^{1/2} n \right )$. Hence,
\begin{align*}
    & \frac{1}{|\ff_k|}
    \sum_{\tilde{S} \in \sDomain_{\sVector', k}}
        \left |
            \ff_k(\excludeSet{\tilde{S}})
            \cap
            \bigcap_{S' \subsetneq S_0} \ff_k(\excludeSet{S'})
        \right |
    \ge
    \frac{1}{|\ff_k|}
    \sum_{\tilde{S} \in \sDomain_{\sVector', k}}
        \left |
            \ff_k(\excludeSet{\tilde{S}})|_{Supp_k}
            \cap
            \bigcap_{S' \subsetneq S_0} \ff_k(\excludeSet{S'})|_{Supp_k}
        \right | \\
    & \ge
    \frac{1}{|\hh_k|}
    \sum_{\tilde{S} \in \sDomain_{\sVector', k}}
        \left [\left |
            \hh_k(\excludeSet{\tilde{S}})
            \cap
            \bigcap_{S' \subsetneq S_0} \hh_k(\excludeSet{S'})
        \right |
    - O \left (n^{-|\tilde{S}| - 1/2} \log^{1/2} n \right )\right ] \\
    & \ge \prod_{k' \in Out_{k}} \binom{\mSet_{k, k'}}{\sVector_{k, k'}'}  - O \left (\sqrt{\frac{\log n}{n}} \right ). 
\end{align*}
Combining the above and inequalities~\eqref{eq: exact order lemma sum to reduce}, \eqref{eq: reduction sum whole exact order lemma}, we obtain
\begin{align}
    \frac{1}{|\ff_k|} \sum_{\tilde{S} \in \sDomain_{\sVector, k}}
        \left |
            \ff_k(\excludeSet{\tilde{S}}) \cap \bigcap_{S' \subsetneq S_0} \ff_k(\excludeSet{S'})
        \right |
    & \ge
    \frac{\mSet_{k, k'} - \sVector_{k, k'}'}{\sVector_{k, k'}} \prod_{k' \in Out_{k}} \binom{\mSet_{k, k'}}{\sVector_{k, k'}'} - O \left (\sqrt{\frac{\log n}{n}} \right ) \nonumber \\
    & \ge
    \prod_{k' \in Out_k} \binom{\mSet_{k, k'}}{\sVector_{k, k'}} - O \left (\sqrt{\frac{\log n}{n}} \right ). \label{eq: exact order lower bound}
\end{align}
Now apply Lemma~\ref{lemma: symmetrization argument} for $\sDomain_{\sVector, k}$, $S_0 = \arg \min_{S \in \sDomain_{\sVector, k}} d_k(S)$ and $\Delta = O \left (\sqrt{\frac{\log n}{n}} \right ) n^{-|S_0|}$. Check the assumptions of the lemma. 
\begin{itemize}
    \item If $\sVector_k$ has one non-zero entry $\sVector_{k, k'}$, then case~\ref{symmetrization lemma: case 1} holds. Let $\tilde{k} \in [\ell] \setminus \{k, k'\}$ be an index such that we want to show $d_{\tilde{k}}(S_0) \ll \Delta$. If $|S_0| = 1$ then $d_{\tilde{k}}(S_0) = O(n^{-2})$ according to Propositions~\ref{proposition: multiplicative property of degrees} and~\ref{corollary: tentacles structure}. Next, assume $|S_0| \ge 2$ and choose $x \in S_0$. Due to the induction hypothesis, $d_k(S_0 \setminus \{x\}) = \Omega(n^{-|S_0| + 1})$. Meanwhile,
    \begin{align*}
        d_k(S_0 \setminus \{x\}) d_{\tilde{k}}(S_0) d_{k'}(S_0) & = O(n^{-3|S_0| + 2}), \\
        d_{\tilde{k}}(S_0) = O(n^{-2 |S_0| + 1}) & = O(n^{-|S_0| - 1}) \ll \Delta.
    \end{align*}
    \item If $\sVector_k$ has more than one non-zero entry, then we are in case~\ref{symmetrization lemma: case 2} and $|S_0| \ge 2$. Let $\tilde{k} \in [\ell] \setminus \{k\}$ be an arbitrary index. If $\tilde{k} \in Out_k$, $d_{\tilde{k}}(S_0) = \exp(-\Omega(n))$ because of either Proposition~\ref{proposition: child  can not overlap} or Lemma~\ref{lemma: direction of tentacles}. Otherwise, choose some $x \in S_0$ and write down
    \begin{align*}
        d_{\tilde{k}}(S_0) d_k(S_0 \setminus \{x\}) \prod_{k' \in Out_k} d_{k'}(S_0 \cap \probabilisticCenter_{k'}) & = O(n^{- 3|S_0| + 2}),\\
        d_{\tilde{k}}(S_0) = O(n^{- 2 |S_0| + 1})& = O(n^{-|S_0| - 1}) \ll \Delta.
    \end{align*}
    Here we use the induction hypothesis and Proposition~\ref{proposition: multiplicative property of degrees}.
\end{itemize}
Thus, the assumptions of Lemma~\ref{lemma: symmetrization argument} indeed hold. For a reminder, the conclusion is that
\begin{align*}
    \mathcal{I}_\Delta = \left \{ \tilde{S} \in \sDomain_{\sVector, k} \mid \frac{1}{|\ff_{k}|} \left | \ff_{k}(\excludeSet{\tilde{S}}) \cap \bigcap_{S' \subsetneq S_0} \ff_{k}(\excludeSet{S'}) \right | \ge d_k(S_0) + \Delta \right \}
\end{align*}
has cardinality $O \left ( \frac{1}{\Delta} \frac{\log n}{n} \right ) = O \left ( n^{|S_0|} \sqrt{\frac{\log n}{n}}\right )$. Decomposing the sum in~\eqref{eq: exact order lower bound}, we obtain
\begin{align}
\label{eq: final sum exact order}
    \frac{1}{|\ff_k|} \sum_{\tilde{S} \in \mathcal I_\Delta}
    \left |
        \ff_k(\tilde{S}) \cap \bigcap_{S' \subsetneq S_0} \ff_k(S')
    \right |
    + (|\sDomain_{\sVector, k}| - |\mathcal I_\Delta|) (d_k(S_0) + \Delta) \ge \prod_{k' \in Out_k} \binom{\mSet_{k, k'}}{\sVector_{k, k'}} - O \left (\sqrt{\frac{\log n}{n}} \right ). 
\end{align}
From Proposition~\ref{proposition: multiplicative property of degrees}, we have
\begin{align*}
    d_{k}(\tilde{S}) \prod_{k' \in Out_k} d_{k'}(\tilde{S} \cap \probabilisticCenter_{k'}) & = O(n^{-|\tilde{S}|}).
\end{align*}
Claim~\ref{claim: inner set degree} guarantees that $d_{k'}(\tilde{S} \cap \probabilisticCenter_{k'}) = \Omega(1)$ for $k ' \in Out_{k}$. Thus  $d_{k}(\tilde{S}) =  O(n^{-|\tilde{S}|})$ and
\begin{align*}
     \frac{1}{|\ff_k|} \sum_{\tilde{S} \in \mathcal I_\Delta}
    \left |
        \ff_k(\tilde{S}) \cap \bigcap_{S' \subsetneq S_0} \ff_k(S')
    \right | 
    \le |\mathcal{I}_\Delta| O(n^{-|S_0|}) = O \left (\sqrt{\frac{\log n}{n}} \right ).
\end{align*}
Consequently, the left-hand side of~\eqref{eq: final sum exact order} is at most
\begin{align*}
    O \left (\sqrt{\frac{\log n}{n}} \right ) + \left  (|\sDomain_{\sVector, k}| - O(n^{|S_0| - 1/2} \log^{1/2} n) \right ) (d_k(S_0) + O(n^{-|S_0| - 1/2} \log^{1/2} n))  \\
    = d_{k}(S_0) |\sDomain_{\sVector, k}| + O \left (\sqrt{\frac{\log n}{n}} \right ) = \prod_{k' \in Out_k} \binom{|A_{k, k'}|}{\sVector_{k, k'}} d_{k}(S_0) + O \left (\sqrt{\frac{\log n}{n}} \right ).
\end{align*}
Whence, due to~\eqref{eq: final sum exact order}, we have
\begin{align*}
    \min_{S \in \sDomain_{\sVector, k}} d_k(S) = d_{k}(S_0) \ge \prod_{k' \in Out_k} \binom{\mSet_{k, k'}}{\sVector_{k, k'}} /\binom{|A_{k, k'}|}{\sVector_{k, k'}} - O \left (n^{-|S_0|-1/2} \log^{1/2} n\right ) .
\end{align*}
The upper bound is obtained using Corollary~\ref{corollary: outside degree}. Let $S \in \sDomain_{\sVector, k}$. Then
\begin{align*}
    d_k(S) & \le \frac{\left | \ff_k(S)|_{Supp_k} \right |}{|\ff_k|} + O \left ( \sqrt{\frac{\log n}{n}}\right ) n^{-|S|} \\ 
    & \le \frac{|\hh_k(S)|}{|\hh_k|} + O \left ( \sqrt{\frac{\log n}{n}}\right ) n^{-|S|} \le \prod_{k' \in Out_k} \binom{\mSet_{k, k'}}{\sVector_{k, k'}} /\binom{|A_{k, k'}|}{\sVector_{k, k'}} + O \left (\sqrt{\frac{\log n}{n}} \right )n^{-|S|}.
\end{align*}
\end{proof}

\subsection{Proof}

\begin{proof}[Proof of the main theorem]
First, we obtain the asymptotics of $\targetFunction$. Form Proposition~\ref{proposition: subfamilies 1/n}, we have 
\begin{align*}
    |\ff_k |_{Supp_k}| \ge (1 - O(n^{-1})) |\ff_k|.
\end{align*}
Moreover, each $x$ lies in at most 2 families $\ff_k |_{Supp_k}, k \in [\ell]$. Thus, Lemma~\ref{lemma: extremal exmaple asymptotics} implies that
\begin{align*}
    \targetFunction = \left (1 + O \left ( n^{-1} \right ) \right ) \cdot
    2^n\cdot
    \prod_{S\in{[\ell]\choose 2}}
        \bigg(
            \frac 1{\mSet_S!}
            \Big(
                \frac{\mSet_S \cdot n}{
                    \sum_{S'\in {[\ell]\choose 2}}
                        \mSet_{S'}
                }
            \Big)^{\mSet_S}
        \bigg).
\end{align*}
To match this asymptotics, due to the definition~\ref{eq: definition of hh_k} of $\hh_k$, we must have
\begin{align*}
    |A_S| = (1 + O(n^{-1})) \cdot \frac{\mSet_S \cdot n}{\sum_{S' \in \binom{[\ell]}{2}}\mSet_{S'}}.
\end{align*}

Second, we obtain the structural part of Theorem~\ref{theorem: maximal families structure}. Define families
\begin{align*}
    \mathcal{P}_k = \bigvee_{k' \in Out_k} \binom{A_{k, k'}}{\le \mSet_{k,k'}}.
\end{align*}
Since we have already proved that $d_{k}(F) > 0$ for every $F \in \mathcal{P}_k$ for large $n$, $\mathcal{P}_k \subset \ff_k$, see Lemma~\ref{lemma: exact order}. The remaining part is to prove that $2^{\probabilisticCenter_k} \vee \mathcal{P}_k \subset \ff_k$. If $k$ is a source then due to Proposition~\ref{proposition: empty center of the source} its center is empty and the statement is already proved. Fix a $k_0$ which is not a source and define
\begin{align}
    \mathcal{E}_k & =  \left \{ e \in \ff_{k} \cap \binom{n}{\mSet_{k, k_0} + 1} \mid e \setminus \probabilisticCenter_{k_0} \in \mathcal{P}_{k_0} \right \}, \label{eq: final proof errors definition} \\
    \mathcal{E} & = \bigcup_{k \in [\ell] \setminus \{k_0\}} \mathcal{E}_k. \nonumber
\end{align}
The condition $e \setminus \probabilisticCenter_{k_0} \in \mathcal{P}_{k_0}$ means that $e \in \hh_{k_0}$. Since $\ff_{k_0}$ is an extremal family, Claim~\ref{claim: hypergraph point of view} implies that $F \in \hh_{k_0} \setminus \ff_{k_0}$ if and only if there exists some $e \in \mathcal{E}$ such that $e \subset F$. Thus, the condition $\mathcal{E} = \varnothing$ is sufficient for $\hh_{k_0} \subset \ff_{k_0}$. Note that for any $k \in [\ell] \setminus \{k_0\}$ we have
\begin{align*}
    \frac{1}{|\ff_k|} \left | \bigcup_{e \in \mathcal{E}_k} \ff_k(\saveSet{e}) \right |
    \le
    \sum_{e \in \mathcal{E}_k} d_k(e)
    \le C \sum_{e \in \mathcal{E}_k} n^{-|e|}
\end{align*}
for some constant $C$. The proof of the last inequality is straightforward and based on considering two cases: $k \in Out_{k_0}$ and $k \not \in Out_{k_0}$. Since $\mathcal{P}_{k_0} \subset \ff_{k_0}$, a set $e \in \mathcal{E}$ can not lie completely in $\mathcal P_{k_0}$, and so $e \cap \probabilisticCenter_{k_0} \neq \varnothing$. Thus, in the former case, we have that $d_k(e)$ is exponentially small due to Lemma~\ref{lemma: direction of tentacles}. The latter case follows from Proposition~\ref{proposition: multiplicative property of degrees}:
\begin{align*}
    d_k(e) d_{k_0}(e \cap \probabilisticCenter_{k_0}) \prod_{k' \in Out_{k_0}} d_{k'}(e \cap \probabilisticCenter_{k'}) \le C' n^{-|e|}
\end{align*}
for some constant $C'$. Since $d_{k_0}(e \cap \probabilisticCenter_{k_0})$ and $d_{k'}(e \cap \probabilisticCenter_{k'})$ are asymptotically constant due to Claim~\ref{claim: inner set degree}, we have $d_{k}(e) = O(n^{-|e|})$.

Define families
\begin{align}
\label{eq: final proof families modification}
    \ff_k' = \begin{cases}
        \ff_k \cup \hh_k, & k = k_0, \\
        \ff_k \setminus \bigcup_{e \in \mathcal{E}_k} \ff_k(\saveSet{e}), & \text{otherwise}.
    \end{cases}
\end{align}
Suppose that $\mathcal{E} \neq \varnothing$. In the proof, we show that modification~\eqref{eq: final proof families modification} can not enlarge the target product and so $\hh_{k_0} \subset \ff_{k_0}$. We have
\begin{align}
\label{eq: final proof modified families product}
    \prod_{k \in [\ell]} |\ff_k'| \ge (1 + \beta) \left (1 - O \left ( \sum_{e \in \mathcal{E}} n^{-|e|} \right ) \right )\prod_k |\ff_k|, 
\end{align}
where $\beta:= \frac{|\hh_{k_0} \setminus \ff_{k_0}|}{|\ff_{k_0}|}$. Thus, it is sufficient to show that
\begin{align}
\label{eq: beta is asymptotically large than degrees}
    \beta \gg C \sum_{e \in \mathcal{E}} n^{-|e|} \ge \sum_{k \in [\ell] \setminus \{k_0\}} \sum_{e \in \mathcal{E}_k} d_{k}(e).
\end{align}
Define $\mathcal{P}_{k_0}' = \bigvee_{k \in Out_{k_0}} \binom{A_{k, k_0}}{\mSet_{k, k_0}}$ the set of inclusion-maximal ``tentacles''. Then
\begin{align}
    \beta & = \frac{1}{|\ff_{k_0}|} \sum_{P \in \mathcal{P}_{k_0}} \left (
        2^{|\probabilisticCenter_{k_0}|} - \left | \ff_{k_0}|_{Supp_{k_0}} \left (P, \bigcup_{k \in Out_{k_0}} A_{k_0, k}\right ) \right |
    \right ) \nonumber \\
    & \overset{\text{Claim~\ref{claim: hh_k approximation quality}}}{\ge} \frac{1 - O(n^{-1})}{|\hh_{k_0}|} \sum_{P \in \mathcal{P}_{k_0}} \left (
        2^{|\probabilisticCenter_{k_0}|} - \left | \ff_{k_0}|_{Supp_{k_0}} \left (P, \bigcup_{k \in Out_{k_0}} A_{k_0, k}\right ) \right |
    \right ) \nonumber \\
    & \overset{\substack{\text{definition~\eqref{eq: definition of hh_k}} \\ \text{of } \hh_k }}{\ge} \Theta \left (n^{-\sum_{k \in Out_{k_0}} \mSet_{k, k_0}} \right ) \sum_{P \in \mathcal{P}_{k_0}'}
    \left ( 1 - \frac{|\ff_{k_0}|_{Supp_{k_0}} \left (P \right )|}{2^{|\probabilisticCenter_{k_0}|}}\right ). \label{eq: final proof beta lower bound 1}
\end{align}
Consider $m = \max_{k, k'} \mSet_{k, k'}$ and define
\begin{align*}
    \mathcal{E}' = \left \{
        F \in \binom{\probabilisticCenter_{k_0}}{m} \vee \mathcal{P}'_{k_0}
        \mid
        \exists e \in \mathcal{E} \text{ s.t. } e \subset F 
    \right \}
\end{align*}
If for some $P \in \mathcal{P}_{k_0}'$ we have $\mathcal{E}'(P) \neq \varnothing$, then there exists $T \in \binom{\probabilisticCenter_{k_0}}{m}$ such that $P \cup T \in \mathcal{E}'$. Thus, any $F \in \ff_{k_0}|_{Supp_{k_0}} (P)$ can not contain $T$. Consequently,
\begin{align*}
    \left |\ff_{k_0} |_{Supp_{k_0}}(P) \right | \le 2^{|\probabilisticCenter_{k_0}|} - 2^{|\probabilisticCenter_{k_0}| - m}, \\
\end{align*}
Rearranging terms, we obtain
\begin{align*}
    1 - \frac{|\ff_{k_0}|_{Supp_{k_0}} \left (P \right )|}{2^{|\probabilisticCenter_{k_0}|}} \ge 2^{-m}.
\end{align*}
Thus,
\begin{align*}
    1 - \frac{|\ff_{k_0}|_{Supp_{k_0}} \left (P \right )|}{2^{|\probabilisticCenter_{k_0}|}} \ge 2^{-m} \indicator \left \{ \mathcal{E}'(P) \neq \varnothing \right \}.
\end{align*}
Combining the above with~\eqref{eq: final proof beta lower bound 1}, we get
\begin{align}
\label{eq: beta lower bound from P' final proof}
    \beta & = \Omega \left ( n^{- \sum_{k \in Out_{k_0}} \mSet_{k, k_0}} \right ) \sum_{P \in \mathcal P'_{k_0}} \indicator \left \{ \mathcal{E}'(P) \neq \varnothing \right \} \nonumber \\
    & \ge c \sum_{P \in \mathcal P'_{k_0}} \indicator \left \{ \mathcal{E}'(P) \neq \varnothing \right \} n^{- |P|}
\end{align}
for some constant $c$.

Family $\mathcal{E}$ is of uniformity at most $m + \sum_{k \in Out_{k_0}} \mSet_{k, k_0}$. Thus, we have
\begin{align}
\label{eq: E decomposition into layers}
    \sum_{e \in \mathcal{E}} n^{-|e|} = \sum_{t = 1}^{m + \sum_{k \in Out_{k_0}} \mSet_{k, k_0}} \sum_{e \in \mathcal E^{(t)}} n^{-|e|} = \sum_{t = 1}^{m + \sum_{k \in Out_{k_0}} \mSet_{k, k_0}} n^{-t} |\mathcal{E}^{(t)}|,
\end{align}
where $\mathcal{E}^{(t)} = \mathcal E \cap \binom{[n]}{t}$. For each set $F \in \mathcal E^{(t)}$ there are $\Omega \left (n^{m + \sum_{k \in Out_{k_0}} \mSet_{k, k_0} - t} \right )$ sets $F'$ in $\mathcal E'$ such that $F \subset F'$, and, conversely, for each $F' \in \mathcal{E}'$ there exist at most $\binom{m + \sum_{k \in Out_{k_0}} \mSet_{k, k_0}}{t}$ sets $F \in \mathcal{E}^{(t)}$ such that $F \subset F'$. Thus, we have
\begin{align*}
    \frac{|\mathcal{E}^{(t)}|}{\binom{m + \sum_{k \in Out_{k_0}} \mSet_{k, k_0}}{t}} & \le C |\mathcal{E}'| n^{- \left \{m + \sum_{k \in Out_{k_0}} \mSet_{k,k_0} - t \right \}}, \\
    n^{-t} |\mathcal{E}^{(t)}| & \le C \sum_{e \in \mathcal{E}'} n^{-|e|}
\end{align*}
for some constant $C$ that depends on $\mSet, t$. Combining the above with~\eqref{eq: E decomposition into layers}, we obtain
\begin{align*}
    \sum_{e \in \mathcal{E}} n^{-|e|} = O(1) \cdot \sum_{e \in \mathcal{E}'} n^{-|e|}.
\end{align*}
Hence, to satisfy~\eqref{eq: beta is asymptotically large than degrees}, it is enough to show that~\eqref{eq: beta lower bound from P' final proof} is asymptotically larger than $\sum_{e \in \mathcal E'} n^{-|e|}$:
\begin{align}
\label{eq: indicator inequality final proof}
    \sum_{P \in \mathcal P'} \indicator \{\mathcal{E}'(P) \neq \varnothing \} n^{-|P|} \gg \sum_{e \in \mathcal{E}} n^{-|e|} = \sum_{P \in \mathcal P'} \left \{ n^{-m} |\mathcal E'(P)| \right \} n^{-|P|}.
\end{align}
Define $[\mathcal{E}'(P)]^{\uparrow} = \{F \in 2^{\probabilisticCenter_{k_0}} \mid \exists e \in \mathcal{E}'(P) \text{ s.t. } e \subset F\}$ the upper closure of $\mathcal E'(P)$ in $2^{\probabilisticCenter_{k_0}}$. Then 
\begin{align}
\label{eq: upper-closure bound}
    |[\mathcal{E}'(P)]^{\uparrow}| \le 2^{|\probabilisticCenter_{k_0}|} - \left |\ff_{k_0}|_{Supp_{k_0}}(P) \right |
\end{align}
since each $F \in [\mathcal{E}'(P)]^{\uparrow}$ can not be contained in $\ff_{k_0}|_{Supp_{k_0}}(P) \subset 2^{\probabilisticCenter_{k_0}}$
due to Lemma~\ref{lemma: exact order}. From Corollary~\ref{corollary: outside degree}, we have
\begin{align*}
    \frac{\left |\ff_{k_0}|_{Supp_{k_0}}(P) \right |}{|\ff_{k_0}|} \ge d_{k_0}(P) - O \left ( n^{-|P| - 1/2} \log^{1/2} n \right ).
\end{align*}
Bounding $d_{k_0}(P)$ from below via Lemma~\ref{lemma: exact order}, we obtain
\begin{align*}
     \frac{\left |\ff_{k_0}|_{Supp_{k_0}}(P) \right |}{|\ff_{k_0}|} \ge \prod_{k \in Out_{k, k_0}} \left [ \binom{|A_{k, k_0}|}{\mSet_{k, k_0}} \right ]^{-1} - O \left ( n^{-|P| - 1/2} \log^{1/2} n \right ).
\end{align*}
Due to Claim~\ref{claim: hh_k approximation quality}, $|\ff_{k_0}| = |\hh_{k_0}| (1 + O(n^{-1})) = (1 + O(n^{-1})) 2^{|\probabilisticCenter_{k_0}|} \prod_{k \in Out_{k_0}} \binom{|A_{k, k_0}|}{\mSet_{k, k_0}}$ and so
\begin{align*}
    \left |\ff_{k_0}|_{Supp_{k_0}}(P) \right | \ge 2^{\probabilisticCenter_{k_0}} \left (1 - O \left ( \sqrt{\frac{\log n}{n}}\right )\right ).
\end{align*}
From the above and~\eqref{eq: upper-closure bound}, it follows that
\begin{align}
\label{eq: upper-closure pure bound}
    |[\mathcal{E}'(P)]^{\uparrow}| \le O \left ( \sqrt{\frac{\log n}{n}}\right ) 2^{|\probabilisticCenter_{k_0}|}.
\end{align}
Consider maximal integer $s$ such that $|\mathcal{E}'(P)| \ge \sum_{i = 0}^{m - 2} \binom{|\probabilisticCenter_{k_0}|}{i} + s \binom{|\probabilisticCenter_{k_0}|}{m - 1}$. If $\left | \mathcal{E}'(P) \right | \ge  \sum_{i = 0}^{m - 2} \binom{|\probabilisticCenter_{k_0}|}{i} + s \binom{|\probabilisticCenter_{k_0}|}{m - 1}$, then the initial segment $\mathcal{B}(|\mathcal E'(P)|)$ in lexicographical order contains all sets starting with $i \in [s]$. Thus, Theorem~\ref{theorem: Kruskal-Katona theorem} implies
\begin{align*}
    |[\mathcal{E}'(P)]^{\uparrow}| \ge 2^{|\probabilisticCenter_{k_0}|} - 2^{|\probabilisticCenter_{k_0}| - s} - \binom{|\probabilisticCenter_{k_0}|}{\le m - 1}.
\end{align*}
Combined with inequality~\eqref{eq: upper-closure pure bound}, it implies $s = 0$ for large enough $n$. Thus, $|\mathcal{E}'(P)| n^{-m} = O(n^{-1})$ and~\eqref{eq: indicator inequality final proof} holds. It leads to~\eqref{eq: beta is asymptotically large than degrees} as it was discussed. Thus, for a families $\ff'_k$ defined in~\eqref{eq: final proof families modification}, we have $\prod_{k \in [\ell]} |\ff'_k| > \prod_{k \in [\ell]} |\ff_k|$ due to~\eqref{eq: final proof modified families product}, the contradiction.
\end{proof}

\section{Sketch of the proof of Theorem~\ref{theorem: extremal graph coloring}}\label{sec5}

    We start by determining the tournament $T_\ell$. For each $k \in [\ell]$, let $\ff_k$ be an extremal family and
    \begin{align*}
        \hh_k = 2^{\probabilisticCenter_k} \vee \bigvee_{k' \in Out_k} \binom{A_{k,k'}}{\le 1}.
    \end{align*}
    We are going to determine the most of the sets of $\ff_k \setminus \hh_k$. Define
    \begin{align*}
        \mathcal{E}_k^{exp} & = \left \{F \in \ff_k \mid F \cap A_{S} \neq \varnothing \text{ for some } S \not \in \binom{Out_k}{2} \right \}.\\
    \end{align*}
    Consider $F \in \mathcal{E}^{exp}_k$ and assume that  $x \in F \cap A_{S}$ for some particular $S \not \in \binom{Out_k}{2}$. Then, there is $s \in S$, such that $k \in Out_s$. If $x \in \probabilisticCenter_s$, then $d_k(x) = e^{-\Omega(n)}$ due to Lemma~\ref{lemma: direction of tentacles}. If $x \in \probabilisticCenter_{s'}$ for $s' \in Out_k$, we have $d_k(x) = e^{-\Omega(n)}$ due to Proposition~\ref{proposition: child  can not overlap}. Consequently, we have $|\mathcal{E}_k^{exp}| \le e^{-\Omega(n)}|\ff_k|$.

    Next, define
    \begin{align*}
        \mathcal{E}_k^{sq} & = \left \{ F \in \ff_k \mid |F \setminus \support \hh_k| \ge 2 \right \}{\setminus \mathcal E^{exp}_k}.
    \end{align*}
    Consider $F \in \mathcal{E}_k^{sq}$. Let $A_{s_1, s_2}, A_{s_3, s_4}$, $s_i \in Out_k$, be such that $x \in F \cap A_{s_1, s_2}, y \in F \cap A_{s_3, s_4}$. All $s_1, s_2, s_3, s_4$ must be distinct because otherwise $|F \cap F'| \ge 2$ for some $s_i$ and $F' \in \hh_{s_i}$. Assume $A_{s_1, s_2} \subset \probabilisticCenter_{s_1}, A_{s_3, s_4} \subset \probabilisticCenter_{s_3}$. Due to Proposition~\ref{proposition: multiplicative property of degrees}, we have
    \begin{align*}
        d_{s_1}(x) d_{s_2}(x) d_{s_3}(y) d_{s_4}(y) d_{k}(x, y) = O(n^{-6}).
    \end{align*}
    Due to Claim~\ref{claim: inner set degree}, $d_{s_1}(x), d_{s_3}(y) \ge \frac{1}{2} - O(n^{-1})$. Due to Lemma~\ref{lemma: exact order}, $d_{s_2}(x), d_{s_4}(y) = \Omega(n^{-1})$. Consequently, we have
    \begin{align}
    \label{eq: intersecting for families outside tentacle}
        d_k(x,y) = O(n^{-4}).
    \end{align}
    Hence, we obtain
    \begin{align*}
        \frac{|\mathcal{E}^{sq}_k|}{|\ff_k|} \le \frac{1}{2} \sum_{S \in \binom{Out_k}{2}} \sum_{S' \in \binom{Out_k \setminus S}{2}} \sum_{x \in A_S} \sum_{y \in S'} d_k(x, y) = O(n^{-2}). 
    \end{align*}
    Given $k \in [\ell]$ and $S \in \binom{Out_k}{2}$, define
    \begin{align*}
        B_{k, S} = \{x \in A_S \mid d_k(x) \ge C n^{-5/2}\},
    \end{align*}
    for some large enough constant $C$. We claim that if $C$ is large enough, then $B_{k, S} \cap B_{k', S} = \varnothing$. Indeed, due to Proposition~\ref{proposition: multiplicative property of degrees}, we have
    \begin{align*}
        d_{k}(x) d_{k'}(x) d_{s_1}(x) d_{s_2}(x) = O(n^{-6}),
    \end{align*}
    where $\{s_1, s_2\} = S$. By the definition of $A_S$, $d_{s_1}(x) d_{s_2}(x) = \Omega(n^{-1})$, so $d_{k}(x)$ and $d_{k'}(x)$ can not both be larger than some $C n^{-5/2}$.
    Define
    \begin{align*}
        \mathcal{G}^1_k = \bigsqcup_{S \in \binom{Out_k}{2}} \left (2^{\probabilisticCenter_k} \vee \binom{B_{k,S}}{1} \vee \bigvee_{s \in Out_k \setminus S} \binom{A_{k, s}}{\le 1} \right ).
    \end{align*}
    Then, we have 
    \begin{align*}
        |\ff_k \setminus (\hh_k \sqcup \mathcal{G}^1_k)| \le |\mathcal{E}^{sq}_k| + |\mathcal{E}^{exp}_k| + n \cdot C n^{-5/2} |\ff_k| = O(n^{-3/2}) |\ff_k|.
    \end{align*}
    Meanwhile, we get
    \begin{align*}
        \frac{|\mathcal{G}^1_k|}{|\hh_k|} = (1 + O(n^{-1})) \sum_{S \in \binom{Out_k}{2}} \frac{|B_{k, S}|}{\prod_{s \in S}|A_{k, s}|} = (1 + O(n^{-1})) \frac{\binom{\ell}{2}^2}{n^2} \sum_{S \in \binom{Out_k}{2}} |B_{k, S}|,
    \end{align*}
    where we use $|A_S| = (1 + O(n^{-1})) n / \binom{\ell}{2}$ from Theorem~\ref{theorem: maximal families structure}. Families $\mathcal{J}^1_k = \hh_k \sqcup \mathcal{G}^1_k$, $k \in [\ell]$, are overlapping, so $|\ff_k \setminus \hh_k| = |\mathcal{G}_k^1| - O(n^{-3 / 2}) |\ff_k|$, since we have $\prod_{k \in [\ell]} |\ff_k| < \prod_{k \in [\ell]} |\mathcal{J}^1_k|$ otherwise. Consequently, we obtain
    \begin{align*}
        \prod_{k \in [\ell]} |\ff_k| & = \prod_{k \in [\ell]} |\hh_k| \left (1 + \frac{|\mathcal{G}^1_k|}{|\hh_k|} + O(n^{-3/2})\right ) \nonumber \\
        & = \left (1 + \frac{\binom{\ell}{2}^2}{n^2} \cdot \sum_{k \in [\ell]} \sum_{S \in \binom{Out_k}{2}} |B_{k, S}| + O(n^{-3/2}) \right ) \prod_{k \in [\ell]} |\hh_k|.
    \end{align*}
    Define $In_S = \bigcap_{s \in S} In_s$. Then, we have
    \begin{align*}
        \sum_{k \in [\ell]} \sum_{S \in \binom{Out_k}{2}} |B_{k, S}| = \sum_{S \in \binom{[\ell]}{2}} \sum_{k \in In_{S}} |B_{k, S}| \le \sum_{S \in \binom{[\ell]}{2}} |A_S| \cdot \indicator \left \{In_S \neq \varnothing \right \}.
    \end{align*}
    Thus, we get
    \begin{align}
    \label{eq: first-order expansion of ff_k - upper bound}
        \prod_{k \in [\ell]} |\ff_k| \le \left ( 1 + \frac{\binom{\ell}{2}}{n} \sum_{S \in \binom{[\ell]}{2}} \indicator \{In_S \neq \varnothing\} + O(n^{-3/2}) \right ) \prod_{k \in [\ell]} |\hh_k|.
    \end{align}
    Hence, we can obtain the following lemma.
    \begin{lemma}
    \label{lemma: functional 1 for tournament}
        Let $\ff_k, k\in [\ell],$ be overlapping families from the extremal example. Then, the corresponding tournament maximizes the following functional:
        \begin{align}
        \label{eq: functional-1 for tournament}
            r(T_\ell) = \sum_{S \in \binom{[\ell]}{2}} \indicator \left \{In_S \neq \varnothing \right \}.
        \end{align}
        In particular, for extremal families $\ff_k$ it holds
        \begin{align}
        \label{eq: first-order expansion of ff_k}
            \prod_{k \in [\ell]} |\ff_k| = \left ( 1 + \frac{\binom{\ell}{2} \cdot  r^*}{n} + O(n^{-3/2}) \right ) \prod_{k \in [\ell]} |\hh_k|,
        \end{align}
        where $r^* = \max_{T}r(T)$ and the maximum is taken over all possible tournaments on $\ell$ vertices.
    \end{lemma}
    \begin{proof}
     Let $\ff_k, k \in [\ell],$ be extremal families and $T_\ell$ be the corresponding tournament. Fix some tournament $\hat{T}_\ell$ that maximizes~\eqref{eq: functional-1 for tournament}, and for each $S$ such that $In_S$ is not empty choose some $k(S) \in In_S$, where $In_S$ is w.r.t. $\hat T_\ell$. For each $k \in [\ell]$, define
        \begin{align*}
            \hat{\mathcal{G}}^1_k = 2^{\probabilisticCenter_k} \vee \bigvee_{S \in \binom{Out_k}{2}, k(S) = k} \left ( \binom{A_{S}}{1} \vee \bigvee_{S' \in \binom{Out_k \setminus S}{2}} \binom{A_{S'}}{\le 1}\right ),
        \end{align*}
        where $Out_k$ is defined with respect to $\hat{T}_\ell$. Also define the corresponding families $\hat{\hh}_k$ as in Theorem~\ref{theorem: maximal families structure}. According to~\eqref{eq: first-order expansion of ff_k - upper bound}, we have
        \begin{align*}
            \left ( 1 + \frac{\binom{\ell}{2} \cdot r(T_\ell)}{n} + O(n^{-3/2})\right ) \prod_{k \in [\ell]} |\hh_k| & \ge \prod_{k \in [\ell]} |\ff_k| \ge  \prod_{k \in [\ell]} |\hat{\hh}_k \sqcup \hat{G}^1_k| \\ 
            & \ge \left ( 1 + \frac{\binom{\ell}{2} \cdot r^*}{n} + O(n^{-2}) \right ) \prod_{k \in [\ell]} |\hh_k|, 
        \end{align*}
        so $r(T_\ell) = r^*$ for large enough $n$, and~\eqref{eq: first-order expansion of ff_k} holds.
    \end{proof}

    Next, we obtain the unique maximizer of~\eqref{eq: functional-1 for tournament} when $\ell = 5$. We have $\sum_{k \in [\ell]} |Out_k| = \binom{\ell}{2}$, so $\max_{k \in [\ell]} |Out_k| \ge \frac{\ell - 1}{2} = 2$. We consider several cases:
    \begin{enumerate}
        \item Suppose that $\max_k |Out_k| = 2$. Then, for all $k \in [5]$ we have $|Out_k| =2$, and, consequently, the functional~\eqref{eq: functional-1 for tournament} is at most the number of vertices of the torunament which is 5.
        \item Suppose that $\max_k |Out_k| = 4$. Then 
        \begin{align*}
            \sum_{S \in \binom{[5]}{2}} \indicator \{In_S \neq \varnothing\} = \binom{\ell - 1}{2} = 6.
        \end{align*} 
        \item Suppose that $\max_k |Out_k| = 3$. Without loss of generality, we assume that the maximum is attained when $k = 2$ and $Out_2 = \{3, 4, 5\}$.
        If $Out_1 = \{2\}$, then $r(T_\ell) \le |\binom{\{3, 4, 5\}}{2}| + |\{\{1, x\} \mid x \in \{3, 4, 5\}\}| = 6$. If $|Out_1| = 2$, without loss of generality, assume $Out_1 = \{2, 3\}$, and $(5, 4) \in E(T_\ell)$. Then, we have
        \begin{align*}
            \{S \mid In_S \neq \varnothing \} \subset \binom{\{3, 4, 5\}}{2} \cup \{\{1, 3\}, \{1, 4\}, \{2, 3\} \},
        \end{align*} 
        so $r(T_{\ell}) \le 6$. Finally, consider the case when $|Out_1| = 3$. By simmetry, we may assume that $Out_1 = \{2, 4, 5\}$. That yields the following:
        \begin{align*}
            \{S \mid In_S \neq \varnothing \} \subset  \binom{\{3, 4, 5\}}{2} \cup \{\{1, 4\}, \{1, 5\}, \{2, 4\}, \{2, 5\} \},
        \end{align*}
        so $r^* \le 7 = r(\mathbf{T}_5)$. This bound is achieved if and only if $Out_3 = \{1, 4, 5\}$. We still did not determine the direction of the edge between $4$ and $5$. It is easy to check that both choices produce the same graph up to an isomorphism.
    \end{enumerate}
    Consequently, the tournament $T_\ell$ is isomorphic to $\mathbf{T}_5$ when $\ell = 5$. In $\mathbf{T}_5$ for all $S \in \binom{[5]}{2} \setminus \{\{4, 5\}\}$ we have $|In_S| \in \{0, 1\}$. For $k \in \{1, 2, 3\}$, let $o_k$ be the unique element of $Out_k \cap \{1, 2, 3\}$ and define
    \begin{align*}
        \mathcal{Q}_k = \left [2^{\probabilisticCenter_k} \vee \binom{A_{o_k, 4}}{1} \vee \binom{A_{k, 5}}{\le 1}\right ] \cup \left [ 2^{\probabilisticCenter_k} \vee \binom{A_{o_k, 5}}{1} \vee \binom{A_{k,4}}{\le 1}\right ].
    \end{align*}
    Note that $\probabilisticCenter_k = A_{i_k, k}$, where $i_k$ is the unique element of $In_k$. We claim that $\hh_k \sqcup \mathcal{Q}_k \subset \ff_k$ for $k \in \{1, 2, 3\}$.
    \begin{claim}
    \label{claim: 2-order tentacles for l = 5}
        Let $\ff_k$ be extremal families, $k \in [\ell], \ell = 5$ with the corresponding tournament $\mathbf{T}_5$. Then, $\hh_k \sqcup \mathcal{Q}_k \subset \ff_k$ for any $k \in [3]$.
    \end{claim}
    \begin{proof}
        Define $\mathcal{E}_k^Q = \bigcup_{k' \in [\ell] \setminus \{k\}} (\mathcal{Q}_{k'} \cap \ff_{k})^{(2)}$. Then, families
        \begin{align*}
            & \ff_k' = \ff_k \setminus \bigcup_{e \in \mathcal{E}^Q_k} \ff_k(\saveSet{e}) \cup \mathcal{Q}_k, \quad k \in [3], \\
            & \ff_{k}' = \ff_k \setminus \mathcal{E}^{Q}_k, \quad k \in \{4, 5\},
        \end{align*}
        are overlapping. Meanwhile, if $e \in \mathcal{E}_k^Q$ then $|e \cap \support \hh_{k_1}| = |e \cap \support \hh_{k_2}| = 1$ for some two distinct indices $k_1, k_2 \in \{1, 2, 3\} \setminus \{k\}$. By the construction of $\mathbf{T}_5$, either $(k_1, k) \in E(\mathbf{T}_5)$ or $(k_2, k) \in E(\mathbf{T}_5)$, so $d_k(e) = e^{-\Omega(n)}$ due to Lemma~\ref{lemma: direction of tentacles} and Proposition~\ref{proposition: child  can not overlap}. If $\mathcal{E}^Q_k$ is not empty for some $k$, we have
        \begin{align*}
            \prod_{k \in [5]} |\ff_{k}'| \ge (1 - n^2 e^{- \Omega(n)})^{\ell} \prod_{k \in \{1, 2, 3\}} \left ( 1 + \frac{1}{|\ff_k|} \left |\bigcup_{k' \in [5] \setminus \{k\}} \bigcup_{e \in \mathcal{E}^Q_{k'}} \mathcal{Q}_k'(\saveSet{e}) \right |\right ) \prod_{k \in [\ell]}|\ff_k|
        \end{align*}
        For any $e \in \mathcal{Q}_k^{(2)}$ which is not fully contained in $\support \hh_k$, we have $|\mathcal{Q}_k(e)| \ge 2^{|\probabilisticCenter_k| - 1}$, so 
        \begin{align*}
            \prod_{k \in \{1, 2, 3\}} \left ( 1 + \frac{1}{|\ff_k|} \left |\bigcup_{k' \in [5] \setminus \{k\}} \bigcup_{e \in \mathcal{E}^Q_{k'}} \mathcal{Q}_k'(\saveSet{e}) \right |\right ) = 1 + \Omega(n^{-3}),
        \end{align*}
        the contradiction with maximality of $\ff_k$.
    \end{proof}

    \begin{figure}
     \centering
     \begin{subfigure}[b][][t]{0.31\textwidth}
         \centering
         \includegraphics[width=\textwidth]{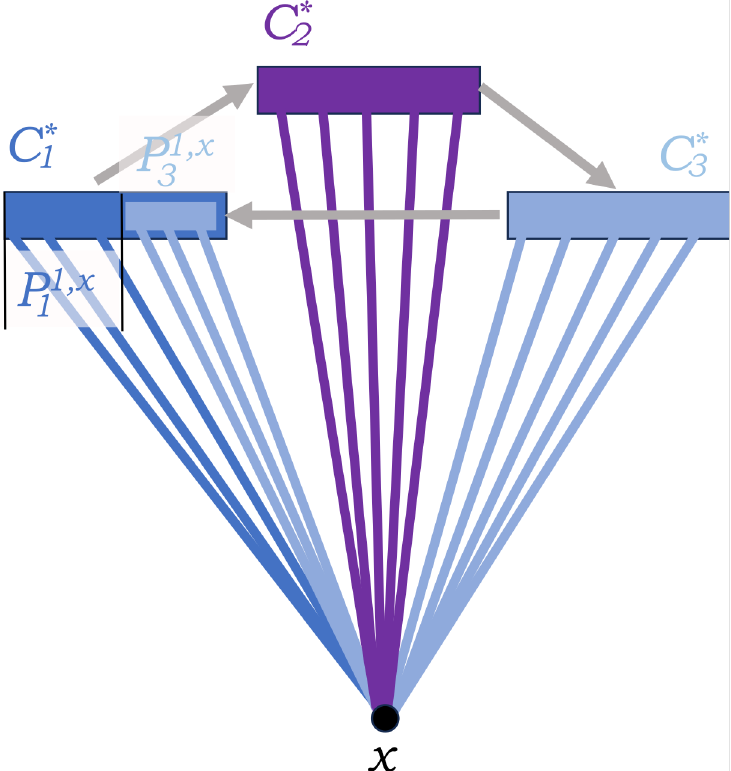}
         \caption{ }
         \label{fig: definition of P}
     \end{subfigure}
     \hfill
     \begin{subfigure}[b][][t]{0.31\textwidth}
         \centering
         \includegraphics[width=\textwidth]{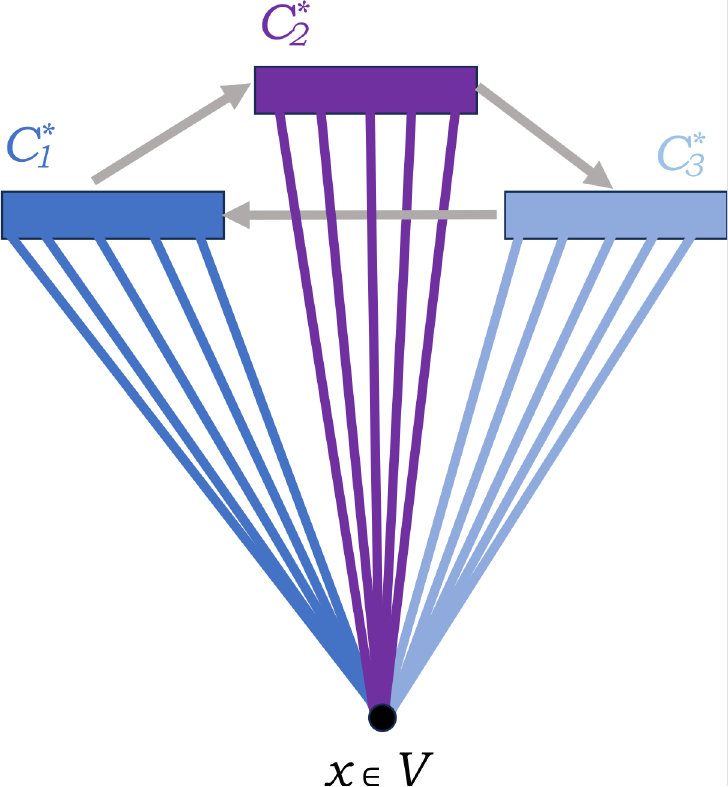}
         \caption{}
         \label{fig: x in V}
     \end{subfigure}
     \hfill
     \begin{subfigure}[b][][t]{0.31\textwidth}
         \centering
         \includegraphics[width=\textwidth]{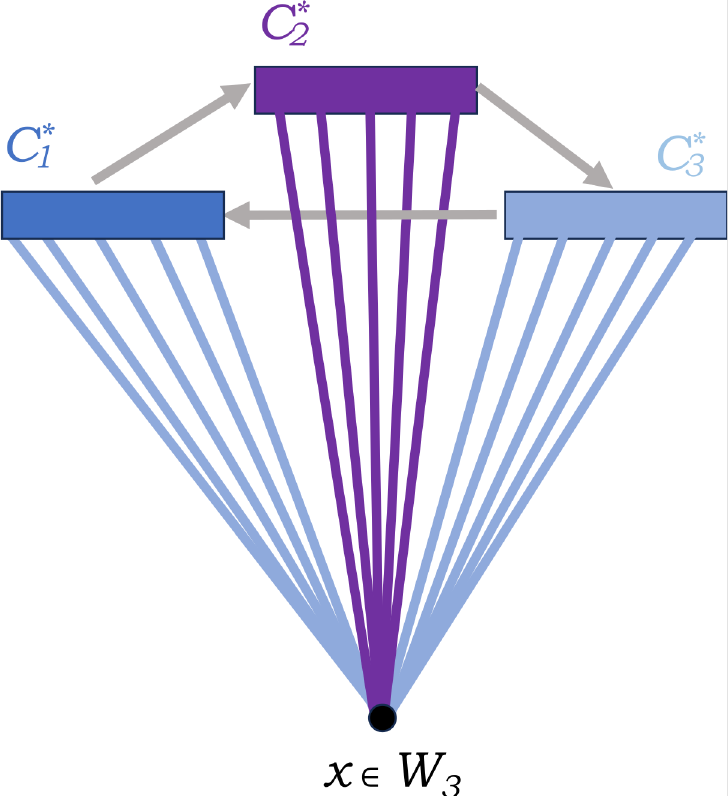}
         \caption{ }
         \label{fig: x in W}
     \end{subfigure}
        \caption{Different configurations of edges between $A_{4, 5}$ and $\probabilisticCenter_{k}$, $k \in [3]$.}
\end{figure}

    As our problem is equivalent to coloring of $K_n$, we are rest to determine the color of edges between $A_{4, 5}$ and $A_{k, k'}$, $k, k' \in \{1, 2, 3\}$, since colors of other edges are determined either by Theorem~\ref{theorem: maximal families structure} or Claim~\ref{claim: 2-order tentacles for l = 5}.

    For each $x \in A_{4, 5}$, define $P_k^{s, x} \subset \probabilisticCenter_k$ as
    \begin{align*}
        P_k^{s, x} = \{y \in \probabilisticCenter_s \mid \{x, y \} \in \ff_k \}.
    \end{align*}
    Clearly, $P_k^{s,x}, k \in [5] \setminus \{k\}$, form a partition of $\probabilisticCenter_s$. For the illustration, see Figure~\ref{fig: definition of P}. Then, we have
    \begin{align*}
        \ff_k = 
            \hh_k \sqcup \mathcal{Q}_k  \cup \binom{[n]}{1} \sqcup \bigsqcup_{
                x \in A_{4, 5}
            } \{\{x\}\} \vee \left \{ \left [2^{P_k^{k, x}} \vee  \binom{P^{o_k, x}_{k}}{\le 1} \right ] \cup \binom{P^{ i_k, x}_{ k}}{1} \right \},
    \end{align*}
    for $k \in [3]$, and 
    \begin{align*}
        \ff_k = \hh_k \cup \binom{[n]}{1} \sqcup \bigsqcup_{x \in A_{4, 5}} \{\{x\}\} \vee \bigsqcup_{k' \in [3]} \binom{P_{k'}^{k, x}}{1} 
    \end{align*}
    for $k \in \{4, 5\}$. Given a family of sets $\ff$, here $\{\{x\}\} \vee \ff$ means attaching $x$ to all sets of $\ff$. Note that if $y \in P_k^{k', x}$ when $k' \in In_k$, we have $d_k(x, y) = e^{-\Omega(n)}$ due to Lemma~\ref{lemma: direction of tentacles}. Consequently, we obtain
    \begin{align}
        \prod_{k \in [5]} |\ff_k| & = \prod_{k \in [3]} |\hh_k \sqcup \mathcal{Q}_k| \left (1 + \frac{1}{|\hh_k \sqcup \mathcal{Q}_k|} \sum_{x \in A_{4, 5}} (1 + |P_{k}^{o_k, x}|) 2^{|P_k^{k, x}|} + e^{-\Omega(n)}\right ) \nonumber \\
        & \quad \times \prod_{k \in \{4, 5\}} |\hh_k| (1 + e^{- \Omega(n)}), \label{eq: 2-order remainder with exp}
    \end{align}
    
    Next, we study when the product
    \begin{align}
    \label{eq: 2-order remainder}
        \prod_{k \in [3]} \left (1 + \frac{1}{|\hh_k \sqcup \mathcal{Q}_k|} \cdot \sum_{x \in A_{4, 5}} (1 + |P_{k}^{o_k, x}|) 2^{|P_k^{k, x}| }\right )
    \end{align}
    achieves its maximum $p^*$. It is a convex function of variables $$\{|P_{k}^{o_k,x}|\}_{x \in A_{4, 5}, k \in [3]} \cup \{|P_k^{k, x}|\}_{k \in [3], x \in A_{4, 5}},$$ and we have $0 \le |P_{i_k}^{k, x}| + |P_{k}^{k, x}| \le |\probabilisticCenter_k|$. So for each $x$ and $k \in [3]$, we have either $P_k^{k,x} = \varnothing, P_{i_k}^{k, x} = \probabilisticCenter_k$ or $P_{i_k}^{k,x} = \varnothing, P_k^{k, x} = \probabilisticCenter_k$. Indeed, if we have $P_{i_k}^{k, x} \not \in \{\varnothing, \probabilisticCenter_k \}$ for some $k$ and $x$, then there is a polynomial gap of magnitude $\Omega(n^{-2})$ between $p^*$ and~\eqref{eq: 2-order remainder}. One can easily construct overlapping families $\ff_k', k \in [5],$ such that $\prod_{k} |\ff_k'| = p^* \prod_{k \in [3]} |\hh_k \sqcup \mathcal{Q}_k| \cdot \prod_{k \in \{4, 5\}} |\hh_k|$, so if $P_{i_k}^{k, x} \not \in \{\varnothing, \probabilisticCenter_k \}$, then \eqref{eq: 2-order remainder with exp} implies $\prod_k |\ff_k'| > \prod_k |\ff_k|$, the contradiction.

    For each $x$, we have two possibilities:
    \begin{enumerate}
        \item for each $k \in [3]$, $P^{k, x}_k = \probabilisticCenter_k$;
        \item there exists some $k \in [3]$ such that $P_k^{k, x} = \probabilisticCenter_k$ and $P_{k}^{o_k, x} = \probabilisticCenter_{o_k}$. Note that then we have $P_{i_k}^{i_k, x} = \probabilisticCenter_{i_k}$.
    \end{enumerate}
    Otherwise, we have $d_k(x) = e^{-\Omega(n)}$ for some $k \in [3]$, so we are able to enlarge the product $\prod_{k \in [\ell]} |\ff_k|$ by recoloring edges from $x$ to $\probabilisticCenter_k$.
    
    The set of $x$ such that the former case holds for them, we denote by $V$. The set of $x$ such that the latter case holds for them with some $k$, we denote by $W_k$. For the illustration, see Figures~\ref{fig: x in V} and~\ref{fig: x in W}.

    Clearly, sets $V, W_1, W_2, W_3$ are disjoint and form a partition of $A_{4, 5}$. Next, for each $k \in [3]$, we have
    \begin{align*}
        |\ff_k| \le |\hh_k| + |\mathcal{Q}_k| + 2^{|\probabilisticCenter_k|} |V| + 2^{|\probabilisticCenter_k|} \cdot \left | \binom{A_{k, o_k}}{\le 1}\right | \cdot  |W_k| + 2^{|\probabilisticCenter_k|} \cdot \left | \binom{W_{o_k}}{1}\right | + e^{-\Omega(n)} \cdot 2^{|\probabilisticCenter_k|}.
    \end{align*}
    Overlapping families
    \begin{align*}
        \mathcal{W}_k =
        \begin{cases}
            \hh_k \sqcup \mathcal{Q}_k \sqcup \left (2^{\probabilisticCenter_k} \vee \binom{V \sqcup W_{o_k}}{1} \right ) \sqcup \left (2^{\probabilisticCenter_k} \vee \binom{W_k}{1} \vee \binom{A_{k, o_k}}{\le 1}\right ) , & k \in [3] \\
            \hh_k, & k \in \{4, 5\},
        \end{cases}
    \end{align*}
    attains this upper bound up to a factor $1 + e^{-\Omega(n)}$. If $V$ is not empty, one can remove some $x \in V$ from $V$ and add $x$ to some $W_k$. That enlarges $\mathcal{W}_k$ by a factor $1 + \Omega(n^{-2})$, but other families $\mathcal{W}_{k'}$, $k' \in [5] \setminus \{k\}$ have their size decreased by a factor $1 + O(n^{-3})$ at most. Thus, $V = \varnothing$.

    Hence, $W_1, W_2, W_3$ form a partition of $A_{4, 5}$. We have all edges of $K_n$ colored, so the only way to enlarge families $\ff_k, k \in [5],$ is to add singletons $\binom{[n]}{1}$ to each of them. Thus, Theorem~\ref{theorem: extremal graph coloring} holds.

\section{Conclusion}

Firstly, we obtained the asympotics of $\targetFunction$. We applied a number of techniques, including  different correlation inequalities, Hoeffding inequality and some hypergraph theory. We have developed some tricks and methods that allowed us to obtain a number of structural results. Thanks to that, we managed to show the asymptotics of $\targetFunction$ and describe extremal examples.



\printbibliography

\newpage

\appendix

\section{Proof of ancillary propositions}

In this section, we give the omitted proofs of some ancillary propositions.

\subsection{Proof of Proposition~\ref{proposition: inner degree with an element}}
\label{subsection: proof of proposition inner degree with an element}

\begin{proof}
For a reminder, 
\begin{align*}
    d_{k_1 \to k_2}^{(\ge \mSet_{k_1, k_2})} = \frac{1}{|\ff_{k_1}|} \left | \{F \in \ff_{k_1} \mid |F \cap \probabilisticCenter_{k_2}| \ge \mSet_{k, k_2} \} \right |.
\end{align*}
Applying Lemma~\ref{lemma: direction of tentacles} with $\ff_{k}' = \ff_{k}|_{Supp_k}$ for families $\ff_{k_1}$ and $\ff_{k_2}$, we obtain
\begin{align*}
    d_{k_1 \to k_2}^{(\ge \mSet_{k_1, k_2})} = 1 - O(n^{-1})
\end{align*}
since $|\ff_{k}|_{Supp_k}| = (1 - O(n^{-1})) |\ff_{k}|$ due to Proposition~\ref{proposition: subfamilies 1/n}. Let $T$ be arbitrary subset of $\probabilisticCenter_{k_1}$ of size at most $m = \max_{S \in \binom{[\ell]}{2}} \mSet_{S}$. Then
\begin{align*}
    \frac{1}{|\ff_{k_1}|}
    \sum_{P \in \binom{A_{k_1, k_2}}{\le \mSet_{k_1, k_2 - 1}}} 
        |\ff_{k_1}(T \cup P, T \cup A_{k_1, k_2})|
    & \le 
    \frac{1}{|\ff_{k_1}|}
    \sum_{P \in \binom{A_{k_1, k_2}}{\le \mSet_{k_1, k_2 } - 1}} 
        |\ff_{k_1}(P, A_{k_1, k_2})|.
\end{align*}
The right-hand side is the portion of sets that intersect $\probabilisticCenter_{k_1}$ in less than $\mSet_{k, k_1}$ elements. Thus, it can be bounded by $1 - d_{k_1 \to k_2}^{(\ge \mSet_{k_1, k_2})} = O(n^{-1})$. Hence,
\begin{align}
d_{k_1}(T) & = \frac{1}{|\ff_{k_1}|} 
    \sum_{P \in \binom{A_{k_1, k_2}}{\le \mSet_{k_1, k_2}}} |\ff_{k_1}(T \cup P, T \cup A_{k_1, k_2})| \nonumber \\
    & = \frac{1}{|\ff_{k_1}|} \sum_{P \in \binom{A_{k_1, k_2}}{\mSet_{k_1, k_2}}}  |\ff_{k_1}(T \cup P, T \cup A_{k_1, k_2})| + O(n^{-1}). \label{eq: d_k(T) upper bound, outside degree lemma}
\end{align}
A simple double-counting argument ensures us that 
\begin{align*}
     \frac{1}{|\ff_{k_1}|} \sum_{P \in \binom{A_{k_1, k_2}}{\mSet_{k_1, k_2}}}  |\ff_{k_1}(T \cup P, T \cup A_{k_1, k_2})| & \le \sum_{P \in \binom{A_{k_1, k_2}}{\mSet_{k_1, k_2}}} d_{k_1}(T \cup P) \\
     & \le
    \frac{1}{\mSet_{k_1, k_2}} \sum_{x \in A_{k_1, k_2}} d_{k_1}(T \cup \{x\}).
\end{align*}
Hence, according to inequality~\eqref{eq: d_k(T) upper bound, outside degree lemma},
\begin{align}
\label{eq: dk(T) upper bound final, outside degree lemma}
    d_{k_1}(T) \le \frac{1}{\mSet_{k_1, k_2}} \sum_{x \in A_{k_1, k_2}} d_{k_1}(T \cup \{x\}) + O(n^{-1}). 
\end{align}
Due to Claim~\ref{claim: inner set degree}, the following lower bound holds:
\begin{align}
\label{eq: dk(T) lower bound, outside degree lemma}
    d_{k_1}(T) \ge 2^{-|T|} - O(n^{-1}).
\end{align}

Define $I_T^{k_1, k_2} = \{x \in A_{k_1,k_2} \mid d_{k_1}(T \cup \{x\}) \le 2^{-|T| - 1} d_{k_1}(x) \}$. Moreover, let $\Delta = n^{-3/2} \cdot \sqrt{\log n}$ and $I_\Delta = \bigcup \mathcal{I}_{\Delta}$, where $\mathcal{I}_{\Delta}$ arises from Lemma~\ref{lemma: symmetrization argument}\ref{symmetrization lemma: case 1} applied for $\sDomain = \binom{A_{k_1, k_2}}{1}$ and $S_0 = \arg \min_{x \in \sDomain} d_{k_1}(x)$. Conditions of Lemma~\ref{lemma: symmetrization argument} is satisfied since for any $k' \not \in \{k_1, k_2\}$ we have
\begin{align*}
    d_{k'}(x) d_{k_0}(x) d_{k_1}(x) = O(n^{-3})
\end{align*}
due to Proposition~\ref{proposition: multiplicative property of degrees}, and so $d_{k'}(x) = O(n^{-2}) \ll \frac{\Delta / 2}{\ell - 2}$.
We have
\begin{align*}
    \sum_{x \in A_{k_1, k_2}} d_{k_1}(T \cup \{x\}) 
    & \le 
    2^{-|T| - 1} \sum_{x \in I_T^{k_1, k_2} \setminus I_\Delta} d_{k_1}(x) +
    2^{-|T|} \sum_{x \in A_{k_1, k_2} \setminus (I_{T}^{k_1, k_2} \cup I_\Delta)} d_{k_1}(x) +
    2^{-|T|} \sum_{x \in I_\Delta} d_{k_1}(x).
\end{align*}
Let us bound the sums separately:
\begin{align*}
    \sum_{x \in I_{T}^{k_1, k_2} \setminus I_\Delta} d_{k_1}(x)
    & \le
    |I_T^{k_1, k_2}| \left (
        \min_{x \in A_{k_1, k_2}} d_{k_1}(x) + \Delta
    \right )
    \le 
    |I_T^{k_1, k_2}| \left (
        \frac{\mSet_{k_1, k_2}}{|A_{k_1, k_2}|} +\Delta + O(n^{-2})
    \right ), \\
    \sum_{x \in A_{k_1, k_2} \setminus (I_\Delta \cup I_{T}^{k_1, k_2})} d_{k_1}(x)
    & \le (|A_{k_1,k_2}| - |I_T^{k_1, k_2}|) 
    \left ( 
        \frac{\mSet_{k_1, k_2}}{|A_{k_1, k_2}|} + \Delta + O(n^{-2})
    \right ), 
\end{align*}
where we use Claim~\ref{claim: 1 element minimum}. We also have $|I_{\Delta}| = \frac{O(\log n)}{n \cdot \Delta}$ for our particular $\Delta = \sqrt{\frac{\log n}{n^3}}$, and so
\begin{align*}
    \sum_{x \in I_\Delta} d_{k_1}(x) & = |I_\Delta| \cdot O(n^{-1}) = O \left ( \sqrt{\frac{\log n}{n}}\right ).
\end{align*}
Combining the lower bound~\eqref{eq: dk(T) upper bound final, outside degree lemma} and upper bound~\eqref{eq: dk(T) lower bound, outside degree lemma} on $d_{k_1}(T)$, we obtain
\begin{align*}
    2^{-|T|} - O(n^{-1}) & \le 2^{-|T| - 1} \frac{|I^{k_1, k_2}_T|}{|A_{k_1, k_2}|} + 2^{-|T|} \cdot \frac{|A_{k_1,k_2}| - |I_T^{k_1, k_2}| }{|A_{k_1, k_2}|} + O \left (\sqrt{\frac{\log n}{n}} \right ), \\
    \frac{1}{2} \frac{|I_T^{k_1, k_2}|}{|A_{k_1, k_2}|} & = O \left (\sqrt{\frac{\log n}{n}} \right ), \\
    |I_T^{k_1, k_2}| & = O(\sqrt{n \log n})
\end{align*}
uniformly over $T \subset \binom{\probabilisticCenter_{k_1}}{\le m}$.  
\end{proof}

\subsection{Proof of Proposition~\ref{proposition: child  can not overlap}}
\label{subsection: proposition child  can not overlap}

\begin{proof}
The proof is similar to that of Lemma~\ref{lemma: direction of tentacles}. Consider a random subset $\randomSubset$ uniformly distributed over $2^{\probabilisticCenter_{k_1}}$. Define $\varepsilon_x = \indicator[x \in \randomSubset]$. Obviously, if $\randomSubset \cup \{x_2\} \in \ff_{k_1}$, then a random variable
\begin{align*}
    \mathbf{f}_{k_0} = \sum_{S \in \binom{\probabilisticCenter_{k_1}}{\mSet_{{k_0}, k_1}}} d_{{k_0}}(S \cup \{x_2\}) \prod_{x \in S} \varepsilon_x
\end{align*}
equals zero. Hence,
\begin{align}
\label{eq: probability child can not overlap}
    \PP(\randomSubset \cup \{x_2\} \in \ff_{k_1}) \le \PP \left ( \mathbf{f}_{k_0} = 0 \right ).
\end{align}

Our first aim is to compute the mathematical expectation of $\mathbf{f}_{k_0}$. We use the following decomposition:
\vspace{0.8cm}
\begin{align}
\label{eq: d_k(x_2) bound, overlapping tentacles lemma}
    d_{k_0}(x_2) = \sum_{S \in \binom{A_{k_0, k_1}}{\mSet_{k_0, k_1}}} d_{k_0}(S \cup \{x_2\}) + \frac{1}{|\ff_{k_0}|}
    \sum_{S \in \binom{A_{{k_0}, k_1}}{\le \mSet_{{k_0}, k_1} - 1}} \left | \ff_{k_0}(S \cup \{x_2\}, \{x_2\} \cup A_{{k_0}, k_1}) \right |.
\end{align}

Corollary~\ref{corollary: outside degree} guarantees that 
\begin{align}
\label{eq: redundant sets, overlapping tentacles lemma}
    \frac{1}{|\ff_{k_0}|}
    \left | \ff_{k_0}(S \cup \{x_2\}, \{x_2\} \cup A_{{k_0}, k_1}) \right | 
    & \le  \frac{1}{|\ff_{k_0}|}  
    \left | \ff_{k_0}|_{Supp_{k_0}}(S \cup \{x_2\}, \{x_2\} \cup A_{{k_0}, k_1}) \right | \nonumber  \\
    & \qquad + \sum_{x \not \in Supp_{k_0}} d_{k_0}(S \cup \{x, x_2\})  \nonumber \\
    & = \frac{1}{|\ff_{k_0}|}
    \left | \ff_{k_0}|_{Supp_{k_0}}(S \cup \{x_2\}, \{x_2\} \cup A_{{k_0}, k_1}) \right | \nonumber \\
    & \qquad + O \left ( n^{-|S| - 3/2} \log^{1/2} n\right ).
\end{align}
Meanwhile, 
\begin{align}
\label{eq: hh_k applied bound, overlapping tentacles lemma}
    \sum_{S \in \binom{A_{{k_0}, k_1}}{\le \mSet_{{k_0}, k_1} - 1}} \left | \ff_{k_0}|_{Supp_{k_0}}(S \cup \{x_2\}, \{x_2\} \cup A_{{k_0}, k_1}) \right | 
    \le
    \sum_{S \in \binom{A_{{k_0}, k_1}}{\le \mSet_{{k_0}, k_1} - 1}} \left | \hh_{k_0}(S \cup \{x_2\}, \{x_2\} \cup A_{{k_0}, k_1}) \right |
\end{align}
since $\ff_{k_0} |_{Supp_{k_0}} \subset \hh_{k_0}$. The order of right-hand side can be easily computed by the definition~\eqref{eq: definition of hh_k} of $\hh_{k_0}$. It turns out to be $O(n^{-2}) \cdot |\hh_{k_0}|$. At the same time, we have $|\ff_{k_0} |_{Supp_{k_0}}| = (1 - O(n^{-1})) |\hh_{k_0}|$ from Claim~\ref{claim: hh_k approximation quality} and Proposition~\ref{proposition: subfamilies 1/n}. Thus, the right-hand side of~\eqref{eq: hh_k applied bound, overlapping tentacles lemma} is at most $O(n^{-2}) \cdot |\hh_{k_0}| = O(n^{-2}) \cdot |\ff_{k_0}|$. Combining inequalities~\eqref{eq: d_k(x_2) bound, overlapping tentacles lemma},~\eqref{eq: redundant sets, overlapping tentacles lemma}, we derive
\begin{align*}
    d_{k_2}(x_2) \le \sum_{S \in \binom{A_{k_0, k_1}}{\mSet_{k_0, k_1}}} d_{k_0}(S \cup \{x_2\}) + O \left (\sqrt{\frac{\log n}{n^3}} \right ).
\end{align*}
By our assumptions, $d_{k_0}(x_2) = \frac{\mSet_{k_0, k_2}}{|A_{k_0, k_2}|} + O \left (\sqrt{\frac{\log n}{n^3}} \right )$, and so 
\begin{align*}
    \sum_{S \in \binom{A_{k_0, k_1}}{\mSet_{k_0, k_1}}} d_{k_0}(S \cup \{x_2\}) \ge \frac{\mSet_{k_0, k_1}}{|A_{k_0, k_1}|} - O \left ( \sqrt{\frac{\log n}{n^3}}\right ).
\end{align*} 
Consequently,
\begin{align}
\label{eq: expectation overlapping tentacles lemma}
    \EE [\mathbf{f}_{k_0}] = \EE \left [\sum_{S \in \binom{\probabilisticCenter_{k_1}}{\mSet_{{k_0}, k_1}}} d_{{k_0}}(S \cup \{x_2\}) \prod_{x \in S} \varepsilon_x \right ] \ge 2^{- \mSet_{{k_0}, k_2}}\frac{\mSet_{{k_0}, k_2}}{|A_{{k_0}, k_2}|} - O \left (\sqrt{\frac{\log n}{n^3}} \right ) = \Omega(n^{-1}).
\end{align}
The remaining part is completely analogous to analyzing the exposure martingale in Lemma~\ref{lemma: direction of tentacles}. Define, as previously,
\begin{align*}
    Y_{w} = \EE [\mathbf{f}_{k_0} \, | \, \sigma(\varepsilon_{z_1}, \ldots, \varepsilon_{z_w})], \quad Y_0 = \EE [\mathbf{f}_{k_0}],
\end{align*}
where $z_1, \ldots, z_{|\probabilisticCenter_{k_1}|}$ are the elements of $\probabilisticCenter_{k_1}$ in some order. Define $U_w = \{z_1, \ldots, z_w\}$. The martingale difference is
\begin{align*}
    |Y_w - Y_{w - 1}| \le \left |\frac12 - \varepsilon_{z_{w}} \right | \sum_{S \in \binom{\probabilisticCenter_{k_1} \setminus \{z_w\}}{\mSet_{k_0, k_1} - 1}} d_{k_0}(S \cup \{z_w, x_2\}) 2^{-|S \setminus U_{w - 1}|} \prod_{z \in S \cap U_{w - 1}} \varepsilon_z.
\end{align*}
Due to Proposition~\ref{proposition: multiplicative property of degrees}, we have
\begin{align*}
    d_{k_0}(S \cup \{z_w, x_2\}) d_{k_1}(S \cup \{z_w\})d_{k_2}(x_2) = O(n^{-|S|-2}),
\end{align*}
and, since $S \cup \{z_w\} \subset \probabilisticCenter_{k_1}$ and $x_2 \in \probabilisticCenter_{k_2}$, we have $d_{k_0}(S \cup \{z_w, x_2\}) = O(n^{-|S| - 2})$. Thus,    $|Y_w - Y_{w-1}| = O(n^{-2})$. We have $Y_0 = \EE [\mathbf{f}_{k_0}]$ and $Y_{|\probabilisticCenter_{k_1}|} = \mathbf{f}_{k_0}$. Applying Theorem~\ref{theorem: Hoeffding's iniequality}, we obtain
\begin{align*}
    \PP \left ( \EE [\mathbf{f}_{k_0}] - \mathbf{f}_{k_0} \ge \EE [\mathbf{f}_{k_0}]\right ) 
    \le
    \exp \left (
         - \Omega(n^3) \cdot \EE^2 [\mathbf{f}_{k_0}]
    \right ).
\end{align*}
Due to~\eqref{eq: expectation overlapping tentacles lemma}, we obtain
\begin{align*}
    \PP \left ( \EE [\mathbf{f}_{k_0}] - \mathbf{f}_{k_0} \ge \EE [\mathbf{f}_{k_0}]\right ) = e^{-\Omega(n)}.
\end{align*}
Obviously, the above probability bounds~\eqref{eq: probability child can not overlap}, and, hence, $\PP (\mathbf{X} \cup \{x_2\} \in \ff_{k_1}) = e^{-\Omega(n)}$. Finally,
\begin{align*}
    |\ff_{k_1}(x_2)| \le \prod_{k \in [\ell] \setminus \{k_1\}} \binom{n}{\le \mSet_{k, k_1}} \PP (\mathbf{X} \cup \{x_2\} \in \ff_{k_1}) 2^{|\probabilisticCenter_{k_1}|}.
\end{align*}
Using $|\ff_{k_1}| \ge 2^{|\probabilisticCenter_{k_1}|}$, we infer
\begin{align*}
    d_{k_1}(x_2) \le \prod_{k \in [\ell] \setminus \{k_1\}} \binom{n}{\le \mSet_{k, k_1}} e^{-\Omega(n)} = e^{- \Omega(n)}. 
\end{align*}
\end{proof}

\end{document}